\documentclass{amsart}
\usepackage{latexsym,amsthm,mathrsfs,amsmath,amscd,enumerate,enumitem, amscd,color,dsfont, cancel}

\usepackage[cspex,bbgreekl]{mathbbol}
\usepackage{amsfonts}
\usepackage{amssymb}             

\DeclareSymbolFontAlphabet{\mathbbl}{bbold}
\DeclareSymbolFontAlphabet{\mathbb}{AMSb}%

 \newtheorem{thm}{Theorem}[section]
 \newtheorem{cor}[thm]{Corollary}
 \newtheorem{lem}[thm]{Lemma}
 \newtheorem{prop}[thm]{Proposition}
\theoremstyle{definition}

 \theoremstyle{remark}
 \newtheorem{rem}[thm]{Remark}

\usepackage{hyperref}

\newcommand{\supp}{\mathop{\mathrm{supp}}}

\newcommand{\esssup}{\mathop{\mathrm{ess\, sup \;}}}
\newcommand{\essinf}{\mathop{\mathrm{ess\, inf \;}}}
\newcommand{\LpQ}{L^{p(\cdot)}_{\mathcal Q}(\mathbb R^n)}
\newcommand{\HIp}{H_I^{p(\cdot)}(\mathbb R^n)}
\newcommand{\HLp}{H_L^{p(\cdot)}(\mathbb R^n)}

\newcommand{\hLp}{h_L^{p(\cdot)}(\mathbb R^n)}
\newcommand{\HLIp}{H_{L+I}^{p(\cdot)}(\mathbb R^n)}

\numberwithin{equation}{section}
\allowdisplaybreaks

\begin{document}

\title[]
 {Local Hardy spaces with variable exponents associated to non-negative self-adjoint operators satisfying Gaussian estimates}

\author[V. Almeida]{V\'ictor Almeida}
\address{V\'ictor Almeida, Jorge J. Betancor, Lourdes Rodr\'iguez-Mesa\newline
	Departamento de An\'alisis Matem\'atico, Universidad de La Laguna,\newline
	Campus de Anchieta, Avda. Astrof\'isico S\'anchez, s/n,\newline
	38721 La Laguna (Sta. Cruz de Tenerife), Spain}
\email{valmeida@ull.es, jbetanco@ull.es, lrguez@ull.es}

\author[J. J. Betancor]{Jorge J. Betancor}

\author[E. Dalmasso]{Estefan\'ia Dalmasso}
\address{Estefan\'ia Dalmasso\newline
	Instituto de Matem\'atica Aplicada del Litoral, UNL, CONICET, FIQ.\newline Colectora Ruta Nac. N° 168, Paraje El Pozo,\newline S3007ABA, Santa Fe, Argentina}
\email{edalmasso@santafe-conicet.gov.ar}

\author[L. Rodr\'iguez-Mesa]{Lourdes Rodr\'iguez-Mesa}

\thanks{Last modification: \today.}



\date{}


\begin{abstract}
In this paper we introduce variable exponent local Hardy spaces $\hLp$ associated with a non-negative self-adjoint operator $L$. We assume that, for every $t>0$, the operator $e^{-tL}$ has an integral representation whose kernel satisfies a Gaussian upper bound. We define $\hLp$ by using an area square integral involving the semigroup $\{e^{-tL}\}_{t>0}$. A molecular characterization of $\hLp$ is established. As an application of the molecular characterization we prove that $\hLp$ coincides with the (global) Hardy space $\HLp$ provided that $0$ does not belong to the spectrum of $L$. Also, we show that $\hLp=H_{L+I}^{p(\cdot)}(\mathbb R^n)$.
\end{abstract}

\maketitle

\section{Introduction}

Hardy spaces were originated in the first half of the 20th century in the setting of Fourier series and complex analysis in one variable. The foundations of the real Hardy space $H^p(\mathbb R^n)$, $0<p\leq 1$, were laid in the celebrated paper of Fefferman and Stein \cite{FS}. After that article and until today, Hardy spaces, and their applications and generalizations, have been an active and flourishing area of work. As it is well-known, an important fact is that Hardy spaces $H^p(\mathbb R^n)$ serve as a substitute of $L^p(\mathbb R^n)$ when $0<p\leq 1$ in many applications.

Hardy spaces, initially defined on $\mathbb R^n$, have been extended to other settings having different underlying spaces (see \cite{AM}, \cite{AMR}, \cite{CW}, \cite{HyYaYa}, \cite{Stri} and \cite{Tol}, among
others). Hardy spaces $H^p(\mathbb R^n)$ are closely connected with the Laplacian operator $\Delta=\sum_{i=1}^{n}\partial^2/\partial x_i^2$ in $\mathbb R^n$. These spaces can be characterized in several ways (maximal functions, Littlewood-Paley square functions, Riesz transforms,...) where $\Delta$ appears. In order to analyze problems where other operators $L$, different from $\Delta$, are involved, new Hardy spaces have been defined in the last decade and they are adapted, in some sense, to $L$.

It is usually considered a non-negative and self-adjoint operator $L:D(L)\subset L^2(\mathbb R^n)\rightarrow L^2(\mathbb R^n)$. Then, the operator $L$ generates a bounded analytic semigroup $\{e^{-tL}\}_{t>0}$ (\cite{Ou1}). Also, it is assumed that the so-called heat semigroup $\{e^{-tL}\}_{t>0}$ associated with $L$ satisfies some kind of off-diagonal estimates. Auscher, Duong and McIntosh (\cite{ADM}) and Duong and Yang (\cite{DY2} and \cite{DY1}) defined Hardy spaces related to operators $L$ such that, for every $t>0$, the operator $e^{-tL}$ is an integral operator whose kernel satisfies a pointwise Gaussian upper bound. There exist operators, for instance, second order divergence form elliptic operators on $\mathbb R^n$ with complex coefficients, where pointwise heat kernels bounds may fail. In \cite{AMR} and \cite{HM} Hardy spaces associated with operators using only Davies-Gaffney type estimates in place of pointwise kernel bounds were studied.

Weighted Hardy spaces adapted to operators such that $\{e^{-tL}\}_{t>0}$ fulfills reinforced off-diagonal estimates were introduced in \cite{BCKYY}. This off-diagonal estimate property had been considered by Auscher and Martell (\cite{AMar}).

The principle that Hardy spaces $H^p(\mathbb R^n)$ are like $L^p(\mathbb R^n)$ when $0<p<1$ breaks down in some key points (see \cite{Go}). In order to solve this problem, Goldberg (\cite{Go}) introduced the so-called local Hardy spaces $h^p(\mathbb R^n)$. These Hardy spaces are more suited to problems related with partial differential equations. Distributions in $H^p(\mathbb R^n)$ are boundary values of conjugate harmonic functions in the upper-half space $\mathbb R^{n+1}_+$. For distributions in $h^p(\mathbb R^n)$, $\mathbb R^{n+1}_+$ is replaced by the strip $\mathbb R^n\times (0,1)$. Local Hardy spaces can be characterized by the corresponding localized maximal operators, Littlewood-Paley square functions and Riesz transforms. Local Hardy spaces have also been extended to other settings (\cite{BD}, \cite{CMM}, \cite{DY}, \cite{T} and \cite{Ya}) and adapted to operators (\cite{CMY}, \cite{GLiYa}, \cite{JiYaZh} and \cite{Kem}).

Hardy spaces with variable exponent were studied by Cruz-Uribe and Wang (\cite{CuW}) and Nakai and Sawano (\cite{NS}). The role of $L^p$-spaces is now played by the variable exponent Lebesgue spaces $L^{p(\cdot)}(\mathbb R^n)$. In the monographs \cite{CuF} and \cite{DHHR} a systematic and exhaustive study of $L^{p(\cdot)}$-spaces is presented. We will give below the properties of the spaces $L^{p(\cdot)}(\mathbb R^n)$ that we will use throughout this paper.

In \cite{CuW} and \cite{NS}, the variable exponent Hardy spaces $H^{p(\cdot)}(\mathbb R^n)$ are characterized by using the maximal functions and atomic representations. Sawano \cite{Sa} generalized the atomic characterizations of $H^{p(\cdot)}(\mathbb R^n)$. Hardy spaces $H^{p(\cdot)}(\mathbb R^n)$ are described in terms of Riesz transforms and intrinsic square functions in \cite{YaZhN} and \cite{ZhYaLi}, respectively. The dual spaces of $H^{p(\cdot)}(\mathbb R^n)$ are characterized as Campanato spaces in \cite{NS}.

Variable exponent Hardy spaces $H^{p(\cdot)}(X)$ when $X$ is a RD-homogeneous space were studied by Zhuo, Sawano and Yang (\cite{ZhSaYa}).

Lorentz spaces with variable exponents were investigated by Ephremidze, Kokilashvili and Samko \cite{EKS} and Kempka and Vyb\'iral \cite{KV}. In \cite{ABR} and \cite{LiYaYu} Hardy-Lorentz spaces with variable exponents have also been studied.

Diening, H\"ast\"o and Roudenko (\cite{DHR}) defined the variable exponent local Hardy space $h^{p(\cdot)}(\mathbb R^n)$ as a Triebel-Lizorkin space $F_{p(\cdot),2}^0(\mathbb R^n)$ with variable integrability. Nakai and Sawano (\cite{NS}) characterized $h^{p(\cdot)}(\mathbb R^n)$ by using local maximal functions. Local Hardy spaces of Musielak-Orlicz type were studied in \cite{YaYa}. Musielak-Orlicz and variable exponent Hardy spaces do not cover each other.

Hardy spaces with variable exponent associated with operators have been studied in the last years. Yang and Zhuo (\cite{YaZh1} and \cite{ZhYa}) considered non-negative and self-adjoint operators such that the semigroup $\{e^{-tL}\}_{t>0}$ satisfies upper Gaussian bounds. In \cite{YaZhZh} the pointwise boundedness is replaced by reinforced off-diagonal estimates.

Our objective in this article is to introduce local Hardy spaces with variable exponents associated with operators. We prove a molecular characterization for these local Hardy spaces. By using this characterization, we establish relations between our local Hardy spaces with variable exponent and its global counterpart. 

Before stating our results we recall some definitions and properties that we will use in the sequel.

Assume that $p:\mathbb R^n\rightarrow (0,\infty)$ is a measurable function. We define 
\[p^-=\underset{x\in \mathbb R^n}{\essinf} p(x), \quad p^+=\underset{x\in \mathbb R^n}{\esssup} p(x).\]
Here, we suppose that $0<p^-\leq p^+<\infty$. The modular of a measurable complex-valued function $f$ defined on $\mathbb R^n$ is given by
\[\varrho_{p(\cdot)}(f):=\int_{\mathbb R^n} |f(x)|^{p(x)}dx.\]
The variable exponent Lebesgue space $L^{p(\cdot)}(\mathbb R^n)$ consists of all those measurable complex-valued functions $f$ on $\mathbb R^n$ for which $\varrho_{p(\cdot)}(f)<\infty$. The Luxemburg norm on $L^{p(\cdot)}(\mathbb R^n)$ is defined by
\[\|f\|_{p(\cdot)}=\inf\bigg\{\lambda>0: \varrho_{p(\cdot)}\left(\frac{f}{\lambda}\right)\leq 1 \bigg\}, \quad f\in L^{p(\cdot)}(\mathbb R^n).\]
An extensive study about $L^{p(\cdot)}$-spaces can be encountered in \cite{CuF} and \cite{DHHR}.

A fundamental tool in harmonic analysis and, in particular, in the study of Hardy spaces is the Hardy-Littlewood maximal function $\mathcal M$. Several authors have studied $L^{p(\cdot)}$-boundedness properties of $\mathcal M$. The class of globally log-H\"older continuous functions plays an important role. 

We say that an exponent $p$ defined on $\mathbb R^n$ is globally log-H\"older continuous, abbreviated $p\in C^{\rm log}(\mathbb R^n)$, if the following two properties hold:
\begin{enumerate}[label=(LH\arabic*)]
\item \label{def: LH0} There exists $C>0$ such that 
\begin{equation*}
	|p(x)-p(y)|\leq \frac{C}{\log(e+1/|x-y|)},\quad x,y\in\mathbb R^n.
	\end{equation*}
\item \label{def: LHinfty} There exist $C>0$ and $p_\infty\geq 0$ for which
	\begin{equation*}
	|p(x)-p_\infty|\leq \frac{C}{\log(e+|x|)}, \quad  x\in \mathbb R^n.
	\end{equation*}
\end{enumerate}

In \cite{DHHMS} it was proved that $\mathcal M$ is bounded from $L^{p(\cdot)}(\mathbb R^n)$ into itself provided that $p\in C^{\rm log}(\mathbb R^n)$ and $p^->1$ (see also \cite{CuFN}). It is possible to obtain boundedness results for $\mathcal M$ under weaker conditions (\cite{D}, \cite{Le}, \cite{Nek}).

Throughout this paper we consider a densely defined operator $L:D(L)\subset L^2(\mathbb R^n)\rightarrow L^2(\mathbb R^n)$, where $D(L)$ denotes the domain of $L$, satisfying the following two assumptions: 
\begin{enumerate}[label=(A\arabic*)]
\item \label{positive-sa} $L$ is non-negative and self-adjoint. (Note that according to \ref{positive-sa} the ope\-ra\-tor $-L$ generates a bounded analytic semigroup $\{e^{-tL}\}_{t>0}$ (see \cite{Ou1})).
\end{enumerate}

\begin{enumerate}[label=(A\arabic*)]
\setcounter{enumi}{1}
\item \label{heatkernel} For every $t>0$, there exists a measurable function $W_t^L$ defined on $\mathbb R^n\times \mathbb R^n$ such that, for every $f\in L^2(\mathbb R^n)$, 
\[e^{-tL}f(x)=\int_{\mathbb R^n} W_t^L(x,y)f(y)dy,\quad x\in \mathbb R^n,\]
and it satisfies the following pointwise upper Gaussian bound:
\begin{equation}\label{gaussiankernel}
|W_t^L(x,y)|\leq C \frac{e^{-c\frac{|x-y|^2}{t}}}{t^{n/2}}, \quad x,y\in\mathbb R^n, 
\end{equation}
for some positive constants $c$ and $C$.
\end{enumerate}

From \cite[Theorem~6.17]{Ou1}, since $\{e^{-tL}\}_{t>0}$ is a bounded analytic semigroup and, for every $t>0$, \eqref{gaussiankernel} holds, it follows that, for every $k\in \mathbb N _0=\mathbb{N}\cup \{0\}$, there exist $c, C>0$ such that
\begin{equation}\label{gaussianderivatives}
\left|t^k \frac{\partial^k}{\partial t^k} W_t^L(x,y)\right|\leq \frac{C e^{-c|x-y|^2/t}}{t^{n/2}}, \quad x,y\in \mathbb R^n, \, t>0.
\end{equation}

As it can be seen in \cite[p. 227]{McI}, since $L$ satisfies the assumption \ref{positive-sa}, $L$ has a bounded functional calculus and satisfies quadratic estimates (see for instance \cite[\S 8, Theorem, p.~225]{McI}). In particular, for every $k\in \mathbb N_0$, we have that 
\begin{equation}\label{boundedL-Pfunction}
\|g_L^k (f)\|_2\leq C\|f\|_2, \quad f\in L^2(\mathbb R^n),
\end{equation}
where $g_L^k$ represents the $k$-th vertical Littlewood-Paley square function associated with $L$ which is defined by
\[g_L^k(f)(x)=\left(\int_0^\infty |(tL)^k e^{-tL}(f)(x)|^2 \frac{dt}{t}\right)^{1/2},\quad x\in \mathbb R^n,\]
for $f\in L^2(\mathbb R^n)$.

There exist many operators satisfying the assumptions \ref{positive-sa} and \ref{heatkernel}. For instance, the following operators fulfill both assumptions (\cite{Ou2}):
\begin{enumerate}[label=(\roman*)]
\item $L=-\sum\limits_{k,j=1}^n \frac{\partial}{\partial x_j}\left(a_{kj}\frac{\partial}{\partial x_k}\right)$ where $a_{kj}=a_{jk}$ are real-valued functions in $L^\infty(\mathbb R^n)$, $k,j=1,\dots,n$, and $L$ is a uniformly elliptic operator;
\item the Schr\"odinger operator with magnetic field defined by
\[L=-\sum\limits_{k=1}^n \left(\frac{\partial}{\partial x_k}-i b_k\right)^2+V,\]
where, for every $k\in \{1,...,n\}$, $b_k\in L^2_{\rm{loc}}(\mathbb R ^n)$ is a real-valued function and $0\leq V\in L^1_{\rm{loc}}(\mathbb R ^n)$.
\end{enumerate}

We now define the local Hardy space $\hLp$ with variable exponent $p(\cdot)$ associated to the operator $L$. Our definition is motivated by those ones due to Carbonaro, McIntosh and Morris (\cite{CMM}), who defined the local Hardy space $h^1$ of differential form on Riemannian manifolds, and Cao, Mayboroda and Yang (\cite{CMY}) who introduced local Hardy spaces $h^p_A$, $0<p\leq 1$ associated with inhomogeneous higher order elliptic operators $A$.

We consider the area square function $S_L$ defined by
\[S_L(f)(x)=\left(\int_0^\infty \int_{B(x,t)} |t^2 L e^{-t^2L}(f)(y)|^2 \frac{dy dt}{t^{n+1}}\right)^{1/2},\quad x\in \mathbb R^n,\]
and the localized area square integral function $S_L^{\rm{loc}}$ defined by
\[S_L^{\rm{loc}}(f)(x)=\left(\int_0^1 \int_{B(x,t)} |t^2 L e^{-t^2L}(f)(y)|^2 \frac{dy dt}{t^{n+1}}\right)^{1/2},\quad x\in \mathbb R^n,\]
where $B(x,t)$ denotes the ball in $\mathbb R^n$ centered at $x$ with radius $t$.

By \eqref{boundedL-Pfunction} we deduce that both $S_L$ and $S_L^{\rm{loc}}$ are bounded (sublinear) operators from $L^2(\mathbb R^n)$ into itself.

The (global) Hardy space $\HLp$ associated with $L$ was defined in \cite{YaZh1} as follows. A function $f\in L^2(\mathbb R^n)$ is in $\mathbb H^{p(\cdot)}_L(\mathbb R^n)$ when $S_L(f)\in L^{p(\cdot)}(\mathbb R^n)$. The Hardy space $\HLp$ is the completion of $\mathbb H^{p(\cdot)}_L(\mathbb R^n)$ with respect to the quasi-norm $\|\cdot\|_{\HLp}$ given by
\[\|f\|_{\HLp}=\|S_L(f)\|_{p(\cdot)}, \quad f\in \mathbb H^{p(\cdot)}_L(\mathbb R^n).\]

We introduce the local Hardy space $\hLp$ associated with $L$ in the following way. A function $f\in L^2(\mathbb R^n)$ is in $\mathbbl{h}^{p(\cdot)}_L(\mathbb R^n)$ when $S_L^{\rm{loc}}(f)\in L^{p(\cdot)}(\mathbb R^n)$ and $S_I(e^{-L}f)\in L^{p(\cdot)}(\mathbb R^n)$. Here, $S_I$ denotes the area square integral associated with the identity operator. We consider the quasi-norm $\|\cdot\|_{\hLp}$ given by
\[\|f\|_{\hLp}=\|S_L^{\rm{loc}}(f)\|_{p(\cdot)}+\|S_I(e^{-L}f)\|_{p(\cdot)}, \quad f\in \mathbbl h^{p(\cdot)}_L(\mathbb R^n).\]
We define the local Hardy space $\hLp$ as the completion of $\mathbbl{h}^{p(\cdot)}_L(\mathbb R^n)$ with respect to $\|\cdot\|_{\hLp}$. 

Our local Hardy space $\hLp$ should meet the following two properties:
\begin{enumerate}[label=(\Roman*)]
\item \label{globalinlocal}$\HLp \subset \hLp$;
\item \label{var=const}if $p(x)=p$, $x\in \mathbb R^n$, with $0<p\leq 1$, then $\hLp=h^p_L(\mathbb R^n)$ where, by $h^p_L(\mathbb R^n)$ we denote the local Hardy space (with constant exponent) defined in \cite{Go}, \cite{JiYaZh} and \cite{Kem}.
\end{enumerate}

It is not clear from the above definition that \ref{globalinlocal} and \ref{var=const} hold. However these properties will be transparent by using the molecular characterization of $\hLp$ that we establish in Theorem \ref{molchar}.

Let $B=B(x_B,r_B)$ be a ball in $\mathbb R^n$ with $x_B\in \mathbb R^n$ and $r_B>0$. If $\lambda>0$, we define $\lambda B=B(x_B,\lambda r_B)$. For every $i\in\mathbb N$, $S_i(B)=2^iB\setminus 2^{i-1}B$ and $S_0(B)=B$. Let $M\in \mathbb N _0$ and $\varepsilon>0$. We distinguish two types of local molecules: a measurable function $m\in L^2(\mathbb R^n)$ is a $(p(\cdot),2,M,\varepsilon)_{L,{\rm loc }}$-molecule associated to the ball $B=B(x_B, r_B)$ 
\begin{enumerate}[label=(\arabic*)]
\item[(a)] of (I)-type: when $r_B\geq 1$ and $\|m\|_{L^2(S_i(B))}\leq 2^{-i\varepsilon}|2^iB|^{1/2}\|\chi_{2^iB}\|_{p(\cdot)}^{-1}$, $i\in \mathbb N _0$, 
\item[(b)] of (II)-type: $r_B\in (0,1)$ and there exists $b\in L^2(\mathbb R^n)$ such that $m=L^M b$ and, for every $k\in \{0,1,\dots,M\}$,
\[\|L^k(b)\|_{L^2(S_i(B))}\leq 2^{-i\varepsilon}r_B^{2(M-k)}|2^iB|^{1/2}\|\chi_{2^iB}\|_{p(\cdot)}^{-1},\quad i\in \mathbb N _0. \]
\end{enumerate}

We say that $f\in L^2(\mathbb R^n)$ is in the space $\mathbbl h^{p(\cdot)}_{L,\textrm{mol},M,\varepsilon}(\mathbb R^n)$ when, for every $j\in \mathbb N$, there exist $\lambda_j>0$ and a $(p(\cdot),2,M,\varepsilon)_{L,{\rm loc}}$-molecule $m_j$ associated to the ball $B_j$ such that $f=\sum\limits_{j\in \mathbb{N}}\lambda_j m_j$ in $L^2(\mathbb R^n)$ and 
\[\left(\sum\limits_{j\in \mathbb{N}}\left(\frac{\lambda_j \chi_{B_j}}{\|\chi_{B_j}\|_{p(\cdot)}}\right)^{\underline{p}}\right)^{1/\underline{p}}\in L^{p(\cdot)}(\mathbb R^n).\]
Here and in the sequel, $\underline{p}=\min\{1,p^-\}$. We define the quasi-norm $\|\cdot\|_{h^{p(\cdot)}_{L,\textrm{mol},M,\varepsilon}(\mathbb R^n)}$ as follows. For every $f\in \mathbbl h^{p(\cdot)}_{L,\textrm{mol},M,\varepsilon}(\mathbb R^n)$, 
\[\|f\|_{h^{p(\cdot)}_{L,\textrm{mol},M,\varepsilon}(\mathbb R^n)}=\inf \left\|\left(\sum\limits_{j\in \mathbb{N}}\left(\frac{\lambda_j \chi_{B_j}}{\|\chi_{B_j}\|_{p(\cdot)}}\right)^{\underline{p}}\right)^{1/\underline{p}}\right\|_{p(\cdot)},\]
where the infimum is taken over all the sequences $\{(\lambda_j,B_j)\}_{j\in \mathbb N}$ such that, for every $j\in \mathbb N$, $\lambda_j>0$ and $B_j$ is a ball for which there exists a $(p(\cdot),2,M,\varepsilon)_{L,{\rm loc}}$-molecule $m_j$ associated with it, verifying $f=\sum\limits_{j\in \mathbb{N}}\lambda_j m_j$ in $L^2(\mathbb R^n)$ and 
\[\left(\sum\limits_{j\in \mathbb{N}}\left(\frac{\lambda_j \chi_{B_j}}{\|\chi_{B_j}\|_{p(\cdot)}}\right)^{\underline{p}}\right)^{1/\underline{p}}\in L^{p(\cdot)}(\mathbb R^n).\]
The space $h^{p(\cdot)}_{L,\textrm{mol},M,\varepsilon}(\mathbb R^n)$ is the completion of $\mathbbl h^{p(\cdot)}_{L,\textrm{mol},M,\varepsilon}(\mathbb R^n)$ with respect to \newline
$\|\cdot\|_{h^{p(\cdot)}_{L,\textrm{mol},M,\varepsilon}(\mathbb R^n)}$.

We now establish a molecular characterization of the space $\hLp$.

\begin{thm}\label{molchar}
Let $p\in C^{\rm log}(\mathbb R^n)$ such that $p^+<2$. 
\begin{enumerate}[label=(\roman*)]
\item \label{molinlocal}If $\varepsilon >n(\frac{1}{p^-}-\frac{1}{p^+})$ and $M\in \mathbb N$ such that $2M>n(2/p^--1/2-1/p^+)$, then $h^{p(\cdot)}_{L,\textrm{mol},M,\varepsilon}(\mathbb R^n)\subset \hLp$.
\item \label{localinmol}If $\varepsilon>0$ and $M\in \mathbb N$, then $\hLp \subset h^{p(\cdot)}_{L,\textrm{mol},M,\varepsilon}(\mathbb R^n)$.
\end{enumerate}
The embeddings in \ref{molinlocal} and \ref{localinmol} are algebraic and topological.
\end{thm}

In \cite[Definition~1.6]{ZhYa} Zhuo and Yang defined global $(p(\cdot),2,M)_L$-atoms with $M\in \mathbb N$ as follows. A function $a$ is a $(p(\cdot),2,M)_L$-atom associated with the ball $B$ when there exists $b\in L^2(\mathbb R^n)$ such that $a=L^M b$, $\supp (L^k b)\subset B$ for every $k\in \{0,1,\dots, M\}$ and 
\[\|L^k b\|_2\leq r_B^{2(M-k)}|B|^{1/2}\|\chi_B\|_{p(\cdot)}^{-1}, \quad k\in \{0,1,\dots, M\}, \]
where $r_B$ represents the radius of $B$.

Zhuo and Yang (\cite[Theorem~1.8]{ZhYa}) proved that if $p\in C^{\rm log}(\mathbb R^n)$ with $p^+\in (0,1]$ (in fact, the proof is also valid for $p^+\in (0,2)$) and $M\in \mathbb N$, $M>\frac{n}{2}[\frac{1}{p^-}-1]$, then, for every $f\in L^2(\mathbb  R^n)\cap \HLp$, there exist, for each $j\in \mathbb N$, $\lambda_j>0$ and a $(p(\cdot),2,M)_L$-atom associated with the ball $B_j$ such that $f=\sum_{j\in \mathbb{N}}\lambda_j a_j$ in $L^2(\mathbb R^n)$ and 
\[\left(\sum_{j\in\mathbb{N}}\left(\frac{\lambda_j \chi_{B_j}}{\|\chi_{B_j}\|_{p(\cdot)}}\right)^{\underline{p}}\right)^{1/\underline{p}}\in L^{p(\cdot)}(\mathbb R^n).\]
If $\alpha\in \mathbb R$, we denote $[\alpha]=\inf\{k\in\mathbb{Z}: k+1>\alpha\}$. 

It is not hard to see that if $a$ is a $(p(\cdot),2,M)_L$-atom, then $a$ is also a $(p(\cdot),2,M,\varepsilon)_{L, \rm{loc}}$-molecule, for every $M\in \mathbb N$ and $\varepsilon>0$.

From Theorem \ref{molchar} we deduce that $\HLp \subset \hLp$, provided that $p\in C^{\rm log}(\mathbb R^n)$ and $p^+<2$. Hence, the above property \ref{globalinlocal} holds.

On the other hand, the molecular and atomic characterizations established in \cite{Go}, \cite{GLiYa}, \cite{JiYaZh} and \cite{Kem}, and Theorem \ref{molchar} allow us to prove that the above property \ref{var=const} is also satisfied.

Also, from Theorem \ref{molchar} we deduce the following result.
\begin{cor}\label{Lp=Hp} Let $p\in C^{\rm log}(\mathbb R^n)$ such that $p^+<2$. Then, $\hLp \subset L^{p(\cdot)}(\mathbb R^n)$. Furthermore, $\hLp=\HLp=L^{p(\cdot)}(\mathbb R^n)$, provided that $p^->1$.
\end{cor}
In order to show Theorem \ref{molchar} we need to consider the global Hardy space $\HIp$ associated with the identity operator. Note that, for every $t>0$, the operator $e^{-tI}$ is not an integral operator. However, as it was proved in \cite[Remark~3.9(b)]{CMY}, $e^{-tI}$ satisfies the Davies-Gaffney estimates, but it does not verify reinforced off-diagonal estimates. Then, the Hardy space $\HIp$ is not included in the ones studied in \cite{YaZhZh}. A remarkable property is the equality $\HIp=\LpQ$, established in Theorem \ref{LpQ=HIp}. The space $\LpQ$ is a space of mixed-norm with variable exponent (see Section 2 for details), and it is an extension to variable exponent of certain function spaces considered in \cite{CMY} and \cite{CMM}.

In the proof of Theorem \ref{molchar} we also consider the local tent space $t^{p(\cdot)}_2(\mathbb R^n)$ that is used to prove Theorem \ref{molchar}\ref{molinlocal}.

We now state a generalization of \cite[Theorem~7]{Kem}.

\begin{thm}\label{local=global}Let $p\in C^{\rm log}(\mathbb R^n)$ such that $p^+<2$. Then, $\hLp=\HLp$, provided that $\inf \sigma(L)>0$, where $\sigma(L)$ represents the spectrum of $L$ in $L^2(\mathbb R^n)$. The equality is algebraic and topological.
\end{thm}

Kempainnen's result (\cite[Theorem~7]{Kem}) appears when $p(x)\equiv 1$, $x\in \mathbb R^n$, in Theorem \ref{local=global}.

We prove Theorem \ref{local=global} by using molecular characterizations of $\hLp$ (Theorem \ref{molchar}) and of $\HLp$ (\cite[Theorem~1.8]{ZhYa}).

The Hermite operator (also called harmonic oscillator) is the Schr\"odinger operator with potential $V(x)=|x|^2$, $x\in \mathbb R^n$, that is, $H=-\Delta+|x|^2$. The spectrum of $H$ in $L^2(\mathbb R^n)$ (in fact, in $L^p(\mathbb R^n)$ for $1<p<\infty$) is $\sigma(H)=\{2k+n:k\in \mathbb N\}$. Hence, Theorem \ref{local=global} applies for $H$.

The twisted Laplacian operator $\mathcal L$ is defined by
\[\mathcal L=-\frac{1}{2} \sum\limits_{j=1}^n ((\partial_{x_j}+i y_j)^2+(\partial{y_j}-i x_j)^2),\quad (x,y)\in \mathbb R^n \times \mathbb R^n.\]
It is a magnetic Schr\"odinger operator with potential zero. The spectrum of $\mathcal L$ is $\sigma(\mathcal L)=\{2k+n: k\in \mathbb N\}$. Spectral projections for $\mathcal L$ have been studied by Koch and Ricci (\cite{KR}), Stempak and Zienkiewicz (\cite{SZi}) and Thangavelu (\cite{Th}). Theorem \ref{local=global} also works for $L=\mathcal L$.

The Hardy space $H^1_{\mathcal L}(\mathbb R^{2n})$ associated with the twisted Laplacian $\mathcal L$ was studied by Mauceri, Piccardello and Ricci (\cite{MPR}). More recently, Hardy spaces $H^p_{\mathcal L}(\mathbb R^{2n})$, for $0<p<1$, were treated in \cite{HuLi} and \cite{HuWa}. The strategy employed in \cite{MPR}, and then in \cite{HuLi} and \cite{HuWa} is ``ad-hoc'' for the operator $\mathcal L$. Our procedure applies to a wider class of operators. The Hardy spaces $H^p_{\mathcal L}(\mathbb R^{2n})$, for $0<p\leq 1$, are defined by using maximal functions given in terms of the so called twisted convolution (see for instance \cite{HuLi}, \cite{HuWa} and \cite{MPR}). They prove that $H^p_{\mathcal L}(\mathbb R^{2n})$ is actually a local Hardy space $h^p_{\mathcal L}(\mathbb R^{2n})$ defined by a local maximal function and closely connected with the classical local Hardy space $h^p(\mathbb R^{2n})$ introduced by Goldberg \cite{Go}. In \cite{MPR}, it is fundamental an atomic description of the space $H^1_{\mathcal L}(\mathbb R^{2n})$ by using atoms that satisfy a twisted null moment property that depends on the center of the ball associated with the atom. For us, it is not clear how to extend this atomic characterization to variable exponent Hardy spaces associated with $\mathcal L$. The obstruction is that the molecular quasi-norm works in a very different way when we have variable exponents. These differences are transparent in the results of H\"ast\"o (\cite{Ha}). On the other hand, global Hardy spaces associated with general magnetic Schr\"odinger operators were investigated in \cite{JiYaYa1} with constant exponents and in \cite{YaYa} in the Musielak-Orlicz setting. The procedure used in those articles depends strongly on the properties of certain partial differential equations that do not hold for general operators.

Another consequence of the molecular characterization in Theorem \ref{local=global} is given in the following result.

\begin{thm}\label{localglobalL+I} Let $p\in C^{\rm log}(\mathbb R^n)$ with $p^+<2$. Then, $\hLp=\HLIp$ algebraically and topologically.
\end{thm}

This result generalizes the one obtained in \cite{JiYaZh} where it is proved that $h^1_L(\mathbb R^n)=H^1_{L+I}(\mathbb R^n)$. Other property of this type was established for $0<p\leq 1$ and inhomogeneous higher order elliptic operators in \cite[Theorem~4.6]{CMY}. 

Since $\sigma(L+I)\subset [1,\infty)$, by combining Theorems \ref{local=global} and \ref{localglobalL+I} we get the following Corollary.

\begin{cor}\label{localL+mI} Let $p\in C^{\rm log}(\mathbb R^n)$ with $p^+<2$. Then, for every $m\in \mathbb N$, $\hLp=H^{p(\cdot)}_{L+mI}(\mathbb R^n)$ algebraically and topologically.
\end{cor}

The remaining of the paper is organized as follows. In Section 2 we recall definitions and main properties of tent and local tent spaces with variable exponents. In Section 3 we establish that $H^{p(\cdot)}_I(\mathbb R^n)$ coincides with certain spaces denoted by $\LpQ$ what allow us to give an atomic and a molecular decomposition of $H^{p(\cdot)}_I(\mathbb R^n)$. The proofs of Theorem \ref{molchar} and Corollary \ref{Lp=Hp} are presented in Section 4. Theorems \ref{local=global} and \ref{localglobalL+I} are proved in Section 5 and 6 respectively.

Throughout this paper, by $C$ and $c$ we shall always denote positive constants that may change from one line to another.

\section{Tent and local tent spaces with variable exponent}
Tent spaces were introduced by Coifman, Meyer and Stein in their celebrated paper \cite{CMS}. The definition and main properties of variable exponent tent spaces can be found in \cite{YaZh1} and \cite{ZhYaLi}.

Let $F$ be a complex-valued measurable function on $\mathbb{R}^{n+1}_+$. We define $T(F)$ as follows
\[T(F)(x)=\left(\int_0^\infty \int_{B(x,t)}|F(y,t)|^2\frac{dydt}{t^{n+1}}\right)^{1/2},\quad x\in \mathbb{R}^n.\]

We say that $F$ is in the tent space $T_2^{p(\cdot )}(\mathbb{R}^n)$ when $T(F)\in L^{p(\cdot )}(\mathbb{R}^n)$. On $T_2^{p(\cdot )}(\mathbb{R}^n)$ we consider the quasi-norm $\|\cdot \|_{T_2^{p(\cdot )}}$ defined by
\[\|F\|_{T_2^{p(\cdot )}(\mathbb{R}^n)}=\|T(F)\|_{p(\cdot )},\quad F\in T_2^{p(\cdot )}(\mathbb{R}^n).
\]
Thus, $T_2^{p(\cdot )}(\mathbb{R}^n)$ is a quasi-Banach space. When $p(x)=r$, $x\in \mathbb{R}^n$, for some $r\in (0,\infty )$, the space $T_2^{p(\cdot )}(\mathbb{R}^n)$ reduces to the tent space $T_2^r(\mathbb{R}^n)$ defined in \cite{CMS}.

Let $q\in (1,\infty )$. A measurable function $A$ in $\mathbb{R}^{n+1}_+$ is said to be a $(T_2^{p(\cdot)},q)$-atom associated to a ball $B=B(x_B,r_B)$ with $x_B\in \mathbb{R}^n$ and $r_B>0$, when

$(i)$ $\mbox{supp }A\subset \widehat{B}$, where $\widehat{B}$ denotes the tent supported on $B$, that is, 
\[\widehat{B}=\{(x,t)\in \mathbb{R}^{n+1}_+: x\in B(x_B,r_B-t),\;0<t<r_B\}.
\]

$(ii)$ $\|A\|_{T^q_2(\mathbb{R}^n)}\leq |B|^{1/q}\|\chi _B\|_{p(\cdot )}^{-1}$.

Furthermore, we say that $A$ is a $(T_2^{p(\cdot)},\infty)$-atom when $A$ is a $(T_2^{p(\cdot)},q)$-atom for every $q\in (1,\infty )$.

In \cite[Theorem 2.16]{ZhYaLi} the following atomic decomposition of the elements of $T_2^{p(\cdot )}(\mathbb{R}^n)$ was established.

\begin{lem}\label{LemmaT1}
Let $p\in C^{\rm log}(\mathbb{R}^n)$. Then, there exists $C>0$ such that, for every $F\in T_2^{p(\cdot )}(\mathbb{R}^n)$ we can find, for each $j\in \mathbb{N}$, $\lambda _j>0$ and a $(T_2^{p(\cdot )},\infty )$-atom $A_j$ associated with the ball $B_j$, such that
\[F(x,t)=\sum_{j\in \mathbb{N}}\lambda _jA_j(x,t),
\]
where the series converges absolutely for almost every $(x,t)\in \mathbb{R}^{n+1}_+$, in $T_2^{p(\cdot )}(\mathbb{R}^n)$ and, when $F\in T_2^2(\mathbb{R}^n)$ also in $T_2^2(\mathbb{R}^n)$, and we have that
\[\left\|\left(\sum\limits_{j\in \mathbb N} \left(\frac{\lambda_j  \chi_{B_j}}{\|\chi_{B_j}\|_{p(\cdot)}}\right)^{\underline{p}}\right)^{1/\underline{p}}\right\|_{p(\cdot)}\leq C\|F\|_{T_2^{p(\cdot )}(\mathbb{R}^n)}.
\]
\end{lem}

Non-uniformly local tent spaces with constant exponent have been studied by Amenta and Kemppainen (\cite{AK}). We consider uniformly local tent spaces with variable exponents. 

Let $F$ be a measurable complex-valued function on $\mathbb{R}^n\times (0,1)$. We define $T^{\textrm{loc}}(F)$ by
\[T^{\textrm{loc}}(F)(x)=\left(\int_0^1\int_{B(x,t)}|F(y,t)|^2\frac{dydt}{t^{n+1}}\right)^{1/2},\quad x\in \mathbb{R}^n.
\]
We say that $F$ is in the local tent space $t_2^{p(\cdot )}(\mathbb{R}^n)$ when $T^{\textrm{loc}}(F)\in L^{p(\cdot )}(\mathbb{R}^n)$.

We define 
\[\|F\|_{t_2^{p(\cdot )}(\mathbb{R}^n)}=\|T^{\textrm{loc}}(F)\|_{p(\cdot )},\quad F\in t_2^{p(\cdot )}(\mathbb{R}^n).
\]

The space $t_2^{p(\cdot )}(\mathbb{R}^n)$ endowed with the quasi-norm $\|\cdot \|_{t_2^{p(\cdot )}(\mathbb{R}^n)}$ is a quasi-Banach space. By taking $p(x)=r$, $x\in \mathbb{R}^n$, with $r\in (0,1)$, the space $t_2^{p(\cdot )}(\mathbb{R}^n)$ reduces to the space $t_2^r(\mathbb{R}^n)$ considered in \cite{CMY} and \cite{CMM}.

We define atoms in $t_2^{p(\cdot )}(\mathbb{R}^n)$ as follows. Here we consider only 2-atoms. A measurable function $A$ is a $(t_2^{p(\cdot )},2)$-atom associated with the ball $B$ when $\mbox{supp }A\subset \widehat{B}\cap (\mathbb{R}^n\times (0,1))$ and $\|A\|_{t_2^2(\mathbb{R}^n)}\leq |B|^{1/2}\|\chi _B\|_{p(\cdot )}^{-1}$.

By combining the arguments in \cite[Theorem 3.6]{CMM} and \cite[Theorem 2.16]{ZhYaLi} (see also \cite{Ru}) we can prove the following atomic representation for the elements of $t_2^{p(\cdot )}(\mathbb{R}^n)$.

\begin{lem}\label{LemmaT2}
Let $p\in C^{\rm log}(\mathbb{R}^n)$. For a certain $C>0$ we have that, for every $F\in t_2^{p(\cdot )}(\mathbb{R}^n)$ there exist, for each $j\in \mathbb{N}$, $\lambda _j>0$ and a $(t_2^{p(\cdot )},2)$-atom $A_j$ associated with the ball $B_j$, such that
\[F(x,t)=\sum_{j\in \mathbb{N}}\lambda _jA_j(x,t),\quad (x,t)\in \mathbb{R}^n\times (0,1),
\]
where the series converges absolutely for almost every $(x,t)\in \mathbb{R}^n\times (0,1)$, in $t_2^{p(\cdot )}(\mathbb{R}^n)$ and, in $t_2^2(\mathbb{R}^n)$, provided that $F\in t_2^2(\mathbb{R}^n)$, and also 
\[\left\|\left(\sum\limits_{j\in \mathbb N} \left(\frac{\lambda_j  \chi_{B_j}}{\|\chi_{B_j}\|_{p(\cdot)}}\right)^{\underline{p}}\right)^{1/\underline{p}}\right\|_{p(\cdot)}\leq C\|F\|_{t_2^{p(\cdot )}(\mathbb{R}^n)}.
\]
\end{lem}

\section{The spaces \texorpdfstring{$\LpQ$}{LpQ} and \texorpdfstring{$\HIp$}{HIp}}

A sequence of cubes $\mathcal{Q}=\{Q_j\}_{j\in \mathbb N}$ is said to be a unit cube structure on $\mathbb R^n$ when the following properties are satisfied:
	\begin{enumerate}[label=(\roman*)]
		\item \label{covering}$\cup_{j\in\mathbb{N}}\ Q_j=\mathbb R^n$;
		\item \label{disjointinteriors}$Q_i^\circ \cap Q_j^\circ =\emptyset$, $i,j\in \mathbb N$, $i\neq j$;
		\item \label{unitballs} there exist $0<\delta\leq 1$ and a sequence $\{B_j\}_{j\in\mathbb N}$ of balls in $\mathbb R^n$ with radius $1$ such that $\delta B_j\subset Q_j \subset B_j$, for each $j\in \mathbb{N}$.
	\end{enumerate}
  Here $A^\circ$ denotes the interior of the set $A$. 

Suppose that $\mathcal Q$ is a unit cube structure. The space $\LpQ$ consists of all those measurable functions $f$ on $\mathbb R^n $ for which the series
\[\sum\limits_{j\in\mathbb{N}} |Q_j|^{-1/2} \left\|f\chi_{Q_j}\right\|_2\chi_{Q_j}\]
converges in $L^{p(\cdot)}(\mathbb R^n)$. We define, for $f \in \LpQ$, the quantity
\[\|f\|_{\LpQ}=\left\|\sum\limits_{j\in\mathbb{N}} |Q_j|^{-1/2} \left\|f\chi_{Q_j}\right\|_2\chi_{Q_j}\right\|_{p(\cdot)}.\]
It is clear that $\LpQ\subset L^{2}_{\rm{loc}}(\mathbb R^n)$.

Note that, according to the property \ref{unitballs} of $\mathcal Q$, for every $f\in \LpQ$
\begin{equation}\label{equivnormLpQ}
(\delta ^n\omega_n)^{1/2}\|f\|_{\LpQ}\leq \left\|\sum\limits_{j\in\mathbb{N}}  \left\|f\chi_{Q_j}\right\|_2\chi_{Q_j}\right\|_{p(\cdot)}\leq \omega_n^{1/2}\|f\|_{\LpQ} 
\end{equation}
where $\omega_n=|B(0,1)|$. Thus, a function $f\in \LpQ$ if and only if the series $\sum\limits_{j\in\mathbb{N}}  \left\|f\chi_{Q_j}\right\|_2 \chi_{Q_j}$ converges in $L^{p(\cdot)}(\mathbb R^n)$.

Suppose now that $\alpha:\mathbb N\rightarrow \mathbb N$ is a bijective mapping and $f\in \LpQ$. We define, for every $m\in \mathbb N$, $F_m=\sum\limits_{j=1}^m \left\|f\chi_{Q_{\alpha(j)}}\right\|_2\chi_{Q_{\alpha(j)}}$. It is clear that the sequence $\{F_m\}_{m\in \mathbb N}$ is increasing and verifies $F_m\leq \sum\limits_{j\in\mathbb{N}}  \left\|f\chi_{Q_j}\right\|_2\chi_{Q_j}$. Then, since $p^+<\infty$, according to \cite[Lemma 3.2.8 (c)]{DHHR}, $F_m\rightarrow \sum_{j\in \mathbb{N}} \left\|f\chi_{Q_{\alpha(j)}}\right\|_2\chi_{Q_{\alpha(j)}}$, when $m \rightarrow \infty$ in $L^{p(\cdot)}(\mathbb R^n)$, and \[\sum_{j\in \mathbb{N}} \left\|f\chi_{Q_{\alpha(j)}}\right\|_2\chi_{Q_{\alpha(j)}}\leq \sum_{j\in \mathbb{N}}  \left\|f\chi_{Q_j}\right\|_2\chi_{Q_j}.\]
Similarly, we can see that \[\sum_{j\in \mathbb{N}}\left\|f\chi_{Q_j}\right\|_2\chi_{Q_j}\leq\sum_{j\in \mathbb{N}} \left\|f\chi_{Q_{\alpha(j)}}\right\|_2\chi_{Q_{\alpha(j)}} .\]
Hence, the series $\sum\limits_{j\in \mathbb{N}}\left\|f\chi_{Q_{\alpha(j)}}\right\|_2\chi_{Q_{\alpha(j)}}$ converges unconditionally in $L^{p(\cdot)}(\mathbb R^n)$.
By \cite[Theorem 2.4]{Ha}, for every $f\in \LpQ$, we know that
\[\|f\|_{\LpQ}\sim \left(\sum_{j\in \mathbb{N}} \left(\left\|f\chi_{Q_j}\right\|_2\left\|\chi_{Q_j}\right\|_{p(\cdot)}\right)^{p_\infty}\right)^{1/p_\infty}\]
where $p_\infty$ was defined in \ref{def: LHinfty}.

The space $\LpQ$ does not depend on the unit cube structure on $\mathbb R^n$.

\begin{prop} Let $p\in C^{\rm log}(\mathbb R^n)$ and suppose that $\mathcal Q$ and $\mathcal R$ are two unit cube structures on $\mathbb R^n$. Then, $\LpQ=L^{p(\cdot)}_{\mathcal R}(\mathbb R^n)$ algebraical and topologically.
\end{prop}

\begin{proof}We write $\mathcal Q=\{Q_j\}_{j\in \mathbb N}$ and $\mathcal R=\{R_j\}_{j\in \mathbb N}$. There exists $J_0\in \mathbb N$ such that, for every $j\in \mathbb N$, the set
	\[A_j=\{i\in \mathbb N: R_i\cap Q_j\neq \emptyset\}\]
has at most $J_0$ elements and then we can find a cube $S_j$ in $\mathbb R^n$ such that its side-length is less or equal than $2J_0$ and $\cup_{i\in A_j} R_i\subset S_j$. According to \cite[Lemma~2.6]{ZhYaLi}, we know that there exists $C>0$ such that, for every $j\in \mathbb{N}$,   
\[\frac{\|\chi_{S_j}\|_{p(\cdot)}}{\|\chi_{R_i}\|_{p(\cdot)}}\leq C\left(\frac{|S_j|}{|R_i|}\right)^{1/p^-}\leq C,\quad i\in A_j.\]
Hence, given $f\in L^{p(\cdot)}_{\mathcal R}(\mathbb R^n)$, 
\begin{align*}
\sum\limits_{j\in\mathbb{N}}\left(\left\|f\chi_{Q_j}\right\|_2\left\|\chi_{Q_j}\right\|_{p(\cdot)}\right)^{p_\infty}&\leq \sum\limits_{j\in\mathbb{N}}\left(\left\|f\chi_{\{\cup_{i\in A_j} R_i\}}\right\|_2\left\|\chi_{S_j}\right\|_{p(\cdot)}\right)^{p_\infty}\\ 
&\leq C \sum\limits_{j\in\mathbb{N}}\left(\sum\limits_{i\in A_j}\left\|f\chi_{R_i}\right\|_2\left\|\chi_{R_i}\right\|_{p(\cdot)}\right)^{p_\infty}\\
&\leq C \sum\limits_{i\in\mathbb{N}}\left(\left\|f\chi_{R_i}\right\|_2\left\|\chi_{R_i}\right\|_{p(\cdot)}\right)^{p_\infty},
\end{align*}
where we have used that there exists $J_1\in \mathbb N$ such that for every $i\in \mathbb N$, the set $\{j\in \mathbb N: R_i\cap Q_j\neq \emptyset\}$ has at most $J_1$ elements. Thus, $f\in \LpQ$. By repeating the argument used above, interchanging the roles of $\mathcal Q$ and $\mathcal R$, we obtain the desired result.
\end{proof}

Let $B$ be a ball in $\mathbb R^n$ with radius greater than or equal to $1$. A measurable function $a$ on $\mathbb R^n$ is called a $L^{p(\cdot)}_{\mathcal Q}$-atom associated with $B$ when $a$ is supported on $B$ and $\left\|a\right\|_2\leq |B|^{1/2}\|\chi_B\|_{p(\cdot)}^{-1}$.

Next, we establish a characterization of $\LpQ$ by using atoms.

\begin{thm}\label{atomchar}Assume that $p\in C^{\rm log}(\mathbb R^n)$ with $p^+<2$, $\mathcal Q$ is a unit cube structure on $\mathbb R^n$  and $f$ a measurable function on $\mathbb R^n$. The following assertions are equivalent.

(i) $f\in \LpQ$.

(ii) For every $j\in \mathbb N$, there exist $\lambda_j>0$ and a $L^{p(\cdot)}_{\mathcal Q}$-atom $a_j$ associated with a ball $B_j$ with radius greater than or equal to $1$ such that $f=\sum\limits_{j\in\mathbb{N}} \lambda_j a_j$ in $\LpQ$ and
	\[\left\|\left(\sum\limits_{j\in\mathbb{N}} \left(\frac{\lambda_j \chi_{B_j}}{\|\chi_{B_j}\|_{p(\cdot)}}\right)^{\underline{p}} \right)^{1/\underline{p}}\right\|_{p(\cdot)}<\infty.\]
Moreover, the quantities $\|f\|_{\LpQ}$ and $\|f\|_{L^{p(\cdot)}_{\mathcal Q,at}(\mathbb R^n)}$ are comparable, where
\[\|f\|_{L^{p(\cdot)}_{\mathcal Q,at}(\mathbb R^n)}=\inf\left\|\left(\sum\limits_{j\in\mathbb{N}} \left(\frac{\lambda_j \chi_{B_j}}{\|\chi_{B_j}\|_{p(\cdot)}}\right)^{\underline{p}} \right)^{1/\underline{p}}\right\|_{p(\cdot)}\]
and the infimum is taken over all the sequences $\{(\lambda_j,B_j)\}_{j\in \mathbb N}$ verifying that $\lambda_j>0$ and there exists a $L^{p(\cdot)}_{\mathcal Q}$-atom $a_j$ associated with the ball $B_j$ with radius greater or equal than $1$ such that $f=\sum\limits_{j\in\mathbb{N}} \lambda_j a_j$ in $\LpQ$.
\end{thm}

\begin{rem}
The statement $(i)\Rightarrow (ii)$ is true when $p\in C^{\rm log}(\mathbb{R}^n)$ with $p_+<\infty$.
\end{rem}
\begin{proof} 

We write $\mathcal Q=\{Q_j\}_{j\in \mathbb N}$. Since $\mathcal Q$ is a unit cube structure, there exists $0<\delta\leq 1$ such that, for every $j\in \mathbb N$, $\delta B_j \subset Q_j \subset B_j$ for certain ball $B_j$ with radius $1$. 

Suppose first that $f\in \LpQ$. For every $j\in \mathbb N$ for which $\left\|f\chi_{Q_j}\right\|_2\neq 0$, we define 
	\[\lambda_j=\frac{\left\|f\chi_{Q_j}\right\|_2 \|\chi_{B_j}\|_{p(\cdot)}}{|B_j|^{1/2}}\quad \text{and}\quad a_j=\frac{1}{\lambda _j}f\chi_{Q_j},\]
and, in other case, $\lambda_j=0$ and $a_j=0$. 
Clearly, $f=\sum\limits_{j\in \mathbb N} \lambda_j a_j$ a.e. in $\mathbb R^n$.

Let $j_1,j_2\in \mathbb N$, $j_1\leq j_2$. We have that 
\[\sum\limits_{j=j_1}^{j_2} \lambda_j a_j = \sum\limits_{j=j_1}^{j_2} f\chi_{Q_j}\quad \rm{a.e.}\mbox{ in }\mathbb{R}^n,\]
and, by \eqref{equivnormLpQ}, and since $Q_j^\circ \cap Q_k^\circ =\emptyset$, $j,k\in \mathbb N$, $j\neq k$,
\[\left\|\sum\limits_{j=j_1}^{j_2} \lambda_j a_j\right\|_{\LpQ}=\left\|\sum\limits_{j=j_1}^{j_2} f \chi_{Q_j}\right\|_{\LpQ}\leq C\left\|\sum\limits_{j=j_1}^{j_2} \left\|f\chi_{Q_j}\right\|_2 \chi_{Q_j}\right\|_{p(\cdot)}.\]
This means that $\left\{\sum_{j=1}^m \lambda_j a_j\right\}_{m\in \mathbb N}$ is a Cauchy sequence in $\LpQ$ and the completeness of $\LpQ$ implies that $\lim_{m\rightarrow \infty} \sum_{j=1}^m \lambda_j a_j=g$ in the sense of $\LpQ$, for certain $g\in \LpQ$. Then, there exists an increasing function $h:\mathbb N\rightarrow \mathbb N$ such that $\lim_{m\rightarrow \infty} \sum_{j=1}^{h(m)} \lambda_j a_j=g$ a.e. in $\mathbb R^n$. Hence, $f=g$ a.e. in $\mathbb{R}^n$.

It is clear that, for each $j\in \mathbb{N}$, $a_j$ is a $L_{\mathcal{Q}}^{p(\cdot )}$-atom associated to $B_j$.

Also, by \cite[Lemma 2.6]{ZhYaLi} and since $Q_j\subset B_j\subset \delta ^{-1}Q_j$, we have that  $\|\chi _{B_j}\|_{p(\cdot )}\sim\|\chi_{Q_j}\|_{p(\cdot)}\sim \|\chi_{\delta^{-1} Q_j}\|_{p(\cdot)}$, for each $j\in \mathbb N$. Here, if $\alpha >0$ and $Q$ is a cube with sides of length $\ell _Q$, $\alpha Q$ represents the cube with the same center as $Q$ and with sides of length $\alpha \ell _Q$. By using \cite[Remark 3.16]{YaZh1} and that  $Q_j^\circ \cap Q_k^\circ =\emptyset$, $j,k\in \mathbb N$, $j\neq k$, we deduce that
\begin{align*}
\left\|\left(\sum\limits_{j\in \mathbb N} \left(\frac{\lambda_j \chi_{B_j}}{\|\chi_{B_j}\|_{p(\cdot)}}\right)^{\underline{p}}\right)^{1/\underline{p}} \right\|_{p(\cdot)}&\leq C \left\|\left(\sum\limits_{j\in \mathbb N} \left(\frac{\lambda_j \chi_{\delta ^{-1}Q_j}}{\|\chi_{\delta^{-1}Q_j}\|_{p(\cdot)}}\right)^{\underline{p}}\right)^{1/\underline{p}} \right\|_{p(\cdot)}\\
&\hspace{-4cm}\leq  C \left\|\left(\sum\limits_{j\in \mathbb N} \left(\frac{\lambda_j \chi_{Q_j}}{\|\chi_{Q_j}\|_{p(\cdot)}}\right)^{\underline{p}}\right)^{1/\underline{p}} \right\|_{p(\cdot)}\leq C \left\|\sum\limits_{j\in \mathbb N} \frac{\lambda_j \chi_{Q_j}}{\|\chi_{Q_j}\|_{p(\cdot)}}\right\|_{p(\cdot)}\\
&\hspace{-4cm}\leq C\|f\|_{\LpQ}<\infty.
\end{align*}

Conversely, assume that $f=\sum_{j\in\mathbb{N}} \lambda_j a_j$ in $\LpQ$ where, for every $j\in\mathbb N$, $\lambda_j>0$ and $a_j$ is a $L^{p(\cdot)}_{\mathcal Q}$-atom associated with the ball $B_j$ with radius greater or equal than $1$ satisfying that
	\[\left\|\left(\sum\limits_{j\in\mathbb{N}} \left(\frac{\lambda_j \chi_{B_j}}{\|\chi_{B_j}\|_{p(\cdot)}}\right)^{\underline{p}} \right)^{1/\underline{p}}\right\|_{p(\cdot)}<\infty.\]
For every $j\in \mathbb N$, we define the set $J_j=\{i\in \mathbb N: B_j\cap Q_i\neq \emptyset\}$. Then, we can write
\begin{align*}
\|f\|_{\LpQ}& \leq  \left\| \sum\limits_{i\in \mathbb N} \sum\limits_{j\in \mathbb N} \lambda_j \left\|a_j\chi_{Q_i}\right\|_{2} \chi_{Q_i}\right\|_{p(\cdot)}=\left\|  \sum\limits_{j\in \mathbb N} \lambda_j \sum\limits_{i\in J_j} \left\|a_j\chi_{Q_i}\right\|_{2} \chi_{Q_i}\right\|_{p(\cdot)}\\
&=\left\| \sum\limits_{j\in \mathbb N} \lambda_j b_j\right\|_{p(\cdot)}\leq \left\| \left(\sum\limits_{j\in \mathbb N} \left(\lambda_j b_j\right)^{\underline{p}}\right)^{1/\underline{p}}\right\|_{p(\cdot)},
\end{align*}
where $b_j=\sum\limits_{i\in J_j} \left\|a_j\chi_{Q_i}\right\|_{2} \chi_{Q_i}$, $j\in \mathbb N$. 
It is clear that, since the radius of $B_j$ is at least $1$ and $Q_j$ is contained in a ball of radius $1$, then $\supp b_j\subset 3B_j$. Thus,
\begin{equation*}
\|b_j\|_2=\left(\sum\limits_{i\in J_j} \left\|a_j \chi_{Q_i}\right\|_2^2\right)^{1/2}\leq \left\|a_j\right\|_2\leq |B_j|^{1/2}\|\chi_{B_j}\|_{p(\cdot)}^{-1}\leq C|3B_j|^{1/2}\|\chi_{3B_j}\|_{p(\cdot)}^{-1}
\end{equation*}
Moreover, since $p^+<2$, according to \cite[Lemma 4.1]{Sa} and \cite[Remark 3.16]{YaZh1} we have that
\begin{align*}
\left\|\left(\sum\limits_{j\in \mathbb N} \left(\lambda_j b_j\right)^{\underline{p}}\right)^{1/\underline{p}}\right\|_{p(\cdot)}&\leq C\left\|\left(\sum\limits_{j\in \mathbb N} \left(\frac{\lambda_j  \chi_{3B_j}}{\|\chi_{3B_j}\|_{p(\cdot)}}\right)^{\underline{p}}\right)^{1/\underline{p}}\right\|_{p(\cdot)}\\
&\leq C\left\|\left(\sum\limits_{j\in \mathbb N} \left(\frac{\lambda_j  \chi_{B_j}}{\|\chi_{B_j}\|_{p(\cdot)}}\right)^{\underline{p}}\right)^{1/\underline{p}}\right\|_{p(\cdot)}.
\end{align*}
The proof is thus finished.
\end{proof}

\begin{rem}\label{remarkA1}In the proof of Theorem \ref{atomchar} it is established that, if $p\in C^{\rm log}(\mathbb R^n)$ then, for every $f\in \LpQ$, there exist, for every $j\in \mathbb N$, $\lambda_j>0$ and a $L^{p(\cdot)}_{\mathcal Q}$-atom $a_j$ associated with a ball $B_j$ with radius $1$ such that $f=\sum\limits_{j\in\mathbb{N}} \lambda_j a_j$ in $\LpQ$ and
	\[\left\|\left(\sum\limits_{j\in\mathbb{N}} \left(\frac{\lambda_j \chi_{B_j}}{\|\chi_{B_j}\|_{p(\cdot)}}\right)^{\underline{p}} \right)^{1/\underline{p}}\right\|_{p(\cdot)}\sim \|f\|_{\LpQ}.\]
\end{rem}

As a consequence of Theorem \ref{atomchar} we obtain the following result.

\begin{cor}\label{LpQinLp}Assume that $p\in C^{\rm log}(\mathbb R^n)$ with $p^+<2$, and $\mathcal Q$ is a unit cube structure on $\mathbb R^n$. Then, $\LpQ$ is contained in $L^{p(\cdot)}(\mathbb R^n)$.
\end{cor}

\begin{proof}Let $f\in \LpQ$. By virtue of Theorem \ref{atomchar} we can write $f=\sum_{j\in\mathbb{N}} \lambda_j a_j$ in $\LpQ$ where for every $j\in \mathbb N$, $\lambda_j>0$ and $a_j$ is a $L^{p(\cdot)}_{\mathcal Q}$-atom associated with a ball $B_j$ with radius at least $1$ satisfying that 
	\[\left\|\left(\sum\limits_{j\in\mathbb{N}} \left(\frac{\lambda_j \chi_{B_j}}{\|\chi_{B_j}\|_{p(\cdot)}}\right)^{\underline{p}} \right)^{1/\underline{p}}\right\|_{p(\cdot)}<\infty.\]
Let $j_1,j_2\in \mathbb N$, $j_1\leq j_2$. Since $p^+<2$, according to \cite[Proposition 2.11]{ZhSaYa} we deduce that
\[\left\|\sum\limits_{j=j_1}^{j_2} \lambda_j a_j\right\|_{p(\cdot)}\leq \left\|\left(\sum\limits_{j=j_1}^{j_2} (\lambda_j a_j)^{\underline{p}}\right)^{1/\underline{p}}\right\|_{p(\cdot)}\leq C\left\|\left(\sum\limits_{j\in \mathbb N} \left(\frac{\lambda_j  \chi_{B_j}}{\|\chi_{B_j}\|_{p(\cdot)}}\right)^{\underline{p}}\right)^{1/\underline{p}}\right\|_{p(\cdot)}.\] 
Then, the series $\sum_{j\in \mathbb N} \lambda_j a_j$ defines a function $g$ in $L^{p(\cdot)}(\mathbb R^n)$. Moreover, we can find an increasing function $h:\mathbb N\rightarrow \mathbb N$ such that $\lim_{m\rightarrow \infty} \sum\limits_{j=1}^{h(m)} \lambda_j a_j=g$ a.e. in $\mathbb R^n$. Hence, $f=g$ a.e. in $\mathbb{R}^n$, so $f\in L^{p(\cdot)}(\mathbb R^n)$.
\end{proof}

\begin{rem}The inclusion established in Corollary \ref{LpQinLp} is proper. In order to prove this fact, we modify the function considered in \cite[Remark~3.10]{CFJY}. Suppose that  $p\in C^{\rm log}(\mathbb R^n)$ with $p^+<2$, so $p^-<2$. We define
	\[f(x)=|x|^{\sigma -\frac{n}{p^-}}\chi_{B(0,1)(x)},\quad x\in\mathbb R^n,\]
where $n(1/p^--1/p^+)<\sigma<n(1/p^--1/2)$. It is not hard to see that $f\notin L^2_{\rm{loc}}(\mathbb R^n)$. Then, $f\notin \LpQ$. On the other hand, we have that
\[\int_{\mathbb R^n} |f(x)|^{p(x)} dx=\int_{B(0,1)} |x|^{\big(\sigma -\frac{n}{p^-}\big)p(x)} dx\leq \int_{B(0,1)} |x|^{\big(\sigma -\frac{n}{p^-}\big)p^+} dx<\infty.\]
Hence, $f\in L^{p(\cdot)}(\mathbb R^n)$.
\end{rem}

Our next objective is to characterize the Hardy space $\HIp$ associated with the identity operator $I$, by using atoms and molecules (see Propositions \ref{PA.2} and \ref{PA.3}).

Let $M\in \mathbb N _0$ and $\varepsilon>0$. A function $a\in L^2(\mathbb R^n)$ is said to be a $(p(\cdot),2,M)$-atom associated with the ball $B=B(x_B,r_B)$, with $x_B\in \mathbb{R}^n$ and $r_B>0$, when $\supp a \subset B$ and, for every $k\in \{0,1,\dots, M\}$, $\|a\|_{2}\leq r_B^{2k}|B|^{1/2}\|\chi_{B}\|_{p(\cdot)}^{-1}$. The atomic Hardy space $H_{I,\textrm{at},M}^{p(\cdot)}(\mathbb R^n)$ and the quasi-norm $\|\cdot\|_{H_{I,\textrm{at},M}^{p(\cdot)}(\mathbb R^n)}$ are defined in the usual way.

A function $m\in L^2(\mathbb R^n)$ is a $(p(\cdot),2,M,\varepsilon)$-molecule associated with the ball $B=B(x_B,r_B)$, where $x_B\in \mathbb R^n$ and $r_B>0$, when for every $i\in \mathbb N _0$,
\begin{equation}\label{molec}
\|m\|_{L^2(S_i(B))}\leq 2^{-i\varepsilon}r_B^{2k}|2^i B|^{1/2}\|\chi_{2^i B}\|_{p(\cdot)}^{-1},\quad k\in \{0,1,\dots, M\}. 
\end{equation}

The molecular Hardy space $H_{I,\textrm{mol},M,\varepsilon}^{p(\cdot)}(\mathbb R^n)$ and the quasi-norm $\|\cdot\|_{H_{I,\textrm{mol},M,\varepsilon}^{p(\cdot)}(\mathbb R^n)}$ are defined in the usual way.

We observe that for every $M\in \mathbb{N}_0$, if $a$ is a $(p(\cdot ), 2,M)$-atom, then $a$ is also a $(p(\cdot ),2,M,\varepsilon)$-molecule, for each $\varepsilon >0$. Thus, 
$H_{I,\textrm{at},M}^{p(\cdot)}(\mathbb R^n)$ is continuously contained in $H_{I,\textrm{mol},M,\varepsilon}^{p(\cdot)}(\mathbb R^n)$, for every $M\in \mathbb{N}_0$ and $\varepsilon >0$.

Also, if $m$ is a $(p(\cdot ), 2,M,\varepsilon )$-molecule associated to the ball $B=B(x_B,r_B)$, with $x_B\in \mathbb{R}^n$ and $r_B>0$, from \eqref{molec} and by using \cite[Lemma~3.8]{YaZhZh}, we get
\begin{align*}
\|m\|_2^2&=\sum_{i\in \mathbb{N}}\|m\|_{L^2(S_i(B_j))}^2\leq r_B^{4k}\sum_{i\in \mathbb{N}}2^{-2i\varepsilon }|2^iB|\|\chi _{2^iB}\|_{p(\cdot )}^{-2}\nonumber\\
&\leq Cr_B^{4k}|B|\|\chi _{B}\|_{p(\cdot )}^{-2}\sum_{i\in \mathbb{N}}2^{-2i(\varepsilon-n/2+n/p^+)},\quad k\in \{0,1,...,M\}.
\end{align*}
If $\varepsilon >n(1/2-1/p^+)$ we conclude that
\begin{equation}\label{A4}
\|m\|_2\leq Cr_B^{2k}|B|^{1/2}\|\chi _B\|_{p(\cdot )}^{-1},\quad k\in \{0,1,...,M\}.
\end{equation}

The following result will be very useful in the proof of Proposition \ref{PA.2}. Let us consider, for every $M\in \mathbb N$, the operator 
\[\Pi_M(F)(x)=\int_0^\infty t^{2M}e^{-t^2}F(x,t)\frac{dt}{t}, \quad F\in T_2^2(\mathbb R^n).\]

\begin{lem}\label{PiM}Let $M\in \mathbb N$. There exists $C_0>0$ such that if $A$ is a $(T^{p(\cdot)}_2,2)$-atom associated with the ball $B$, then $C_0\Pi_M(A)$ is a $(p(\cdot),2,M)$-atom associated with $B$ (and then, it is a $(p(\cdot ),2,M,\varepsilon)$-molecule associated with $B$, when $\varepsilon>0$). Moreover, $\Pi_M$ defines a bounded operator from $T_2^2(\mathbb R^n)$ into $L^2(\mathbb R^n)$ and it can be extended from $T_2^2(\mathbb R^n)\cap T_2^{p(\cdot)}(\mathbb R^n)$ to $T_2^{p(\cdot)}(\mathbb R^n)$ as a bounded operator from $T_2^{p(\cdot)}(\mathbb R^n)$ into $H_{I,\textrm{at},M}^{p(\cdot)}(\mathbb R^n)$ (and then, into $H_{I,\textrm{mol},M,\varepsilon}^{p(\cdot)}(\mathbb R^n)$, when $\varepsilon>0$) and, when $p^+<2$, into $L^{p(\cdot)}(\mathbb R^n)$. 
\end{lem}

\begin{proof}Suppose that $F\in L^2(\mathbb R^{n+1}_+)$. H\"older's inequality leads to
	\begin{align*}
	\|\Pi_M(F)\|_2^2&= \int_{\mathbb R^n} \left(\int_0^\infty t^{2M}e^{-t^2}|F(y,t)|\frac{dt}{t}\right)^2 dy\\
	&\leq \int_{\mathbb R^n} \int_0^\infty (t^{2M}e^{-t^2})^2 \frac{dt}{t}\int_0^\infty |F(y,t)|^2\frac{dt}{t} dy\\
	&\leq C \int_{\mathbb R^n}  \int_0^\infty \int_{B(y,t)} |F(y,t)|^2dx\frac{dtdy}{t^{n+1}} =C\|F\|_{T^2_2(\mathbb R^n)}^2.
	\end{align*}
Hence, the operator $\Pi_M$ is bounded from $T_2^2(\mathbb R^n)$ into $L^2(\mathbb R^n)$.

Let $A$ be a $(T^{p(\cdot)}_2,2)$-atom associated with the ball $B=B(x_B,r_B)$ where $x_B\in \mathbb R^n$ and $r_B>0$. Since $\supp A \subset \widehat{B}$, we have that $\supp \Pi_M(A)\subset B$ and 
\[\|\Pi_M(A)\|_{L^2(S_i(B))}=\left(\int_{S_i(B)} \left|\int_{0}^{r_B} t^{2M}e^{-t^2} A(x,t) \frac{dt}{t}\right|^2 dx\right)^{1/2}=0, \quad i\in\mathbb N .\]
On the other hand, 
\begin{align*}
\|\Pi_M(A)\|_{L^2(B)}&=\left(\int_B \left|\int_0^{r_B} t^{2M}e^{-t^2} A(x,t) \frac{dt}{t}\right|^2 dx\right)^{1/2}\\
&\leq \left(\int_B \int_0^{r_B} t^{4M-1}e^{-2t^2}dt\int_0^{r_B}|A(x,t)|^2\frac{dtdx}{t}\right)^{1/2}\\
&\leq C \left(\int_0^{r_B}t^{4M-1}e^{-2t^2}dt\right)^{1/2}\|A\|_{T_2^2(\mathbb{R}^n)}.
\end{align*}
Since 
$$
\int_0^{r_B}t^{4M-1}e^{-2t^2}dt\leq \int_0^{r_B}t^{4M-1}dt\leq Cr_B^{4M},
$$
and, for every $k\in \{0,1,\dots, M-1\}$,
$$
\int_0^{r_B}t^{4M-1}e^{-2t^2}dt\leq r_B^{4k}\int_0^\infty t^{4(M-k)-1}e^{-2t^2}dt\leq Cr_B^{4k},
$$
we conclude that, 
$$
\|\Pi _M(A)\|_{L^2(B)}\leq Cr_B^{2k}|B|^{1/2}\|\chi _B\|_{p(\cdot )}^{-1},\quad k\in 
\{0,1,...,M\}.
$$
Hence, there exists $C_0>0$ independent of $A$ such that $C_0 \Pi_M(A)$ is a $(p(\cdot),2,M)$-atom and so, $C_0 \Pi_M(A)$ is a $(p(\cdot),2,M,\varepsilon )$-molecule, when $\varepsilon>0$. 

Let $F$ be in $T_2^{p(\cdot)}(\mathbb R^n)\cap T_2^2(\mathbb R^n)$. By Lemma \ref{LemmaT1} we can write $F=\sum\limits_{j\in \mathbb N} \lambda_j A_j$, in $T_2^2(\mathbb R^n)$ and in $T_2^{p(\cdot)}(\mathbb R^n)$, where for every $j\in \mathbb N$, $\lambda_j>0$ and $A_j$ is a $(T_2^{p(\cdot)},2)$-atom associated with a ball $B_j$ satisfying that 
\begin{equation}\label{A1}
\left\|\left(\sum\limits_{j\in \mathbb N} \left(\frac{\lambda_j  \chi_{B_j}}{\|\chi_{B_j}\|_{p(\cdot)}}\right)^{\underline{p}}\right)^{1/\underline{p}}\right\|_{p(\cdot)}\leq C \|F\|_{T_2^{p(\cdot)}(\mathbb R^n)}.
\end{equation}

Since $\Pi_M$ is bounded from $T_2^2(\mathbb R^n)$ into $L^2(\mathbb R^n)$, we have that 
\[\Pi_M(F)=\sum_{j\in \mathbb{N}}\lambda_j \Pi_M(A_j),\quad \mbox{ in }L^2(\mathbb{R}^n).\]

By taking into account that $C_0\Pi_M(A_j)$ is a $(p(\cdot),2,M)$-atom associated with $B_j$, it is clear from \eqref{A1} that $\|\Pi_M(F)\|_{H_{I,\textrm{at},M}^{p(\cdot)}(\mathbb R^n)}\leq C\|F\|_{T_2^{p(\cdot)}(\mathbb R^n)}$. Also, according to \cite[Proposition 2.11]{ZhSaYa} we deduce that
\begin{align*}
\|\Pi_M(F)\|_{p(\cdot)}&\leq \left\|\left(\sum_{j\in \mathbb{N}} \left(\lambda_j \Pi_M(A_j)\right)^{\underline{p}}\right)^{1/\underline{p}}\right\|_{p(\cdot)}\\
&\leq C \left\|\left(\sum\limits_{j\in \mathbb N} \left(\frac{\lambda_j  \chi_{B_j}}{\|\chi_{B_j}\|_{p(\cdot)}}\right)^{\underline{p}}\right)^{1/\underline{p}}\right\|_{p(\cdot)}\leq C \|F\|_{T_2^{p(\cdot)}(\mathbb R^n)},
\end{align*}
provided that $p^+<2$.
\end{proof}

\begin{prop}\label{PA.2} Let $p\in C^{\rm log}(\mathbb R^n)$, $M\in\mathbb N$ and $\varepsilon>0$. There exists $C>0$ such that, for every $f\in L^2(\mathbb R^n)\cap H_I^{p(\cdot)}(\mathbb R^n)$, there exist, for each $j\in \mathbb N$, $\lambda_j>0$ and a $(p(\cdot),2,M)$-atom $a_j$ associated with the ball $B_j$ such that $f=\sum_{j\in \mathbb N} \lambda_j a_j$ in $L^2(\mathbb R^n)$ and in $H_{I,{\rm{at}}, M}^{p(\cdot)}(\mathbb R^n)$, verifying
	\[\left\|\left(\sum\limits_{j\in \mathbb N} \left(\frac{\lambda_j  \chi_{B_j}}{\|\chi_{B_j}\|_{p(\cdot)}}\right)^{\underline{p}}\right)^{1/\underline{p}}\right\|_{p(\cdot)}\leq C \|f\|_{H_I^{p(\cdot)}(\mathbb R^n)}.\]
So we have that
\[H_I^{p(\cdot)}(\mathbb R^n)\subset H_{I,\textrm{at},M}^{p(\cdot)}(\mathbb R^n)\subset H_{I,\textrm{mol},M,\varepsilon}^{p(\cdot)}(\mathbb R^n).\]
\end{prop}

\begin{proof}\label{prop A.3} Let $f\in L^2(\mathbb R^n)\cap H_I^{p(\cdot)}(\mathbb R^n)$. Then, $F(x,t)=te^{-t^2} f(x)\in T_2^2(\mathbb R^n)\cap T_2^{p(\cdot)}(\mathbb R^n)$. According to Lemma \ref{LemmaT1}, there exist, for every $j\in\mathbb N$, $\lambda_j>0$ and a $(T_2^{p(\cdot)},2)$-atom $A_j$ associated with the ball $B_j$ such that $F(x,t)=\sum_{j\in \mathbb N} \lambda_j A_j(x,t)$ in $T_2^2(\mathbb R^n)$ and in $T_2^{p(\cdot)}(\mathbb R^n)$, and satisfying that 
		\[\left\|\left(\sum\limits_{j\in \mathbb N} \left(\frac{\lambda_j  \chi_{B_j}}{\|\chi_{B_j}\|_{p(\cdot)}}\right)^{\underline{p}}\right)^{1/\underline{p}}\right\|_{p(\cdot)}\leq C \|F\|_{T_2^{p(\cdot)}(\mathbb R^n)}=C\|f\|_{H_I^{p(\cdot)}(\mathbb R^n)}.\]
Then, by Lemma \ref{PiM}, we get
	\[\Pi_M(F)=\sum\limits_{j\in \mathbb N}\lambda_j \Pi_M(A_j),\]
	in $L^2(\mathbb R^n)$ and in $H_{I,\textrm{at},M}^{p(\cdot)}(\mathbb R^n)$, and there exists $C_0>0$ such that, for every $j\in \mathbb N$, $C_0\Pi_M(A_j)$ is a $(p(\cdot),2,M)$-atom associated with $B_j$.
    
    It is clear that 
	\[\Pi_M(F)(x)=\int_0^\infty t^{2M-1}e^{-t^2}F(x,t)dt=\frac{\Gamma (M+1)}{2^{M+2}}f(x),\quad x\in \mathbb{R}^n.
    \]
Then, $f=\sum_{j\in \mathbb{N}}\lambda _ja_j$, where $a_j=2^{M+2}(\Gamma (M+1))^{-1}\Pi_M(A_j)$, and the convergence is in $L^2(\mathbb{R}^n$ and in $H_{I,\rm{at}, M}^{p(\cdot)}(\mathbb R^n)$.
	
	The proof can be finished by taking into account that $\mathbb H$ is the completion of $\mathbb H \cap L^2(\mathbb R^n)$ with respect to the quasi-norm $\|\cdot\|_{\mathbb H}$ when $\mathbb H=H_I^{p(\cdot)}(\mathbb R^n)$ or $\mathbb H=H_{I,\textrm{at},M}^{p(\cdot)}(\mathbb R^n)$.
\end{proof}

For the reverse result we need the following useful lemma. 

\begin{lem}\label{tecnico}
Let $p\in C^{\rm log}(\mathbb R^n)$ with $p^+<2$ and $T$ a bounded operator on $L^2(\mathbb{R}^n)$. Consider $f=\sum_{j\in \mathbb{N}}\lambda _jm_j$ in $L^2(\mathbb{R}^n)$, where, for each $j\in \mathbb{N}$, $\lambda _j>0$ and $m_j\in L^2(\mathbb{R}^n)$, and let $\{B_j\}_{j\in \mathbb{N}}$ be a sequence of balls in $\mathbb{R}^n$ such that
$$
\left\|\left(\sum\limits_{j\in \mathbb N} \left(\frac{\lambda_j  \chi_{B_j}}{\|\chi_{B_j}\|_{p(\cdot)}}\right)^{\underline{p}}\right)^{1/\underline{p}}\right\|_{p(\cdot)}<\infty .
$$
Suppose also that, there exist $C>0$ and $\eta >n(1/p^--1/p^+)$ such that
$$
\|T(m_j)\|_{L^2(S_i(B_j))}\leq C2^{-i\eta}|2^iB_j|^{1/2}\|\chi _{2^iB_j}\|_{p(\cdot )}^{-1},\quad i\in \mathbb{N}_0,j\in \mathbb{N}.
$$
Then, $T(f)\in L^{p(\cdot )}(\mathbb{R}^n)$ and
$$
\|T(f)\|_{p(\cdot )}
\leq C \left\|\left(\sum\limits_{j\in \mathbb N} \left(\frac{\lambda_j  \chi_{B_j}}{\|\chi_{B_j}\|_{p(\cdot)}}\right)^{\underline{p}}\right)^{1/\underline{p}}\right\|_{p(\cdot)}.
$$
\end{lem}
\begin{proof}
Since $T$ is bounded in $L^2(\mathbb{R}^n)$ we have that $T(f)=\sum_{j\in \mathbb{N}}\lambda _jT(m_j)$ in $L^2(\mathbb{R}^n)$ and we can write
\begin{align*}
\|T(f)\|_{p(\cdot )}&= \left\|\sum_{i\in \mathbb{N}}\sum_{j\in \mathbb{N}}\lambda _jT(m_j)\chi _{S_i(B_j)}\right\|_{p(\cdot )}\\
&\leq \left\|\sum_{i\in \mathbb{N}}\left(\sum_{j\in \mathbb{N}}\lambda _jT(m_j)\chi _{S_i(B_j)}\right)^{\underline{p}}\right\|_{p(\cdot )/\underline{p}}^{1/\underline{p}}\\
&\leq \left(\sum_{i\in \mathbb{N}}\left\|\sum_{j\in \mathbb{N}}\lambda _jT(m_j)\chi _{S_i(B_j)}\right\|_{p(\cdot )}^{\underline{p}}\right)^{1/\underline{p}}\\
&\leq \left(\sum_{i\in \mathbb{N}}\left\|\left(\sum_{j\in \mathbb{N}}(\lambda _jT(m_j)\chi _{S_i(B_j)})^{\underline{p}}\right)^{1/ \underline{p}}\right\|_{p(\cdot )}^{\underline{p}}\right)^{1/\underline{p}}.
\end{align*}

By using \eqref{A2} and according to \cite[Proposition 2.11]{ZhSaYa} and \cite[Remark 3.16]{YaZh1}, since $p^+<2$ we deduce that
\begin{align*}
\left\|\left(\sum_{j\in \mathbb{N}}(\lambda_jT(m_j)\chi _{S_i(B_j)})^{\underline{p}}\right)^{1/\underline{p}}\right\|_{p(\cdot )}&\leq C2^{-i\eta}\left\|\left(\sum\limits_{j\in \mathbb N} \left(\frac{\lambda_j  \chi_{2^iB_j}}{\|\chi_{2^iB_j}\|_{p(\cdot)}}\right)^{\underline{p}}\right)^{1/\underline{p}}\right\|_{p(\cdot)}\\
&\hspace{-3cm} \leq C2^{-i(\eta -n(1/r-1/p^+))}\left\|\left(\sum\limits_{j\in \mathbb N} \left(\frac{\lambda_j  \chi_{B_j}}{\|\chi_{B_j}\|_{p(\cdot)}}\right)^{\underline{p}}\right)^{1/\underline{p}}\right\|_{p(\cdot)},
\end{align*}
where $0<r<p^-$. Since $\eta >n(1/p^--1/p^+)$, by choosing $0<r<p^-$ such that $\eta >n(1/r-1/p^+)$ we conclude that 
\[\|T(f)\|_{p(\cdot )}\leq C\left\|\left(\sum\limits_{j\in \mathbb N} \left(\frac{\lambda_j  \chi_{B_j}}{\|\chi_{B_j}\|_{p(\cdot)}}\right)^{\underline{p}}\right)^{1/\underline{p}}\right\|_{p(\cdot)}<\infty .
\]
\end{proof}

\begin{prop}\label{PA.3}
Let $p\in C^{\rm log}(\mathbb R^n)$ with $p^+<2$. Let $M\in\mathbb N$, such that $2M>n(2/p^--1/2-1/p^+)$ and $\varepsilon> n(1/p^--1/p^+)$. There exists $C>0$ such that if $f=\sum_{j\in \mathbb{N}}\lambda _jm_j$ in $L^2(\mathbb R^n)$, where for every $j\in \mathbb N$, $\lambda_j>0$ and $m_j$ is a $(p(\cdot),2,M,\varepsilon)$-molecule associated with the ball $B_j$ and
	\[\left\|\left(\sum\limits_{j\in \mathbb N} \left(\frac{\lambda_j  \chi_{B_j}}{\|\chi_{B_j}\|_{p(\cdot)}}\right)^{\underline{p}}\right)^{1/\underline{p}}\right\|_{p(\cdot)}<\infty ,\]
then $f\in H_I^{p(\cdot )}(\mathbb{R}^n)$ and 
\[
\|f\|_{H_I^{p(\cdot)}(\mathbb R^n)}\leq C\left\|\left(\sum\limits_{j\in \mathbb N} \left(\frac{\lambda_j  \chi_{B_j}}{\|\chi_{B_j}\|_{p(\cdot)}}\right)^{\underline{p}}\right)^{1/\underline{p}}\right\|_{p(\cdot)}.
\]
So we have that
\[H_{I,\textrm{mol},M,\varepsilon}^{p(\cdot )}(\mathbb{R}^n)\subset H_I^{p(\cdot )}(\mathbb{R}^n),
\]
algebraic and topologically.
\end{prop}
\begin{proof} 

The operator $I$ does not satisfy a reinforced off-diagonal estimates on balls (see \cite[Assumption 2.4]{YaZhZh}) but we can use the procedure developed in the proof of \cite[Proposition 3.10]{YaZhZh} with some modifications.

Since $S_I$ is bounded from $L^2(\mathbb{R}^n)$ into itself, according to Lemma \ref{tecnico} it is sufficient to show that, there exists $\eta >n(1/p^--1/p^+)$ such that
\begin{equation}\label{A2}
\|S_I(m_j)\|_{L^2(S_i(B_j))}\leq C2^{-i\eta}|2^iB_j|^{1/2}\|\chi _{2^iB_j}\|_{p(\cdot )}^{-1}, 
\end{equation}
for each $i\in \mathbb{N}_0$ and $j\in \mathbb{N}$.

Let $i\in \mathbb{N}_0$ and $j\in \mathbb{N}$. Since $S_I$ is a bounded (sublinear) operator on $L^2(\mathbb{R}^n)$ by \eqref{A4} (for $k=0$), we get
\begin{equation}\label{3.4}
\|S_I(m_j)\|_{L^2(S_i(B_j))}\leq C\|m_j\|_2\leq C|B_j|^{1/2}\|\chi _{B_j}\|_{p(\cdot )}^{-1}.
\end{equation}

Assume now $i\geq 2$ and write
\begin{align*}
\|S_I(m_j)\|_{L^2(S_i(B_j))}^2&=\int_{S_i(B_j)}\left(\int_0^{2^{i-2}r_{B_j}}+\int_{2^{i-2}r_{B_j}}^\infty \right)\int_{B(x,t)}|t^2e^{-t^2}m_j(y)|^2\frac{dydt}{t^{n+1}}dx\\
&=I_1+I_2.
\end{align*}

We observe that if $x\in S_i(B_j)$ and $t\in (0,2^{i-2}r_{B_j})$, then $B(x,t)\subset 2^{i+1}B_j\setminus 2^{i-2}B_j$. Hence, for every $0<t<2^{i-2}r_{B_j}$, 
\begin{align*}
\{(x,y)\in \mathbb{R}^{2n}: x\in S_i(B_j),\;|x-y|<t\}\\
&\hspace{-3cm} \subset \{(x,y)\in \mathbb{R}^{2n}: y\in 2^{i+1}B_j\setminus 2^{i-2}B_j,\;|x-y|<t\},
\end{align*}
and then, 
\begin{align*}
I_1&\leq \int_0^{2^{i-2}r_{B_j}}\int_{2^{i+1}B_j\setminus 2^{i-2}B_j}|t^2e^{-t^2}m_j(y)|^2\int_{B(y,t)}dx\frac{dydt}{t^{n+1}}\\
&=C\int_0^{2^{i-2}r_{B_j}}t^3e^{-2t^2}dt\int_{2^{i+1}B_j\setminus 2^{i-2}B_j}|m_j(y)|^2dy\\
&\leq C\sum_{\ell =-1}^1 \|m_j\|_{L^2(S_{i+\ell}(B_j))}^2\leq C\sum_{\ell =-1}^1 2^{-2(i+\ell )\varepsilon }|2^{i+\ell}B_j|\|\chi _{2^{i+\ell }B_j}\|_{p(\cdot )}^{-2}\\
&\leq C2^{-2i\varepsilon}|2^iB_j|\|\chi _{2^iB_j}\|_{p(\cdot )}^{-2}.
\end{align*}
Hence
\begin{equation}\label{A5}
I_1^{1/2}\leq C2^{-i\varepsilon}|2^iB_j|^{1/2}\|\chi _{2^iB_j}\|_{p(\cdot )}^{-1}.
\end{equation}

On the other hand, by \eqref{A4} and \cite[Lemma~3.8]{YaZhZh} we can write
\begin{align*}
I_2&\leq \int_{2^{i-2}r_{B_j}}^\infty t^{3-n}e^{-2t^2} \int_{\mathbb{R}^n}|m_j(y)|^2\int_{B(y,t)}dxdydt\\
&\leq C\|m_j\|_2^2\int_{2^{i-2}r_{B_j}}^\infty t^3e^{-2t^2}dt\leq Ce^{-c(2^ir_{B_j})^2}r_{B_j}^{4k}|B_j|\|\chi _{B_j}\|_{p(\cdot )}^{-2}\\
&\leq C2^{-2in(1/2-1/p^-)}e^{-c(2^ir_{B_j})^2}r_{B_j}^{4k}|2^iB_j|\|\chi _{2^iB_j}\|_{p(\cdot)}^{-2},
\end{align*}
for each $ k\in \{0,1,...,M\}$.

So, in particular,
\begin{align*}
I_2&\leq C2^{-2in(1/2-1/p^-)}e^{-c(2^ir_{B_j})^2}r_{B_j}^{4M}|2^iB_j|\|\chi _{2^iB_j}\|_{p(\cdot)}^{-2}\\
&\leq C2^{-2i(2M+n(1/2-1/p^-))}|2^iB_j|\|\chi _{2^iB_j}\|_{p(\cdot)}^{-2},
\end{align*}
and then,
\begin{equation}\label{A6}
I_2^{1/2}\leq C2^{-i(2M+n(1/2-1/p^-))}|2^iB_j|^{1/2}\|\chi _{2^iB_j}\|_{p(\cdot)}^{-1}
\end{equation}

By combining \eqref{A5} and \eqref{A6} and taking into account also \eqref{A2} it follows that
\[\|S_I(m_j)\|_{L^2(S_i(B_j))}\leq C2^{-i\min \{\varepsilon, 2M+n(1/2-1/p^-)\}}|2^iB_j|^{1/2}\|\chi _{2^iB_j}\|_{p(\cdot )}^{-1},
\]
for every $i\in \mathbb{N}_0$.

Since $\varepsilon >n(1/p^--1/p^+)$ and $2M>n(2/p^--1/2-1/p^+)$, \eqref{A2} is established and the proof is complete.
\end{proof}

As a consequence of Propositions \ref{PA.2} and \ref{PA.3} we can prove the following result.
\begin{thm}\label{LpQ=HIp}
Let $p\in C^{\rm log}(\mathbb R^n)$ such that $p^+<2$ and $\mathcal{Q}$ a unit cube structure. Then, $L_{\mathcal{Q}}^{p(\cdot )}(\mathbb{R}^n)=H_I^{p(\cdot )}(\mathbb{R}^n)$ algebraic and topologically.
\end{thm}
\begin{proof}
Note firstly that if $a$ is a $L_\mathcal{Q}^{p(\cdot )}$-atom associated with a ball $B$ with radius greater than or equal to 1, then $a$ is also a $(p(\cdot ),2,M)$-atom, and then a $(p(\cdot ), 2, M,\varepsilon )$-molecule) associated with $B$, for every $M\in \mathbb{N}_0$ and $\varepsilon >0$. 

Let $f\in L_\mathcal{Q}^{p(\cdot )}(\mathbb{R}^n)\cap L^2(\mathbb{R}^n)$. According to Remark \ref{remarkA1} we can write $f=\sum_{j\in \mathbb{N}}\lambda _ja_j$ in $L_\mathcal{Q}^{p(\cdot )}(\mathbb{R}^n)$, where, for every $j\in \mathbb{N}$, $\lambda _j>0$ and $a_j$ is a $L_\mathcal{Q}^{p(\cdot )}$-atom associated with the ball $B_j$ having radius equal to 1, and satisfying that
\[\left\|\left(\sum\limits_{j\in \mathbb N} \left(\frac{\lambda_j  \chi_{B_j}}{\|\chi_{B_j}\|_{p(\cdot)}}\right)^{\underline{p}}\right)^{1/\underline{p}}\right\|_{p(\cdot)}\leq C\|f\|_{L_\mathcal{Q}^{p(\cdot )}(\mathbb{R}^n)}.
\]
Here $C$ does not depend on $f$.

Since $f\in L^2(\mathbb{R}^n)=L_\mathcal{Q}^2(\mathbb{R}^n)$, $f=\sum_{j\in \mathbb{N}}\lambda _ja_j$ also in $L^2(\mathbb{R}^n)$. Then, according to Proposition \ref{PA.3}, we obtain that $f\in H_I^{p(\cdot )}(\mathbb{R}^n)$ and $\|f\|_{H_I^{p(\cdot )}(\mathbb{R}^n)}\leq C\|f\|_{L_\mathcal{Q}^{p(\cdot )}(\mathbb{R}^n)}$.

Assume now that $f\in L^2(\mathbb{R}^n)\cap H_I^{p(\cdot )}(\mathbb{R}^n)$. We choose $M\in \mathbb{N}_0$ such that $2M>n(1/p^--1/2)$. By Proposition \ref{PA.2} we have that $f=\sum_{j\in \mathbb{N}}\lambda _ja_j$ in $L^2(\mathbb{R}^n)$ and in $H_{I,{\rm at}, M}^{p(\cdot )}(\mathbb{R}^n)$, where, for every $j\in \mathbb{N}$, $\lambda _j>0$ and $a_j$ is a $(p(\cdot ), 2,M)$-atom associated with the ball $B_j$, satisfying that
\[\left\|\left(\sum\limits_{j\in \mathbb N} \left(\frac{\lambda_j  \chi_{B_j}}{\|\chi_{B_j}\|_{p(\cdot)}}\right)^{\underline{p}}\right)^{1/\underline{p}}\right\|_{p(\cdot)}\leq C\|f\|_{H_I^{p(\cdot )}(\mathbb{R}^n)}.
\]
Here $C>0$ does not depend on $f$.

If the radius $r_{B_j}$ of $B_j$ is greater than or equal to 1, then $a_j$ is also a $L_\mathcal{Q}^{p(\cdot )}$-atom associated with $B_j$. Suppose now that $a$ is a $(p(\cdot ), 2,M)$-atom associated with a ball $B=B(x_B,r_B)$ with $x_B\in \mathbb{R}^n$ and $0<r_B<1$. Our objective is to see that there exists $C_0>0$ such that $C_0a$ is a $L_\mathcal{Q}^{p(\cdot )}$-atom associated with $B(x_B,1)$.

According to \cite[Lemma~3.8]{YaZhZh} we have that
\begin{align*}
\|a\|_2&\leq r_B^{2M}|B|^{1/2}\|\chi _B\|_{p(\cdot )}^{-1}\leq Cr_B^{2M+n/2-n/p^-}|B(x_B,1)|^{1/2}\|\chi _{B(x_B,1)}\|_{p(\cdot )}^{-1}\\
&\leq C|B(x_B,1)|^{1/2}\|\chi _{B(x_B,1)}\|_{p(\cdot )}^{-1}.
\end{align*}

Hence, $C_0a$ is a $L_\mathcal{Q}^{p(\cdot )}(\mathbb{R}^n)$-atom associated with $B(x_B,1)$ for a certain $C_0>0$ which does not depend on $a$.

We deduce from Theorem \ref{atomchar} that $f=\sum_{j\in \mathbb{N}}\lambda _ja_j$ in $L_\mathcal{Q}^{p(\cdot )}(\mathbb{R}^n)$ and 
\[\|f\|_{L_\mathcal{Q}^{p(\cdot )}(\mathbb{R}^n)}\leq C\|f\|_{H_I^{p(\cdot )}(\mathbb{R}^n)}.
\]
Since $L^2(\mathbb{R}^n)\cap \mathbb{H}$ is dense in $\mathbb{H}$, whenever $\mathbb{H}=L_\mathcal{Q}^{p(\cdot )}(\mathbb{R}^n)$ or $\mathbb{H}=H_I^{p(\cdot )}(\mathbb{R}^n)$, the proof is finished.
\end{proof}

\section{Proof of Theorem \ref{molchar} and Corollary \ref{Lp=Hp}}

\subsection{Proof of Theorem \ref{molchar} (i)} 

Let $\varepsilon>n(1/p^--1/p^+)$ and $M\in \mathbb{N}_0$ such that $2M>n(2/p^--1/2-1/p^+)$. Suppose that $f\in L^2(\mathbb{R}^n)\cap h_{L,\textrm{mol},M,\varepsilon}^{p(\cdot )}(\mathbb{R}^n)$. We can write $f=\sum_{j\in \mathbb N} \lambda_jm_j$ in $L^2(\mathbb{R}^n)$, where, for every $j\in \mathbb{N}$, $\lambda _j>0$ and $m_j$ is a $(p(\cdot ),2,M,\varepsilon)_{L,\textrm{loc}}$-molecule associated with the ball $B_j$, satisfying that
\[\left\|\left(\sum\limits_{j\in \mathbb N} \left(\frac{\lambda_j  \chi_{B_j}}{\|\chi_{B_j}\|_{p(\cdot)}}\right)^{\underline{p}}\right)^{1/\underline{p}}\right\|_{p(\cdot)}<\infty .
\]
In order to see that $f\in h_L^{p(\cdot )}(\mathbb{R}^n)$ we have to prove that $S_L^{\textrm{loc}}(f)$ and $S_I(e^{-L}f)$ belongs to $L^{p(\cdot )}(\mathbb{R}^n)$.

We know that the local square function $S_L^{\textrm{loc}}$, $S_I$ and $e^{-L}$ are bounded from $L^2(\mathbb{R}^n)$ into itself. So, according to Lemma \ref{tecnico} the proof will be finished when we show that, there exist $C>0$ and $\eta >n(1/p^--1/p^+)$ such that
\begin{equation}\label{objective}
\|S_L^{\rm loc}(m_j)\|_{L^2(S_i(B_j))}+\|S_I(e^{-L}(m_j)\|_{L^2(S_i(B_j))}\leq C2^{-i\eta }|2^iB_j|^{1/2}\|\chi _{2^iB_j}\|_{p(\cdot )}^{-1}, 
\end{equation}
for each $i\in \mathbb{N}_0$ and $j\in \mathbb{N}$.

Let $m$ be a $(p(\cdot ), 2,M,\varepsilon)_{L,{\rm loc}}$-molecule associated to the ball $B=B(x_B,r_B)$, with $x_B\in \mathbb{R}^n$ and $r_B>0$.

By proceeding as in \eqref{A4} we can see that
\begin{equation}\label{m}
\|m\|_2\leq C|B|^{1/2}\|\chi _B\|_{p(\cdot )}^{-1},
\end{equation}
because $\varepsilon >n(1/2-1/p^+)$. Then, by the $L^2$-boundedness of $S_L^{\rm loc}$, $S_I$ and $e^{-L}$ we can write, for every $i\in \mathbb{N}_0$,
\begin{equation}\label{gen}
\|S_L^{\rm loc}(m)\|_{L^2(S_i(B))}+\|S_I(e^{-L}(m))\|_{L^2(S_i(B))}\leq C\|m\|_2\leq C|B|^{1/2}\|\chi _B\|_{p(\cdot )}^{-1}.
\end{equation}

Assume now that $r_B\geq 1$, that is, $m$ is a $(p(\cdot ), 2,M,\varepsilon)_{L,{\rm loc}}$-molecule of (I)-type. 
Let $i\in \mathbb{N}$, $i\geq 3$. We decompose $m=m_1+m_2$, where $m_1=m\chi _{2^{i+2}B\setminus 2^{i-3}B}$.

The kernel $k_t$ of the integral operator $Le^{-t^2L}$ is given by
\[k_t(x,y)=-\frac{\partial }{\partial s}W_s^L(x,y)_{\big|s=t^2},\quad x,y\in \mathbb{R}^n,
\]
where $W_s^L$, $s>0$, is the kernel of the integral operator $e^{-sL}$. From (\ref{gaussianderivatives}) we deduce that
\begin{equation}\label{B4}
|k_t(x,y)|\leq C\frac{e^{-c|x-y|^2/t^2}}{t^{n+2}},\quad x,y\in \mathbb{R}^n\mbox{ and } t>0.
\end{equation}

We can write 
\begin{align*}
\|S_L^{\rm loc}(m)\|^2_{L^2(S_i(B))}&\leq C\Big(\|S_L^{\rm loc}(m_1)\|^2_{L^2(S_i(B))}\\
&\left.\quad +\int_{S_i(B)}\int_0^1\int_{B(x,t)}\left(\int_{\mathbb{R}^n}t^2|k_t(y,z)||m_2(z)|dz\right)^2\frac{dydt}{t^{n+1}}dx\right)\\
&=I_1+I_2.
\end{align*}

Since $S_L^{\textrm{loc}}$ is bounded from $L^2(\mathbb{R}^n)$ into itself it follows that
\begin{align}\label{ref1}
I_1&\leq C\|m_1\|_2^2=C\|m\chi _{2^{i+2B}\setminus 2^{i-3}B}\|_2^2\leq C\sum_{\ell =-2}^2\|m\|_{L^2(S_{i+\ell}(B))}^2\nonumber\\
&\leq C2^{-2i\varepsilon }|2^iB|\|\chi _{2^iB}\|_{p(\cdot )}^{-2}.
\end{align}

On the other hand, we observe that $t\in (0, 2^{i-2}r_{B})$ when $t\in (0,1)$ because $r_B\geq 1$. Then, if $x\in S_i(B)$, $t\in (0,1)$ and $y\in B(x,t)$, then $y\in 2^{i+1}B\setminus 2^{i-2}B$, and thus $|y-z|\sim 2^ir_B$, when $z\in (2^{i+2}B\setminus 2^{i-3}B)^c$.

This fact, jointly \eqref{B4}, H\"older's inequality and \eqref{m} leads to 
\begin{align}\label{ref2}
I_2&\leq C\|m\|_2^2\int_{S_i(B)}\int_0^1\int_{B(x,t)}\int_{(2^{i+2}B\setminus 2^{i-3}B)^c}\frac{e^{-c|y-z|^2/t^2}}{t^{2n}}dz\frac{dydt}{t^{n+1}}dx\nonumber\\
&\leq C\|m\|_2^2\int_{S_i(B)}\int_0^1\frac{e^{-c(2^ir_{B})^2/t^2}}{t^{2n+1}}\int_{B(x,t)}\int_{\mathbb{R}^n}\frac{e^{-c|y-z|^2/t^2}}{t^n}dzdydtdx\nonumber\\
&\leq C\|m\|_2^2\int_{S_i(B)}\int_0^1\frac{e^{-c(2^ir_{B})^2/t^2}}{t^{n+1}}dtdx\leq C\|m\|_2^2|2^iB|\int_0^1\frac{1}{t^{n+1}}\Big(\frac{t}{2^ir_{B}}\Big)^Ndt\nonumber\\
&\leq C|B|\|\chi _{B}\|_{p(\cdot )}^{-2}|2^iB|(2^ir_{B})^{-N},
\end{align}
where $N>n$.

Since $r_B\geq 1$, according to \cite[Lemma 3.8]{YaZhZh} and by choosing $N$ large enough we get
\[I_2\leq C|2^iB|\|\chi _{2^iB}\|_{p(\cdot )}^{-2}2^{-i(N-2n/p^-)}r_B^{n-N}\leq C2^{-2i\varepsilon}|2^iB|\|\chi _{2^iB}\|_{p(\cdot )}^{-2}.
\]

By combining the above estimates we obtain
\begin{equation}\label{acot1}
\|S_L^{\textrm{loc}}(m)\|_{L^2(S_i(B))}\leq C2^{-i\varepsilon }|2^iB|^{1/2}\|\chi _{2^iB}\|_{p(\cdot )}^{-1},\quad i\in \mathbb{N}, i\geq 3.
\end{equation}

We now deal with $S_I(e^{-L}m)$. We write  
\begin{align}\label{J1J2} 
\|S_I(e^{-L}m)\|_{L^2(S_i(B))}^2&=\int_{S_i(B)} \left(\int_0^{2^{i-2}r_B}+\int_{2^{i-2}r_B}^\infty \right)\int_{B(x,t)}|t^2e^{-t^2}e^{-L}(m)(y)|^2\frac{dydt}{t^{n+1}} \nonumber\\
& =J_1+J_2.
\end{align}

As it was noted before, if $x\in S_i(B)$, $t\in (0,2^{i-2}r_B)$, $y\in B(x,t)$ and $z\in (2^{i+2}B\setminus 2^{i-3}B)^c$, it follows that $|y-z|\sim 2^ir_B$. By taking into account that $e^{-L}$ is bounded on $L^2(\mathbb{R}^n)$ and proceeding as in the estimates of $I_1$ and $I_2$ we obtain
\begin{align*}
J_1&\leq C\int_{S_i(B)}\int_0^{2^{i-2}r_B}(t^2e^{-t^2})^2\int_{B(x,t)}(|e^{-L}(m_1)(y)|^2+|e^{-L}(m_2)(y)|^2)\frac{dydt}{t^{n+1}}dx\\
&\leq C\left(\int_0^{2^{i-2}r_B}(t^2e^{-t^2})^2\int_{\mathbb{R}^n}|e^{-L}(m_1)(y)|^2\int_{B(y,t)}dx\frac{dydt}{t^{n+1}}\right.\\
&\left. \quad +\int_{S_i(B)}\int_0^{2^{i-2}r_B}(t^2e^{-t^2})^2\int_{B(x,t)}\Big(\int_{(2^{i+2}r_B\setminus 2^{i-3}B)^c}e^{-|y-z|^2}|m(z)|dz\Big)^2\frac{dydt}{t^{n+1}}dx\right)\\
&\leq C\left(\|m_1\|_2^2\int_0^\infty t^3e^{-2t^2}dt+e^{-c(2^ir_B)^2}\|m\|_2^2\int_{S_i(B)}\int_0^{2^{i-2}r_B}t^3\int_{B(x,t)}dy\frac{dt}{t^{n}}dx\right)\\
&\leq C\left(\|m_1\|_2^2+e^{-c(2^ir_B)^2}|B|\|\chi _B\|_{p(\cdot )}^{-2}|2^iB|(2^ir_B)^4\right)\\
&\leq C\left(\|m_1\|_2^2+|B|\|\chi _B\|_{p(\cdot )}^{-2}|2^iB|(2^ir_B)^{-N}\right),
\end{align*}
for $N>0$.

As above, by choosing $N$ large enough it follows that
$$
J_1\leq C2^{-2i\varepsilon}|2^iB|\|\chi _{2^iB}\|_{p(\cdot )}^{-2}.
$$

On the other hand, using again the $L^2$-boundedness of $e^{-L}$, \eqref{m} and \cite[Lemma 3.8]{YaZhZh} we get
\begin{align*}
J_2&\leq C\int_{2^{i-2}r_B}^\infty t^{3-n}e^{-2t^2}\int_{\mathbb{R}^n}|e^{-L}m(y)|^2\int_{B(y,t)}dxdydt\\
&\leq C\|m\|_2^2\int_{2^{i-2}r_B}^\infty t^{3}e^{-2t^2}dt\leq Ce^{-c(2^ir_B)^2}|B|\|\chi _B\|_{p(\cdot)}^{-2}\\
&\leq Ce^{-c(2^ir_B)^2}|2^iB|\|\chi _{2^iB}\|_{p(\cdot )}^{-2}2^{-2ni(1/2-1/p^-)}\\
&\leq C2^{-2i(N+n(1/2-1/p^-))}|2^iB|\|\chi _{2^iB}\|_{p(\cdot)}^{-2},
\end{align*}
for $N>0$, because $r_B\geq 1$. We take $N$ sufficiently large to get
$$
J_2\leq C2^{-2i\varepsilon}|2^iB|\|\chi _{2^iB}\|_{p(\cdot )}^{-2},
$$
and thus, we conclude that
\begin{equation}\label{acot2}
\|S_I(e^{-L}m)\|_{L^2(S_i(B))}\leq C2^{-i\varepsilon }|2^iB|^{1/2}\|\chi _{2^iB}\|_{p(\cdot )}^{-1},\quad i\in \mathbb{N}, i\geq 3.
\end{equation}

It is clear also, from \eqref{gen} that, for $i=0,1,2$, 
$$
\|S_L^{\textrm{loc}}(m)\|_{L^2(S_i(B))}+\|S_I(e^{-L}m)\|_{L^2(S_i(B))}\leq C2^{-i\varepsilon }|2^iB|^{1/2}\|\chi _{2^iB}\|_{p(\cdot )}^{-1},
$$
which jointly to \eqref{acot1} and \eqref{acot2} leads to the estimate \eqref{objective} for the molecules of (I)-type, because $\varepsilon >n(1/p^--1/p^+)$.

Assume next that $m$ is a molecule of (II)-type.  According to \cite[(3.13)]{YaZhZh} we can find $C>0$ and $\eta >n(1/p^--1/p^+)$ such that 
\begin{equation}\label{acot3}
\|S_L^{\rm loc}(m)\|_{L^2(S_i(B))}\leq C2^{-i\eta}|2^iB|^{1/2}\|\chi _{2^iB}\|_{p(\cdot )}^{-1},\quad i\in \mathbb{N}.
\end{equation}

Let us now show that this estimate is also satisfied by $S_I(e^{-L}m)$.

Consider $b\in L^2(\mathbb{R}^n)$ such that $m=L^Mb$ and, for every $i\in \mathbb{N}_0$ and $k\in \{0,1,...,M\}$, $\|L^kb\|_{L^2(S_i(B))}\leq 2^{-i\varepsilon}r_B^{2(M-k)}|2^iB|^{1/2}\|\chi _{2^iB}\|_{p(\cdot )}^{-1}$. By proceeding as in the proof of \eqref{A4} we can see that
\begin{equation}\label{B6}
\|b\|_2\leq Cr_B^{2M}|B|^{1/2}\|\chi _B\|_{p(\cdot )}^{-1}.
\end{equation}

We observe that $L^Me^{-L}$ is a bounded operator from $L^2(\mathbb{R}^n)$ into itself, 
and if $K$ denotes the kernel of the integral operator $L^Me^{-L}$, we have that
\[K(x,y)=(-1)^M\frac{\partial ^M}{\partial s^M}W_s^L(x,y)_{\big|s=1},\quad x,y\in \mathbb{R}^n,
\]
where $W_s^L$ denotes the heat kernel associated with $e^{-sL}$ and then
\begin{equation}\label{B7}
|K(x,y)|\leq Ce^{-c|x-y|^2},\quad x,y\in \mathbb{R}^n.
\end{equation}

Let $i\in \mathbb{N}$, $i\geq 3$. We decompose $b$ by $b=b_1+b_2$, where $b_1=b\chi _{2^{i+3}B\setminus 2^{i-3}B}$. By considering $J_1$ and $J_2$ as in \eqref{J1J2} and taking into account the $L^2$-boundedness of $L^Me^{-L}$, \eqref{B6} and \eqref{B7}, we can proceed as above to obtain
\begin{align*}
J_1&\leq C\left(\|b_1\|_2^2+e^{-c(2^ir_B)^2}\|b\|_2^2|2^iB|(2^ir_B)^4\right)\\
&\leq C\left(2^{-2i\varepsilon }|2^iB|\|\chi _{2^iB}\|_{p(\cdot )}^{-2}+e^{-c(2^ir_B)^2}r_B^{4M+n}2^{2in/p^-}|2^iB|\|\chi _{2^iB}\|_{p(\cdot )}^{-2}\right)\\
&\leq C|2^iB|\|\chi _{2^iB}\|_{p(\cdot )}^{-2}(2^{-2i\varepsilon }+2^{-2i(2M +n(1/2-1/p^-))})\\
&\leq C2^{-2i\alpha}|2^{-i}B|\|\chi _{2^iB}\|_{p(\cdot )}^{-2},
\end{align*}
for some $\alpha >n(1/p^--1/p^+)$ because $\varepsilon >n(1/p^--1/p^+)$ and $2M>n(2/p^--1/2-1/p^+)$.

On the other hand, \eqref{B6} leads to
\begin{align*}
J_2&\leq Ce^{-c(2^ir_B)^2}\|b\|_2^2\leq Ce^{-c(2^ir_B)^2}r_B^{4M}|B|\|\chi _B\|_{p(\cdot )}^{-2}\\
&\leq Ce^{-c(2^ir_B)^2}r_B^{4M}2^{-2in(1/2-1/p^-)}|2^iB|\|\chi _{2^iB}\|_{p(\cdot )}^{-2}\\
&\leq C2^{-2i\alpha}|2^{-i}B|\|\chi _{2^iB}\|_{p(\cdot )}^{2},
\end{align*}
for a certain $\alpha >n(1/p^--1/p^+)$ because $2M>n(2/p^--1/2-1/p^+)$.

By combining the above estimates we get, for every $i\in \mathbb{N}$, $i\geq 3$,
\[\|S_I(e^{-L}m)\|_{L^2(S_i(B))}\leq C2^{-i\alpha}|2^{-i}B|^{1/2}\|\chi _{2^iB}\|_{p(\cdot )}^{-1},
\]
for some $\alpha >n(1/p^--1/p^+)$.

Also, by \eqref{gen} we have, for $i=0,1,2$,
\[\|S_I(e^{-L}m)\|_{L^2(S_i(B))}\leq C2^{-i\alpha}|2^iB|^{1/2}\|\chi _{2^iB}\|_{p(\cdot )}^{-1}.
\]
These estimations jointly \eqref{acot3} leads to \eqref{objective} for molecules of (II)-type.

We conclude that $f\in h_L^{p(\cdot )}(\mathbb{R}^n)$. By using a density argument we obtain that $h_{L,{\rm loc}, M,\varepsilon}(\mathbb{R}^n)\subset h_L^{p(\cdot )}(\mathbb{R}^n)$ algebraic and topologically.

\subsection{Proof of Theorem \ref{molchar} (ii)} 

We consider the operator $\pi_N$, $N\in \mathbb{N}_0$, defined by
\[\pi _N(F)(x)=\int_0^1(t^2L)^{N+1}e^{-t^2L}(F(\cdot ,t))(x)\frac{dt}{t},\quad F\in t_{2,c}^2(\mathbb{R}^n),
\]
where $t_{2,c}^2(\mathbb{R}^n)$ denotes the subspace of $t_2^2(\mathbb{R}^n)$ that consists of all of those $F\in t_2^2(\mathbb{R}^n)$ with compact support in $\mathbb{R}^n\times (0,1)$. Note that $t_{2,c}^2(\mathbb{R}^n)$ is dense in $t_2^2(\mathbb{R}^n)$ and in $t_2^{p(\cdot )}(\mathbb{R}^n)$.

The following result is very useful in our proof. 

\begin{lem}\label{pin}
Let $M, N\in \mathbb{N}_0$ and $\varepsilon >0$.

$(i)$ $\pi _N$ can be extended from $t_{2,c}^2(\mathbb{R}^n)$ to $t_2^2(\mathbb{R}^n)$ as a bounded operator from $t_2^2(\mathbb {R}^n)$ to $L^2(\mathbb{R}^n)$.

$(ii)$ Assume that $N\geq M$. There exists $C_0>0$ such that $C_0\pi _N(A)$ is a $(p(\cdot ),2,M,\varepsilon )_{L,\rm{loc}}$-molecule associated with the ball $B$ provided that $A$ is a $(t_2^{p(\cdot )},2)$-atom associated with $B$. Also, $\pi _N$ can be extended to $t^{p(\cdot )}_2(\mathbb {R}^n)$ as a bounded operator from $t^{p(\cdot )}_2(\mathbb {R}^n)$ into $h_{L,\textrm{mol}, M,\varepsilon}^{p(\cdot )}(\mathbb {R}^n)$ and
\[\left\|\left(\sum_{j\in \mathbb{N}}\left(\frac{\lambda _j\chi _{B_j}}{\|\chi _{B_j}\|_{p(\cdot )}}\right)^{\underline{p}}\right)^{1/\underline{p}}\right\|_{p(\cdot )}\leq C\|F\|_{t_2^{p(\cdot )}(\mathbb{R}^n)},
\]
where, for every $j\in \mathbb{N}$, $\lambda _j>0$ and $B_j$ is a ball in $\mathbb{R}^n$ existing a $(t_2^{p(\cdot )},2)$-atom $A_j$ associated with $B_j$ such that $F=\sum_{j\in \mathbb{N}}\lambda _jA_j$ in $t_2^2(\mathbb{R}^n)$ and in $t_2^{p(\cdot )}(\mathbb{R}^n)$. 
\end{lem}
\begin{proof}

We first establish $(i)$. Note that
\begin{align*}
\|F\|_{t_2^2(\mathbb{R}^n)}^2&=\|T^{\textrm{loc}}(F)\|_2^2=\int_{\mathbb{R}^n}\int_0^1|F(y,t)|^2\int_{B(x,t)}dx\frac{dtdy}{t^{n+1}}\\
&=w_n\int_{\mathbb{R}^n}\int_0^1|F(y,t)|^2\frac{dtdy}{t},
\end{align*}
where $w_n=|B(0,1)|$.

Suppose that $F\in t_{2,c}^2(\mathbb{R}^n)$ and $g\in D(L)$. Since $L$ is a self-adjoint operator in $L^2(\mathbb{R}^n)$ we can write
\begin{align*}
\int_{\mathbb{R}^n}\pi _N(F)(x)g(x)dx&=\int_{\mathbb{R}^n}\int_0^1(t^2L)^{N+1}e^{-t^2L}(F(\cdot ,t))(x)\frac{dt}{t}g(x)dx\\
&=\int_0^1\int_{\mathbb{R}^n}
(t^2L)^{N+1}e^{-t^2L}(F(\cdot ,t))(x)g(x)dx\frac{dt}{t}\\
&=\int_0^1\int_{\mathbb{R}^n}F(x,t)
(t^2L)^{N+1}e^{-t^2L}(g)(x)dx\frac{dt}{t}.
\end{align*}
The interchange of the order of integration can be justified by taking into account that the family of operators $\{(t^2L)^{N+1}e^{-t^2L}\}_{t>0}$ is uniformly bounded in $L^2(\mathbb{R}^n)$ (that can be seen from the kernel estimates \eqref{gaussianderivatives})  and the properties of the function $F$ and $g$.

We have that
\begin{align*}
\left|\int_{\mathbb{R}^n}\pi _N(F)(x)g(x)dx\right|&\leq \left(\int_0^1\int_{\mathbb{R}^n}|F(x,t)|^2\frac{dxdt}{t}\right)^{1/2} \\
& \times \left(\int_0^1\int_{\mathbb{R}^n}|(t^2L)^{N+1}e^{-t^2L}(g)(x)|^2\frac{dxdt}{t}\right)^{1/2}\\
&\leq \left(\int_0^1\int_{\mathbb{R}^n}|F(x,t)|^2\frac{dxdt}{t}\right)^{1/2}\|G_{L,N}^{\textrm{loc}}(g)\|_2,
\end{align*}
where 
\[G_{L,N}^{\textrm{loc}}(g)(x)=\int_0^1|(t^2L)^{N+1}e^{-t^2L}(g)(x)|^2\frac{dt}{t},\quad x\in \mathbb{R}^n.
\]
Since $G_{L,N}^{\textrm{loc}}$ is a bounded (sublinear) operator from $L^2(\mathbb{R}^n)$ into itself (see \cite[Lemma 4.1]{LYY}), we get
\[\left|\int_{\mathbb{R}^n}\pi _N(F)(x)g(x)dx\right|\leq C\|F\|_{t_2^2(\mathbb{R}^n)}\|g\|_2.
\]
Duality arguments and the fact that $D(L)$ is dense in $L^2(\mathbb R^n)$ allow us to show that $\pi _N$ can be extended from $t_{2,c}^2(\mathbb{R}^n)$ to $t_2^2(\mathbb{R}^n)$ as a bounded operator from $t_2^2(\mathbb{R}^n)$ into $L^2(\mathbb{R}^n)$ and $(i)$ is established.

(ii) Let now $A$ be a $(t_2^{p(\cdot )},2)$-atom associated with the ball $B=B(x_B,r_B)$, with $x_B\in \mathbb{R}^n$ and $r_B>0$. We are going to see that, for certain $C_0>0$, $C_0\pi _N(A)$ is a $(p(\cdot ),2,M,\varepsilon )_{L,\rm{loc}}$-molecule. 

Suppose first that $r_B\geq 1$.
Since $A\in t_2^2(\mathbb{R}^n)$, we have that $A(x,t)\chi _{(\delta , 1)}(t)\longrightarrow A(x,t)$, as $\delta \rightarrow 0^+$, in $t_2^2(\mathbb{R}^n)$, and then by $(i)$ we obtain that
\[\pi _N(A)(x)=\lim_{\delta \rightarrow 0^+}\int_\delta ^1
(t^2L)^{N+1}e^{-t^2L}(A(\cdot ,t))(x)\frac{dt}{t},\quad \mbox{ in }L^2(\mathbb{R}^n).
\]

Let $i\in \mathbb{N}$, $i\geq 2$. We have that
\[\|\pi _N(A)\|_{L^2(S_i(B))}^2=\lim_{\delta \rightarrow 0^+}\int_{S_i(B)}\left|\int_\delta ^1
(t^2L)^{N+1}e^{-t^2L}(A(\cdot ,t)(x)\frac{dt}{t}\right|^2dx.
\]
The kernel $K_t^N$ of the integral operator $(t^2L)^{N+1}e^{-t^2L}$ is given by
\[K_t^N(x,y)=(-1)^{N+1}t^{2(N+1)}\frac{\partial ^{N+1}}{\partial s^{N+1}}W_s^L(x,y)_{\big|s=t^2},\quad x,y \in \mathbb{R}^n,
\]
for every $t>0$ and from \eqref{gaussianderivatives} we obtain
\begin{equation}\label{K1}
|K_t^N(x,y)|\leq C\frac{e^{-c|x-y|^2/t^2}}{t^n},\quad x,y\in \mathbb{R}^n,t>0.
\end{equation}
It follows that
\begin{align*}
\int_{S_i(B)}\left|\int_\delta ^1
(t^2L)^{N+1}e^{-t^2L}(A(\cdot ,t))(x)\frac{dt}{t}\right|^2dx\\
&\hspace{-3cm}\leq C\int_{S_i(B)}\left(\int_\delta ^1
\int_B\frac{e^{-c|x-y|^2/t^2}}{t^n}|A(y,t)|\frac{dydt}{t}\right)^2dx\\
&\hspace{-3cm}\leq C\int_{S_i(B)}\left(\int_\delta ^1
\frac{e^{-c(2^ir_B)^2/t^2}}{t^n}\int_B|A(y,t)|\frac{dydt}{t}\right)^2dx\\
&\hspace{-3cm}\leq C|2^iB|\|A\|_{t_2^2(\mathbb{R}^n)}^2|B|\int_\delta ^1\frac{e^{-c(2^ir_B)^2/t^2}}{t^{2n+1}}dt\\
&\hspace{-3cm}\leq C|2^iB||B|^2\|\chi _B\|_{p(\cdot )}^{-2}\int_0^1\frac{1}{t^{2n+1}}\Big(\frac{t}{2^ir_B}\Big)^\beta dt\\
&\hspace{-3cm}\leq Cr_B^{2n-\beta}2^{-2i(\beta /2- n/p^-)}|2^iB|\|\chi _{2^iB}\|_{p(\cdot )}^{-2},\quad \delta >0,
\end{align*}
where $\beta >2n$. By choosing $\beta$ large enough we conclude tht, for every $i\in \mathbb{N}$, $i\geq 2$,
\[\|\pi _N(A)\|_{L^2(S_i(B))}\leq C2^{-i\varepsilon}|2^iB|^{1/2}\|\chi _{2^iB}\|_{p(\cdot )},
\]
for certain $C>0$ independent of $i$, $\varepsilon$ and $A$.

When $i=0,1$, according to $(i)$ we can write
\begin{align*}
\|\pi_N(A)\|_{L^2(S_i(B))}\leq & \|\pi _N(A)\|_2\leq C\|A\|_{t_2^2(\mathbb{R}^n)}\leq C|B|^{1/2}\|\chi_ B\|_{p(\cdot )}^{-1} \\
& \leq C2^{-i\varepsilon}|2^iB|^{1/2}\|\chi _{2^iB}\|_{p(\cdot )}^{-1}.
\end{align*}
Thus, we have established that for certain $C_0>0$, $C_0\pi _N(A)$ is a $(p(\cdot ),2,M,\varepsilon)_{L,\rm{loc}}$-molecule of $(I)$-type.

Suppose now $r_B\in (0,1)$. We are going to see that $\pi _N(A)=L^M(b)$, where 
\begin{equation}\label{b}
b=\lim_{\delta \rightarrow 0^+}\int_\delta ^{r_B} (t^2L)^{N-M+1}e^{-t^2L}(A(\cdot ,t))t^{2M-1}dt,
\end{equation}
and
\begin{equation}\label{Lkb}
L^kb=\lim_{\delta \rightarrow 0^+}\int_\delta ^{r_B} (t^2L)^{N-M+k+1}e^{-t^2L}(A(\cdot ,t))t^{2(M-k)-1}dt,\quad k=0,...,M,
\end{equation}
where the limits are understood in $L^2(\mathbb{R}^n)$.

Let $k=0,...,M$. Since, for every $\delta >0$, $L^ke^{-\delta ^2L}$ is bounded from $L^2(\mathbb{R}^n)$ into itself, we can write
\begin{align*}
\int_\delta ^1
(t^2L)^{N+1}e^{-t^2L}(A(\cdot ,t))\frac{dt}{t}&=\int_\delta ^1L^ke^{-\delta ^2L}(t^2L)^{N-k+1}e^{-(t^2-\delta ^2)L}(A(\cdot ,t))t^{2k-1}dt\\
&=L^ke^{-\delta ^2L}\int_\delta ^1(t^2L)^{N-k+1}e^{-(t^2-\delta ^2)L}(A(\cdot ,t))t^{2k-1}dt\\
&=L^k\int_\delta ^1 e^{-\delta ^2L}(t^2L)^{N-k+1}e^{-(t^2-\delta ^2)L}(A(\cdot ,t))t^{2k-1}dt\\
&=L^k\int_\delta ^1 (t^2L)^{N-k+1}e^{-t^2L}(A(\cdot ,t))t^{2k-1}dt,\quad \delta\in (0,1).
\end{align*}

We also observe that, since $A(x,t)t^{2k}\chi _{(\delta,1)}(t)\longrightarrow A(x,t)t^{2k}$, as $\delta \rightarrow 0^+$, in $t_2^2(\mathbb{R}^n)$, by $(i)$ we get
\[\lim_{\delta \rightarrow 0^+}\int_\delta ^1 (t^2L)^{N-k+1}e^{-t^2L}(A(\cdot ,t))(x)t^{2k-1}dt=\pi _{N-k}(t^{2k}A(\cdot ,t))(x)
\]
in $L^2(\mathbb{R}^n)$. Since $L$ is a closed operator we deduce that
\begin{align*}
\pi_N(A)&=\lim_{\delta \rightarrow 0^+}\int_\delta ^1 (t^2L)^{N+1}e^{-t^2L}(A(\cdot ,t))\frac{dt}{t}\\
&=L^k\lim_{\delta \rightarrow 0^+}\int_\delta ^1 (t^2L)^{N-k+1}e^{-t^2L}(A(\cdot ,t))t^{2k-1}dt.
\end{align*}
In particular, when $k=M$, we obtain \eqref{b} and \eqref{Lkb}. 

Let us consider $i\in \mathbb{N}$, $i\geq 2$ and assume that $g\in L^2(\mathbb{R}^n)$ such that $\mbox{supp }g\subset S_i(B)$. We have that
\begin{align*}
\int_{\mathbb{R}^n} & L^k(b)(x)g(x)dx  \\
& = \lim_{\delta \rightarrow 0^+}\int_{\mathbb{R}^n}\int_\delta ^{r_B} (t^2L)^{N-M+k+1}e^{-t^2L}(A(\cdot ,t))(x)t^{2(M-k)-1}dtg(x)dx\\
&=\lim_{\delta \rightarrow 0^+}\int_\delta ^{r_B} \int_{\mathbb{R}^n}(t^2L)^{N-M+k+1}e^{-t^2L}(A(\cdot ,t))(x)g(x)dxt^{2(M-k)-1}dt\\
&=\lim_{\delta \rightarrow 0^+}\int_\delta ^{r_B} \int_{\mathbb{R}^n}(t^2L)^{N-M+k+1}e^{-(t^2-\delta ^2 )L}(A(\cdot ,t))(x)e^{-\delta ^2L}g(x)dxt^{2(M-k)-1}dt\\
&=\lim_{\delta \rightarrow 0^+}\int_\delta ^{r_B} \int_BA(x,t)(t^2L)^{N-M+k+1}e^{-t^2L}(g)(x)dxt^{2(M-k)-1}dt.
\end{align*}
By using H\"older's inequality,\eqref{K1} and \cite[Lemma 3.8]{YaZhZh} it follows that
\begin{align}\label{Gloc}
\left|\int_{\mathbb{R}^n}L^k(b)(x)g(x)dx\right|&\leq r_B^{2(M-k)}\int_0^{r_B}\int_B |A(x,t)||(t^2L)^{N-M+k+1}e^{-t^2L}(g)(x)|\frac{dxdt}{t}\nonumber\\
&\hspace{-2cm}\leq r_B^{2(M-k)}\|A\|_{t_2^2(\mathbb{R}^n)}\|G_{L,N-M+k}^{\rm loc}(g)\|_2\\
&\hspace{-2cm}\leq Cr_B^{2(M-k)}\|A\|_{t_2^2(\mathbb{R}^n)}\left(\int_0^{r_B}\int_B\left(\int_{S_i(B)}\frac{e^{-c|x-y|^2/t^2}}{t^n}|g(y)|dy\right)^2\frac{dxdt}{t}\right)^{1/2}\nonumber\\
&\hspace{-2cm}\leq Cr_B^{2(M-k)}\|A\|_{t_2^2(\mathbb{R}^n)}\left(\int_0^{r_B}\int_B\int_{S_i(B)}\frac{e^{-c|x-y|^2/t^2}}{t^n}|g(y)|^2dy\frac{dxdt}{t}\right)^{1/2}\nonumber\\
&\hspace{-2cm}\leq Cr_B^{2(M-k)}\|A\|_{t_2^2(\mathbb{R}^n)}\|g\|_2\left(\int_0^{r_B}\frac{e^{-c(2^ir_B)^2/t^2}}{t}dt\right)^{1/2}\nonumber\\
&\hspace{-2cm}\leq Cr_B^{2(M-k)}|B|^{1/2}\|\chi _B\|_{p(\cdot )}^{-1}\|g\|_2\left(\int_0^{r_B}\frac{1}{t}\Big(\frac{t}{2^ir_B}\Big)^\beta dt\right)^{1/2}\nonumber\\
&\hspace{-2cm}\leq C2^{-i(\beta /2 -n/2-n/p^-)}r_B^{2(M-k)}|2^iB|^{1/2}\|\chi _{2^iB}\|_{p(\cdot )}^{-1}2^{-i\eta}\|g\|_2, \nonumber
\end{align}
with $\beta >0$. By choosing $\beta $ sufficiently large we get
\[
\|L^k(b)\|_{L^2(S_i(B))}\leq C2^{-i\varepsilon}r_B^{2(M-k)}|2^iB|^{1/2}\|\chi _{2^iB}\|_{p(\cdot )}^{-1}.
\]

By considering that $G_{L,N-M+k}^{\rm{loc}}$ is a bounded (sublinear) operator from $L^2(\mathbb{R}^n)$ into itself, from \eqref{Gloc} we deduce that 
\begin{align*}
\left|\int_{\mathbb{R}^n}L^k(b)(x)g(x)dx\right|&\leq Cr_B^{2(M-k)}\|A\|_{t_2^2(\mathbb{R}^n)}\|g\|_2\leq Cr_B^{2(M-k)}|B|^{1/2}\|\chi _B\|_{p(\cdot )}^{-1}\|g\|_2,
\end{align*}
and we get
\[\|L^k(b)\|_2\leq C2^{-i\varepsilon}r_B^{2(M-k)}|2^iB|^{1/2}\|\chi _{2^iB}\|_{p(\cdot )}^{-1},\quad i=0,1.
\]

Thus we conclude that $\pi _N(A)$ is a $(p(\cdot ),2,M,\varepsilon )_{L,\rm{loc}}$-molecule of $(II)$-type.

To finish the proof of $(ii)$ we have to show that $\pi _N$ can be extended to $t_2^{p(\cdot )}(\mathbb{R}^n)$ as a bounded operator from $t_2^{p(\cdot )}(\mathbb{R}^n)$ into $h_{L,\rm{mol}, M,\varepsilon }^{p(\cdot )}(\mathbb{R}^n)$. 

Let $F\in t_2^{p(\cdot )}(\mathbb{R}^n)\cap t_2^2(\mathbb{R}^n)$. According to Lemma \ref{LemmaT2} we can write $F=\sum_{j\in \mathbb{N}}\lambda _jA_j$ in $t_2^2(\mathbb{R}^n)$ and in $t_2^{p(\cdot )}(\mathbb{R}^n)$, where, for every $j\in \mathbb{N}$, $\lambda _j>0$ and $A_j$ is a $(t_2^{p(\cdot )},2)$-atom associated with the ball $B_j$, and satisfying that
\[\left\|\left(\sum_{j\in \mathbb{N}}\left(\frac{\lambda _j\chi _{B_j}}{\|\chi _{B_j}\|_{p(\cdot )}}\right)^{\underline{p}}\right)^{1/\underline{p}}\right\|_{p(\cdot )}\leq C\|F\|_{t_2^{p(\cdot )}(\mathbb{R}^n)}.
\]
By using $(i)$, $\pi _N(F)=\sum_{j\in \mathbb{N}}\lambda _j\pi _N(A_j)$ in $L^2(\mathbb{R}^n)$. Since, for a certain $C_0>0$, for every $j\in \mathbb{N}$, $C_0\pi _N(A_j)$ is a $(p(\cdot ),2,M,\varepsilon )_{L,\rm{loc}}$-molecule associated to $B_j$ $\pi _N(F)\in h_{L,\textrm{mol}, M,\varepsilon}^{p(\cdot )}(\mathbb{R}^n)$ and 
\[\|\pi _N(F)\|_{h_{L,\textrm{mol}, M,\varepsilon}^{p(\cdot )}(\mathbb{R}^n)}\leq C\|F\|_{t_2^{p(\cdot )}(\mathbb{R}^n)}.
\]

A standard procedure allows us to extend $\pi_N$ to $t_2^{p(\cdot )}(\mathbb{R}^n)$ as a bounded operator from $t_2^{p(\cdot )}(\mathbb{R}^n)$ to $h_{L,\textrm{mol}, M,\varepsilon}^{p(\cdot )}(\mathbb{R}^n)$. 
\end{proof}

Let us consider $N\in \mathbb{N}$, $N\geq M$. Assume that $f\in L^2(\mathbb{R}^n)\cap h_L^{p(\cdot )}(\mathbb{R}^n)$. Then, $S_L^{\textrm{loc}}(f)\in L^{p(\cdot )}(\mathbb{R}^n)$ and $e^{-L}f\in H_I^{p(\cdot )}(\mathbb{R}^n)$.

According to \cite[Theorem 2.3]{BDY} (see also \cite[(3.12) and (3.13)]{CMY}, we can write
\begin{equation}\label{1}
f=c_{N+2}\int_0^1(t^2L)^{N+2}e^{-2t^2L}f\frac{dt}{t}+\sum_{\ell =0}^{N+1}c_\ell L^\ell e^{-2L}f,\quad f\in L^2(\mathbb{R}^n),
\end{equation}
for certain $c_{\ell}\in \mathbb{R},\ell =0,...,N+2$. Here this integral is understood as $\lim_{\delta \rightarrow 0^+}\int_\delta ^1$ in $L^2(\mathbb{R}^n)$. 

Since $e^{-L}f\in L^2(\mathbb{R}^n)\cap H_I^{p(\cdot )}(\mathbb{R}^n)$, by Theorems \ref{atomchar} and \ref{LpQ=HIp}, for every $j\in \mathbb{N}$, there exist $\lambda _j>0$ and a $L_\mathcal{Q}^{p(\cdot )}$-atom $a_j$ associated with the ball $B_j=B(x_{B_j},r_{B_j})$ with $x_{B_j}\in \mathbb{R}^n$ and $r_{B_j}\geq 1$, such that $e^{-L}f=\sum_{j\in \mathbb{N}}\lambda _ja_j$ in $L^2(\mathbb{R}^n)$ and in $H_I^{p(\cdot )}(\mathbb{R}^n)$, and
\[\left\|\left(\sum_{j\in \mathbb{N}}\left(\frac{\lambda _j\chi _{B_j}}{\|\chi _{B_j}\|_{p(\cdot )}}\right)^{\underline{p}}\right)^{1/\underline{p}}\right\|_{p(\cdot )}\leq C\|e^{-L}f\|_{H_I^{p(\cdot )}(\mathbb{R}^n)}.
\]

Let $\ell \in \{0,...,N+1\}$. We have that $L^\ell e^{-2L}f=\sum_{j\in \mathbb{N}}\lambda_jL^\ell e^{-L}a_j$, in $L^2(\mathbb{R}^n)$, because $L^\ell e^{-L}$ is a bounded operator in $L^2(\mathbb{R}^n)$. We are going to see that there exists $C_0>0$ such that $C_0L^\ell e^{-L}a_j$ is a $(p(\cdot ),2,M,\varepsilon )_{L,\rm{loc}}$-molecule of $(I)$-type associated with the ball $B_j$, for every $j\in \mathbb{N}$. Let $j\in \mathbb{N}$. The kernel $K_\ell$ of the integral operator $L^\ell e^{-L}$ is
\[K^\ell (x,y)=(-1)^\ell \frac{\partial ^\ell}{\partial t^\ell }W_t^L(x,y)_{\big|t=1},\quad x,y\in \mathbb{R}^n.
\]
Then, from \eqref{gaussianderivatives}
\[|K^\ell (x,y)|\leq Ce^{-c|x-y|^2},\quad x,y\in \mathbb{R}^n.
\]

Let $i\in \mathbb{N}$, $i\geq 2$. We have that
\begin{align*}
\|L^\ell e^{-L}(a_j)\|_{L^2(S_i(B_j))}^2&  \leq C\int_{S_i(B_j)}\left(\int_{B_j}e^{-c|x-y|^2}|a_j(y)|dy\right)^2dx\\
&\leq C\|a_j\|_2^2\int_{S_i(B_j)}\int_{B_j}e^{-c|x-y|^2}dydx\\
&\leq C\|a_j\|_2^2|2^iB_j|e^{-c(2^ir_{B_j})^2}\\
&\leq C|B_j|\|\chi _{B_j}\|_{p(\cdot )}^{-2}|2^iB_j|e^{-c(2^ir_{B_j})^2 }\\
&\leq 
C\frac{r_{B_j}^n}{(2^ir_{B_j})^\beta}2^{2in/p^-}|2^iB_j|\|\chi _{2^iB_j}\|_{p(\cdot )}^{-2}\\
&\leq C2^{-2i(\beta /2-n/p^-)}|2^iB_j|\|\chi _{2^iB_j}\|_{p(\cdot )}^{-2},
\end{align*}
for $\beta >n$. We choose $\beta$ sufficiently large to get
$$
\|L^\ell e^{-L}(a_j)\|_{L^2(S_i(B_j))}\leq C2^{-i\varepsilon}|2^iB_j|^{1/2}\|\chi _{2^iB_j}\|_{p(\cdot )}^{-1}.
$$
On the other hand, since $L^\ell e^{-L}$ is bounded from $L^2(\mathbb{R}^n)$ into itself, we obtain, for $i=0,1$,
\[\|L^\ell e^{-L}(a_j)\|_{L^2(S_i(B_j))}\leq C\|a_j\|_2\leq C|B_j|^{1/2}\|\chi _{B_j}\|_{p(\cdot )}^{-1}\leq C2^{-i\varepsilon}|2^iB_j|^{1/2}\|\chi _{2^iB_j}\|_{p(\cdot )}^{-1}.
\]
We have thus proved that, for certain $C_0$, $C_0L^\ell e^{-L}(a_j)$ is a $(p(\cdot ), 2,M,\varepsilon )_{L,\rm{loc}}$-molecule of $(I)$-type associated with $B_j$. Note that $C_0$ does not depend on $j$.

We have established that $L^\ell e^{-2L}f\in h_{L,\textrm{mol},M,\varepsilon}^{p(\cdot )}(\mathbb{R}^n)$ and
\begin{align}\label{2}
\|L^\ell e^{-2L}f\|_{h_{L,\textrm{mol},M,\varepsilon}^{p(\cdot )}(\mathbb{R}^n)}&\leq 
\left\|\left(\sum_{j\in \mathbb{N}}\left(\frac{\lambda _j\chi _{B_j}}{\|\chi _{B_j}\|_{p(\cdot )}}\right)^{\underline{p}}\right)^{1/\underline{p}}\right\|_{p(\cdot )}\nonumber\\
&\leq C\|e^{-L}f\|_{H_I^{p(\cdot )}(\mathbb{R}^n)}.
\end{align}

We now define $f_1=\int_0^1(t^2L)^{N+2}e^{-2t^2L}f\frac{dt}{t}$. Since $S_L^{\textrm{loc}}(f)\in L^{p(\cdot )}(\mathbb{R}^n)\cap L^2(\mathbb{R}^n)$, $t^2Le^{-t^2L}f\in t_2^{p(\cdot )}(\mathbb{R}^n)\cap t_2^2(\mathbb{R}^n)$. According to Lemma \ref{LemmaT2}, for every $j\in \mathbb{N}$, there exist $\lambda _j>0$ and a $(t_2^{p(\cdot )},2)$-atom $A_j$ associated with the ball $B_j$ such that $t^2Le^{-t^2L}f(x)=\sum_{j\in \mathbb{N}}\lambda _jA_j(x,t)$, in $t_2^2(\mathbb{R}^n)$ and in $t_2^{p(\cdot )}(\mathbb{R}^n)$, and
\[\left\|\left(\sum_{j\in \mathbb{N}}\left(\frac{\lambda _j\chi _{B_j}}{\|\chi _{B_j}\|_{p(\cdot )}}\right)^{\underline{p}}\right)^{1/\underline{p}}\right\|_{p(\cdot )}\leq C\|S_L^{\textrm{loc}}f\|_{p(\cdot )}.
\]

Then, by Lemma \ref{pin} we have that
\[f_1=\pi _N(t^2Le^{-t^2L}f)=\sum_{j\in \mathbb{N}}\lambda _j\pi_N(A_j)\quad \mbox{ in } L^2(\mathbb{R}^n),
\]
and, for certain $C_0>0$, for every $j\in \mathbb{N}$, $C_0\pi _N(A_j)$ is a $(p(\cdot ), 2,M,\varepsilon )_{L,\rm{loc}}$-molecule associated with $B_j$.

It follows that $f_1\in h_{L,\textrm{loc}, M,\varepsilon}^{p(\cdot)}(\mathbb{R}^n)$ and
\begin{equation}\label{f1}
\|f_1\|_{h_{L,\textrm{loc}, M,\varepsilon}^{p(\cdot)}(\mathbb{R}^n)}\leq C\|S_L^{\textrm{loc}}(f)\|_{p(\cdot )}.
\end{equation}

By taking into account \eqref{1}{, \eqref{2} and \eqref{f1} we obtain that $L^2(\mathbb{R}^n)\cap h_L^{p(\cdot)}(\mathbb{R}^n)\subset L^2(\mathbb{R}^n)\cap h_{L,\textrm{loc}, M,\epsilon}^{p(\cdot)}(\mathbb{R}^n)$ algebraic and topologically.

By using a density argument we conclude that $h_L^{p(\cdot )}(\mathbb{R}^n)\subset h_{L,\textrm{mol}, M,\varepsilon}^{p(\cdot )}(\mathbb{R}^n)$ algebraic and topologically.

\subsection{Proof of Corollary \ref{Lp=Hp}}
Suppose that $f\in L^2(\mathbb{R}^n)\cap h_L^{p(\cdot )}(\mathbb{R}^n)$, $\varepsilon >n(1/p^--1/p^+)$ and $M\in \mathbb{N}$. According to Theorem \ref{molchar}(ii) there exists $C>0$ such that, for every $j\in \mathbb{N}$, there exist $\lambda _j>0$ and a $(p(\cdot ),2,M,\varepsilon )_{L,\rm{loc}}$-molecule $m_j$ associated with the ball $B_j$ such that $f=\sum_{j\in \mathbb{N}}\lambda _jm_j$ in $L^2(\mathbb{R}^n)$, and satisfying that
\[\left\|\left(\sum_{j\in \mathbb{N}}\left(\frac{\lambda _j\chi _{B_j}}{\|\chi _{B_j}\|_{p(\cdot )}}\right)^{\underline{p}}\right)^{1/\underline{p}}\right\|_{p(\cdot )}\leq C\|f\|_{h_L^{p(\cdot )}(\mathbb{R}^n)}.
\]
For every $j\in \mathbb{N}$ and $i\in \mathbb{N}_0$,we have that
\[\|m_j\|_{L^2(S_i(B_j))}\leq 2^{-i\varepsilon}|2^iB_j|^{1/2}\|\chi _{2^iB_j}\|_{p(\cdot )}^{-1}.
\]
According to \cite[Proposition 2.11]{ZhSaYa} we get, for every $j_1,j_2\in \mathbb{N}$, $j_1<j_2$,
\begin{align*}
\left\|\sum_{j=j_1}^{j_2}\lambda _jm_j\chi _{S_i(B_j)}\right\|_{p(\cdot )}&\leq \left\|\left(\sum_{j=j_1}^{j_2}|\lambda _jm_j\chi _{S_i(B_j)}|^{\underline{p}}\right)^{1/\underline{p}}\right\|_{p(\cdot )}\\
&\leq C2^{-i(\varepsilon - n(1/w-1/p^+))}\left\|\left(\sum_{j=j_1}^{j_2}\left(\frac{\lambda _j\chi _{B_j}}{\|\chi _{B_j}\|_{p(\cdot )}}\right)^{\underline{p}}\right)^{1/\underline{p}}\right\|_{p(\cdot )},
\end{align*}
for $w\in (0,p^-)$. Then, since the series 
\[\sum_{j\in \mathbb{N}}\left(\frac{\lambda _j\chi _{B_j}}{\|\chi _{B_j}\|_{p(\cdot )}}\right)^{\underline{p}}
\]
converges in $L^{p(\cdot )/\underline{p}}(\mathbb{R}^n)$, if $\delta >0$ there exists $j_0\in \mathbb{N}$ such that, for every $j_1,j_2\in \mathbb{N}$, with $j_2>j_1\geq j_0$,
\begin{align*}
\left\|\sum_{j=j_1}^{j_2} \lambda _jm_j\right\|_{p(\cdot)}&\leq \left\|\sum_{i\in \mathbb{N}}\sum_{j=j_1}^{j_2}\lambda _jm_j\chi _{S_i(B_j)}\right\|_{p(\cdot )}\\
&\leq \left(\sum_{i\in \mathbb{N}}\Big\|\sum_{j=j_1}^{j_2}\lambda _jm_j\chi _{S_i(B_j)}\Big\|_{p(\cdot )}^{\underline{p}}\right)^{1/\underline{p}}\\
&\leq C\delta\left(\sum_{i\in \mathbb{N}}2^{-i\underline{p} (\varepsilon -n(1/w-1/p^+))}\right)^{1/\underline{p}}\leq C\delta.
\end{align*}
Since $L^{p(\cdot )}(\mathbb{R}^n)$ is complete, the series $\sum_{j\in \mathbb{N}}\lambda _jm_j$ converges in $L^{p(\cdot )}(\mathbb{R}^n)$.

We write $g=\sum_{j\in \mathbb{N}}\lambda _jm_j$ in $L^{p(\cdot )}(\mathbb{R}^n)$. Since $f=\sum_{j\in \mathbb{N}}\lambda _jm_j$ in $L^2(\mathbb{R}^n)$, by using \cite[Lemma 3.2.10]{DHHR} we conclude that $f=g$. Hence $f\in L^{p(\cdot )}(\mathbb{R}^n)$ and $\|f\|_{p(\cdot )}\leq C\|f\|_{h_L^{p(\cdot )}(\mathbb{R}^n)}$.

Suppose that $p^->1$. According to \cite[Lemma 3.4]{LiZhZh} (see also \cite[Lemma 5.1]{GoYa}) the area square integral $S_L$ defines a bounded operator from $L^2(\mathbb{R}^n,w)$ into itself, for every $w\in A_2(\mathbb{R}^n)$. Here $A_2(\mathbb{R}^n)$ denotes the Muckenhoupt class of weights and $L^2(\mathbb{R}^n,w)$ represents the weighted $L^2$ space. By using the generalization of Rubio de Francia's extrapolation theorem established in \cite[Corollary 1.10]{CuMP}, we deduce that there exists $C>0$ such that
\[\|S_Lf\|_{p(\cdot )}\leq C\|f\|_{p(\cdot )},\quad f\in L^2(\mathbb{R}^n)\cap L^{p(\cdot )}(\mathbb{R}^n).
\]
Hence, $L^2(\mathbb{R}^n)\cap L^{p(\cdot )}(\mathbb{R}^n)$ is continuously contained in $H_L^{p(\cdot )}(\mathbb{R}^n)$.

Density arguments allow us to show that $L^{p(\cdot )}(\mathbb{R}^n)=h_L^{p(\cdot )}(\mathbb{R}^n)=H_L^{p(\cdot )}(\mathbb{R}^n)$.

\section{Proof of Theorem \ref{local=global}}

We have to prove that $h^{p(\cdot)}_L(\mathbb R^n)=H^{p(\cdot)}_L(\mathbb R^n)$. Suppose that $f\in h^{p(\cdot)}_L(\mathbb R^n)\cap L^2(\mathbb R^n)$ and consider $M\in \mathbb{N}$, such that $2M>n(2/p^--1/2-1/p^+)$ and $\varepsilon >n(1/p^--1/p^+)$. According to Theorem \ref{molchar} there exist, for every $j\in{\mathbb N}$, $\lambda_j>0$ and a $(p(\cdot),2,M,\varepsilon)_{L,\rm{loc}}$-molecule $m_j$ associated with the ball $B_j$ such that $f=\sum_{j\in \mathbb{N}}\lambda_jm_j$ in $L^2(\mathbb R^n)$ and in $h^{p(\cdot)}_L(\mathbb R^n)$, and

\[\left\|\left(\sum_{j\in \mathbb{N}} \left(\lambda_j\frac{\chi_{B_j}}{\|\chi_{B_j}\|_{p(\cdot)}}\right)^{\underline{p}}\right)^{1/\underline{p}}\right\|_{p(\cdot)}\sim \;\;\|f\|_{h^{p(\cdot)}_L(\mathbb R^n)}.\]
We are going to see that $f\in H^{p(\cdot)}_L(\mathbb R^n)$ and that

\[\|f\|_{H^{p(\cdot)}_L(\mathbb R^n)}\leq C\left\|\left(\sum_{j\in \mathbb{N}} \left(\lambda_j\frac{\chi_{B_j}}{\|\chi_{B_j}\|_{p(\cdot)}}\right)^{\underline{p}}\right)^{1/\underline{p}}\right\|_{p(\cdot)}.\]
We have that
\[S_L(f)\leq S_L^{\rm{loc}}(f)+S_L^{\infty}(f),
\]
where
\[S_L^{\infty}(f)(x)=\left(\int_1^\infty\int_{B(x,t)}|t^2Le^{-t^2L}(f)(y)|^2\frac{dydt}{t^{n+1}}\right)^{1/2}, \quad x\in \mathbb R^n.
\]
Since $f\in L^2(\mathbb{R}^n)\cap h^{p(\cdot)}_L(\mathbb R^n)$, $S_L^{\rm{loc}}(f)\in L^{p(\cdot)}(\mathbb {R}^n)$ and 
\[\|S_L^{\rm{loc}}(f)\|_{p(\cdot)}\leq \|f\|_{h_L^{p(\cdot )}(\mathbb{R}^n)}\leq C\left\|\left(\sum_{j\in \mathbb{N}} \left(\lambda_j\frac{\chi_{B_j}}{\|\chi_{B_j}\|_{p(\cdot)}}\right)^{\underline{p}}\right)^{1/\underline{p}}\right\|_{p(\cdot)}.\]
By Lemma \ref{tecnico}, in order to prove that $S_L^{\infty}(f)\in L^{p(\cdot)}(\mathbb R^n)$ and that
\[
\|S_L^\infty(f)\|_{p(\cdot)}\leq C\left\|\left(\sum_{j\in \mathbb{N}} \left(\lambda_j\frac{\chi_{B_j}}{\|\chi_{B_j}\|_{p(\cdot)}}\right)^{\underline{p}}\right)^{1/\underline{p}}\right\|_{p(\cdot)},
\]
it is sufficient to see that there exist $C>0$ and $\eta >n(1/p^--1/p^+)$ such that 
\begin{equation}\label{C1}
\|S_L^\infty(m_j)\|_{L^2(S_i(B_j))}\leq C2^{-i\eta}|2^iB_j|\|\chi_{2^iB_j}\|_{p(\cdot)}^{-1}, \quad i \in \mathbb{N}_0,\;j\in\mathbb N.
\end{equation}
Assume that $m$ is a $(p(\cdot),2,M,\varepsilon)_{L,{\rm loc}}$-molecule associated with the ball $B=B(x_B,r_B)$, being $x_B\in \mathbb R^n$ and $r_B>0$.

Suppose first that $m$ is of $(I)$-type. Then $r_B\geq 1$. We observe that, since $\sigma (L)>0$, for every $t>0$, the kernel $k_t$ of the integral operator $Le^{-t^2L}$ satisfies that
\[
|k_t(x,y)|\leq C\frac{e^{-c(t^2+|x-y|^2/t^2)}}{t^{n+2}},\;\;\;x,y\in\mathbb R^ ,
\]
(see \cite[Proposition 2.2]{CS}). Consider $i\in\mathbb N$, $i\geq 3$. As in the proof of Theorem \ref{molchar} $(i)$ we decompose $m=m_1+m_2$, where $m_1=m\chi _{2^{i+2}B\setminus 2^{i-3}B}$ and write \begin{align*}
 \|S_L^{\infty}(m)\|_{L^2(S_i(B))}^2 & \leq C\Big(\|S_L(m_1)\|_{L^2(S_i(B))}^2\\
 &\quad + \int_{S_i(B)}\int_1^\infty\int_{B(x,t)}\Big(\int_{\mathbb{R}^n}t^2|k_t(y,z)||m_2(z)|dz\Big)^2dydtdx\Big)\\
&=\mathbb{I}_1+\mathbb{I}_2.
\end{align*}

As in \eqref{ref1}, the $L^2$-boundedness of $S_L$ gives
$$
\mathbb{I}_1\leq C2^{-2i\varepsilon}|2^iB|\|\chi _{2^iB}\|_{p(\cdot )}^{-2}
$$

On the other hand
\begin{align}\label{acotI2}
 \mathbb{I}_2 &\leq C\int_{S_i(B)}\int_1^\infty\int_{B(x,t)}\left(\int_{(2^{i+2}B\setminus 2^{i-3}B)^c}\frac{e^{-c(t^2+|x-y|^2/t^2)}}{t^{n}}|m(z)|dz\right)^2 dydtdx \nonumber\\
 &\leq C\|m\|_2^2\int_{S_i(B)}\left(\int_1^{2^{i-2}r_B}+\int_{2^{i-2}r_B}^\infty\right)\nonumber\\
 &\quad \times \int_{B(x,t)}\int_{(2^{i+2}B\setminus 2^{i-3}B)^c}\frac{e^{-c(t^2+|x-y|^2/t^2)}}{t^{2n}}dzdydtdx\nonumber\\
& =\mathbb{I}_{2,1}+\mathbb{I}_{2,2}.
\end{align}

By proceeding as in \eqref{ref2} and using \eqref{m} we have, for every $\alpha >0$, 
\begin{align*}
 \mathbb{I}_{2,1}\leq & C\|m\|^2_2\int_{S_i(B)}\int_0^{2^{i-2}r_B}e^{-c(t^2+(2^i r_B)^2/t^2)}dtdx\\
& \leq C\|m\|^2_2|2^iB|\int_0^{2^{i-2}r_B}\frac{1}{t^\alpha}\left(\frac{t}{2^i r_B}\right)^\alpha d t\\
& \leq C |B|\|\chi_{B}\|_{p(\cdot)}^{-2}|2^i B|(2^i r_B)^{1-\alpha}.
\end{align*}

Since $r_B\geq 1$, by considering \cite[Lemma 3.8]{YaZhZh} and choosing $\alpha$ large enough we obtain
$$
\mathbb{I}_{2,1}\leq C|2^iB|\|\chi _{2^iB}\|_{p(\cdot )}^{-2}2^{-i(\alpha -1-2n/p^-)}r_B^{n+1-\alpha }\leq C2^{-2i\varepsilon}|2^iB|\|\chi _{2^iB}\|_{p(\cdot )}^{-2}.
$$

Also, in analogous way we get
\begin{align*}
 \mathbb{I}_{2,2}& \leq C \|m\|_2^2\int_{S_i(B)}\int_{2^{i-2}r_B}^\infty e^{-ct^2} d tdx\leq C|B|\|\chi_{B}\|_{p(\cdot)}^{-2}|2^i B|e^{-c(2^i r_B)^2}\\
 &\leq C 2^{-2i\varepsilon} |2^i B|\|\chi_{2^i B}\|_{p(\cdot)}^{-2}.
\end{align*}

By putting together the above estimates we conclude that, for $i\geq 3$,
\[\|S_L^\infty(m)\|_{L^2(S_i(B))}\leq C2^{-i\varepsilon}|2^iB|^{1/2}\|\chi_{2^iB}\|_{p(\cdot)}^{-1}.\]

By taking into account the $L^2$-boundedness of $S_L$ and \eqref{m} this estimation is also true for $i=0,1,2$, and then \eqref{C1} is established for molecules of (I)-type. 

We now consider that $m$ is a molecule of (II)-type. Then $r_B\in (0,1)$ and there exists $b\in L^2(\mathbb R^n)$ such that $m=L^Mb$ and, for every $k\in\{0,...,M\}$,
\begin{equation}\label{acotb}
\|L^kb\|_{L^2(S_i(B))}\leq 2^{-i\varepsilon}r_B^{2(M-k)}|2^iB|^{1/2}\|\chi_{2^iB}\|_{p(\cdot)}^{-1},\;\;\;i\in\mathbb N _0.
\end{equation}

We have that
$$
\|S_L^\infty (m)\|_{L^2(S_i(B))}^2=\int_{S_i(B)}\int_1^\infty\int_{B(x,t)}|(t^2L)^{M+1}e^{-t^2L}(b)(y)|^2\frac{d y d t}{t^{4M+n+1}}d x.
$$

Let $i\in\mathbb N$, $i\geq 3$. We write $b=b_1+b_2$, where $b_1=b\chi _{2^{i+2}B\setminus 2^{i-3}B}$. 

We note that, for every $t>0$, the kernel $K_t$ of the operator $L^{M+1}e^{-t^2L}$ satisfies
\[|K_t(x,y)|\leq C\frac{e^{-c(t^2+|x-y|^2/t^2)}}{t^{2M+n+2}},\;\;\;x,y\in\mathbb R^n.\]

Then, we get
\begin{align*}
\|S_L^{\infty}(m)\|_{L^2(S_i(B))}^2& \leq C\left( \int_{S_i(B)}\int_1^\infty\int_{B(x,t)}|(t^2L)^{M+1}e^{-t^2L}(b_1)(y)|^2\frac{d y d t}{t^{4M+n+1}}d x\right. \\
&\hspace{-2cm} \left. \quad + \int_{S_i(B)}\int_1^\infty\int_{B(x,t)}\Big(\int_{\mathbb{R}^n}t^{2M+2}|K_t(y,z)||b_2(z)|dz\Big)^2\frac{dydt}{t^{4M+n+1}}dx\right) \\
&\hspace{-2cm}=\mathbb{J}_1+\mathbb{J}_2.
\end{align*}

By using \eqref{acotb} we obtain
\begin{align*}
 \mathbb{J}_1&\leq  C\int_{\mathbb{R}^n}\int_1^\infty |(t^2L)^{M+1}e^{-t^2L}(b_1)(y)|^2\frac{dt}{t}dy \\
&\leq C\|G_{L,M}(b_1)\|_2^2\leq C\|b_1\|_2^2 \leq C 2^{-2i\varepsilon}|2^iB|\|\chi_{2^iB}\|_{p(\cdot)}^{-2}.
\end{align*}

Here the Littlewood-Paley function $G_{L,N}$, $N\in \mathbb{N}$, is defined by
\[G_{L,N}(g)(x)=\left(\int_0^\infty|(t^2L)^{N+1}e^{-t^2L}(g)(x)|^2\frac{d t}{t}\right)^{1/2},\]
and it is a bounded (sublinear) operator from $ L^2(\mathbb R^n)$ into itself (\cite[Lemma 4.1]{LYY}).

On the other hand,
\begin{align*}
\mathbb{J}_2 & \leq C\int_{S_i(B)}\int_1^\infty \int_{B(x,t)}\Big(\int_{(2^{i+2}B\setminus 2^{i-1}B)^c}\frac{e^{-c(t^2+|x-y|^2/t^2)}}{t^{n}}b(z)dz\Big)^2 \frac{dydt}{t^{4M+n+1}}dx \\
&\leq C\|b\|_2^2\int_{S_i(B)}\int_1^\infty \int_{B(x,t)}\int_{(2^{i+2}B\setminus 2^{i-1}B)^c}\frac{e^{-c(t^2+|x-y|^2/t^2)}}{t^{2n}}dz \frac{dydt}{t^{4M+n+1}}dx\\
& =\mathbb{J}_{2,1}+\mathbb{J}_{2,2},
\end{align*}
where 
$$
\mathbb{J}_{2,1}=C\|b\|_2^2\int_{S_i(B)}\int_1^{2^{i-2}r_B} \int_{B(x,t)}\int_{(2^{i+2}B\setminus 2^{i-1}B)^c}\frac{e^{-c(t^2+|x-y|^2/t^2)}}{t^{2n}}dz \frac{dydt}{t^{4M+n+1}}dx
$$
provided that $2^{i-2}r_B>1$. If $2^{i-2}r_B\leq 1$, we define $\mathbb{J}_{2,1}=0$.

By proceeding in the same way as \eqref{acotI2} and taking into account \eqref{B6} and \cite[Lemma 3.8]{YaZhZh} we get, for $\alpha >1$,
\begin{align*}
 \mathbb{J}_{2}& \leq C\|b\|^2_2|2^iB|((2^ir_B)^{1-\alpha}+e^{-c(2^ir_B)^2})\leq C\|b\|_2^2|2^iB|(2^ir_B)^{1-\alpha }\\
 &\leq C |2^iB|\|\chi_{2^iB}\|_{p(\cdot)}^{-2}2^{2in/p^-}r_B^{4M+n}(2^i r_B)^{1-\alpha}  \\
& \leq C |2^i B|\|\chi_{2^iB}\|_{p(\cdot)}^{-2}2^{-i(\alpha -1-2n/p^-)}r_B^{4M+n+1-\alpha}.
\end{align*}

By choosing $\alpha =4M+n+1$, it follows that
$$
\mathbb{J}_2\leq C2^{-i(4M+n-2n/p^-)}|2^i B|\|\chi_{2^iB}\|_{p(\cdot)}^{-2},
$$
and since $2M>n(2/p^--1/2-1/p^+)$, there exists $\eta >n(1/p^--1/p^+)$ such that
\[
\mathbb{J}_2\leq  C2^{-2i\eta}|2^i B|\|\chi_{2^iB}\|_{p(\cdot)}^{-2}.\]

By combining the above estimates it follows that 
\[\|S_L^\infty(m)\|_{L^2(S_i(B))}\leq C2^{-i\eta}|2^i B|^{1/2}\|\chi_{2^iB}\|_{p(\cdot)}^{-1}, \quad i\in \mathbb{N}, i\geq 3,
\]
with $\eta >n(1/p^--1/p^+)$. When $i=0,1,2$ as in the previous case we get also this estimation.

Thus we have proved that \eqref{C1} holds for local molecules of (II)-type. 

The proof can be completed by using a density argument.

\section{Proof of Theorem \ref{localglobalL+I}}

Firstly we prove that $H_{L+I}^{p(\cdot)}(\mathbb R^n)\subset h_L^{p(\cdot)}(\mathbb R^n)$. Let $f\in H_{L+I}^{p(\cdot)}(\mathbb R^n)\cap L^2(\mathbb R^n)$. 
Consider $M\in \mathbb{N}$ such that $M>n(2/p^--1/2-1/p^+)+1/2$ and $\varepsilon >n(1/p^--1/p^+)$. According to \cite[Theorem 3.15]{YaZhZh}, for every $j\in\mathbb N$, there exist $\lambda_j >0$ and a $(p(\cdot),2,M,\varepsilon)_{L+I}$-molecule $m_j$ associated to the ball $B_j=B(x_{B_j},r_{B_j})$ with $x_{B_j}\in\mathbb R^n$ and $r_{B_j}>0$, such that $f=\sum_{j\in \mathbb{N}}\lambda_jm_j$ in $L^2(\mathbb R^n)$ and in $H_{L+I}^{p(\cdot)}(\mathbb R^n)$, and 
\[\left\|\left(\sum_{j\in \mathbb{N}} \left(\lambda_j\frac{\chi_{B_j}}{\|\chi_{B_j}\|_{p(\cdot)}}\right)^{\underline{p}}\right)^{1/\underline{p}}\right\|_{p(\cdot)}\sim \;\|f\|_{H^{p(\cdot)}_{L+I}(\mathbb R^n)}.\]

Our objective is to see that $S_L^{\rm{loc}}(f)\in L^{p(\cdot)}(\mathbb R^n)$, $e^{-L}f\in H^{p(\cdot )}_I(\mathbb{R}^n)$ and that
\[\|S_L^{\rm{loc}}(f)\|_{p(\cdot)}+\|S_I(e^{-L}f)\|_{p(\cdot )}\leq C\left\|\left(\sum_{j\in \mathbb{N}} \left(\lambda_j\frac{\chi_{B_j}}{\|\chi_{B_j}\|_{p(\cdot)}}\right)^{\underline{p}}\right)^{1/\underline{p}}\right\|_{p(\cdot)}.
\]

By considering Lemma \ref{tecnico} it is sufficient to show that there exist $C>0$ and $\eta >n(1/p^--1/p^+)$ such that, for every $j\in\mathbb N$,
\begin{align}\label{D1}
\|S_L^{\rm{loc}}(m_j)\|_{L^2(S_i(B_j))}+\|S_I(e^{-L}f)(m_j)\|_{L^2(S_i(B_j))}&\nonumber\\
&\hspace{-2cm}\leq C2^{-i\eta}|2^i B_j|^{1/2}\|\chi_{2^iB_j}\|_{p(\cdot)}^{-1}, \;\;\;i\in\mathbb N _0.
\end{align}

Let $m$ be a $(p(\cdot ),2,M,\varepsilon)_{L+I}$-molecule associated to $B=B(x_B,r_B)$, with $x_B\in \mathbb{R}^n$ and $r_B>0$. Then, there exists $b\in L^2(\mathbb{R}^n)$ such that $m=(L+I)^Mb$ and, for every $k\in \{0,...,M\}$,
\begin{equation}\label{Lkb2}
\|(L+I)^k(b)\|_{L^2(S_i(B))}\leq 2^{-i\varepsilon}r_B^{2(M-k)}|2^iB|^{1/2}\|\chi _{2^iB}\|_{p(\cdot )}^{-1},\quad i\in \mathbb{N}_0.
\end{equation}

We note that if $r_B\geq 1$, then $m$ is also a $(p(\cdot),2,M,\varepsilon)_{L,{\rm loc}}$-molecule of (I)-type. Then, according to (\ref{objective}), (\ref{D1}) holds for $m$, provided that $r_B\geq 1$.

Assume now that $r_{B}\in (0,1)$. We can write
\begin{equation}\label{D2}
Le^{-t^2L}-(L+I)e^{-t^2(L+I)}=(1-e^{-t^2})Le^{-t^2L}-e^{-t^2}e^{-t^2L},\;\;\;t>0 ,
\end{equation}
so we have that
\[S_L^{\rm{loc}}(g)\leq S_{L+I}^{\rm{loc}}(g)+{\mathcal H}(g)+{\mathcal Q}(g),\]
where
\[{\mathcal H}(g)(x)=\left(\int_0^1\int_{B(x,t)}|t^2(1-e^{-t^2})Le^{-t^2L}(g)(y)|^2 \frac{d y d t}{t^{n+1}}\right)^{1/2},\;\;x\in\mathbb R^n,\]
and
\[{\mathcal Q}(g)(x)=\left(\int_0^1\int_{B(x,t)}|t^2e^{-t^2}e^{-t^2L}(g)(y)|^2\frac{d y d t}{t^{n+1}}\right)^{1/2} ,\;\;x\in\mathbb R^n,\]
for every $g\in L^2(\mathbb R^n)$.

According to \cite[(3.13)]{YaZhZh} we can find $C>0$ and $\eta >n(1/p^--1/p^+)$ such that 
\[\|S_{L+I}^{\rm loc}(m)\|_{L^2(S_i(B))}\leq \|S_{L+I}(m)\|_{L^2(S_i(B))}\leq  C2^{-i\eta} |2^i B|^{1/2}\|\chi_{2^iB}\|_{p(\cdot)}^{-1}\;\;\;i\in\mathbb N _0.\]

We analyze the operator ${\mathcal H}$. We can write
\begin{align*}
\|{\mathcal H}(m)\|_{L^2(S_i(B))}^2& \\
&\hspace{-2cm}\leq C\int_{S_i(B)}\left(\int_0^{2^{i/2}r_B}+\int_{2^{i/2}r_B}^1\right)\int_{B(x,t)} |t^4Le^{-t^2L}(m)(y)|^2 \frac{d y d t}{t^{n+1}}d x \\
&\hspace{-2cm}=I_1+I_2,
\end{align*}
provided that $2^{i/2}r_B < 1$, and, in other case $I_2=0$.

Let $i\in \mathbb{N}$, $i\geq 3$. We observe that if $x\in S_i(B)$, $t\in (0,2^{i/2}r_B)$, $y\in B(x,t)$ and $z\in (2^{i+2}B\setminus 2^{i-3}B)^c$, then $|y-z|\geq c2^{i}r_B$. Then, by decomposing $m=m_1+m_2$, with $m_1=m\chi _{2^{i+2}B\setminus 2^{i-3}B}$, we have that 
\begin{align*}
 I_1&\leq C\left( \int_{S_i(B)}\int_0^{2^{i/2}r_B}\int_{B(x,t)}|t^4Le^{-t^2L}(m_1)|^2 \frac{dydt}{t^{n+1}}dx\right.\\
&\quad \left. +\int_{S_i(B)}\int_0^{2^{i/2}r_B}\int_{B(x,t)} \Big(t^2\int_{(2^{i+2}B\setminus 2^{i-3}B)^c}\frac{e^{-c|z-y|^2/t^2}}{t^n}|m(z)|dz\Big)^2 \frac{dydt}{t^{n+1}}dx\right) \\
& \leq C\left(\|S_L(m_1)\|_2^2+\right.\\
&\quad \left. +\|m\|_2^2\int_{S_i(B)}\int_0^{2^{i/2}r_B}t^{3-n}e^{-c(2^i r_B)^2/t^2}\int_{B(x,t)}\int_{\mathbb{R}^n}\frac{e^{-c|z-y|^2/t^2}}{t^{2n}}dzdydtdx \right)\\
& \leq C\left(\|m_1\|_2^2+\|m\|_2^2|2^iB|\int_0^{2^{i/2}r_B}t^{3-n}e^{-c(2^i r_B)^2/t^2}dt\right) \\
& \leq C\left(2^{-2i\varepsilon}|2^iB|\|\chi_{2^iB}\|_{p(\cdot)}^{-2}\right.\\
&\quad \left. +|2^iB|\|\chi_{2^iB}\|_{p(\cdot)}^{-2}2^{-2in(1/2-1/p^-)}r_B^n\int_0^{2^{i/2}r_B}t^{3-n}\left(\frac{t}{2^ir_B}\right)^{\gamma}dt\right) \\
&  \leq C |2^iB|\|\chi_{2^iB}\|_{p(\cdot)}^{-2}\left(2^{-2i\varepsilon}+2^{-2in(1/2-1/p^-)}r_B^n\frac{(2^{i/2}r_B)^{\gamma-n+4}}{(2^ir_B)^{\gamma}}\right) \\
& \leq C |2^iB|\|\chi_{2^iB}\|_{p(\cdot)}^{-2}\left(2^{-2i\varepsilon}+2^{-2i(n(1/2-1/p^-)+\gamma /4-1+n/4)} r_B^4\right)\\
&\leq C2^{-2i\varepsilon}|2^iB|\|\chi_{2^iB}\|_{p(\cdot)}^{-2},
\end{align*}
where $\gamma >n-4$ and it is chosen large enough in order that $n(1/2-1/p^-)+\gamma /4-1+n/4>\varepsilon$. 

On the other hand, by taking into account the Gaussian estimates \eqref{gaussianderivatives} it follows that the family of operators $\{e^{-t^2L/2}(L+I)^Mt^{2M}e^{-t^2L/2}\}_{t\in (0,1)}$ is bounded in the space ${\mathcal L}(L^2(\mathbb R^n))$ that consists of those linear bounded operators from $L^2(\mathbb R^n)$ into itself and that is endowed with the usual norm. Then, we have that
\begin{align*}
I_2 & =\int_{S_i(B)}\int_{2^{i/2}r_B}^1\int_{B(x,t)} |t^4Le^{-t^2L}((L+I)^Mb)(y)|^2 \frac{d y d t}{t^{n+1}}dx \\
& \leq C\int_{S_i(B)}\int_{2^{i/2}r_B}^1\int_{B(x,t)} |t^2Le^{-t^2L/2}(L+I)^Mt^{2M}e^{-t^2L/2}(b)(y)|^2 \frac{d y d t}{t^{n+4M-3}}d x  \\
& \leq C\int_{2^{i/2}r_B}^1\int_{\mathbb R^n} |t^2Le^{-t^2L/2}(L+I)^Mt^{2M}e^{-t^2L/2}(b)(y)|^2 \frac{d y d t}{t^{4M-3}}  \\
&  \leq C \frac{\|b\|_2^2}{(2^{i/2}r_B)^{4M-2}}\leq C|2^iB|\|\chi_{2^iB}\|_{p(\cdot)}^{-2}r_B^{4M}\frac{2^{-2in(1/2-1/p^-)}}{(2^{i/2}r_B)^{4M-2}} \\
& \leq C |2^iB|\|\chi_{2^iB}\|_{p(\cdot)}^{-2}2^{-2i(n(1/2-1/p^-)+M-1/2)}.
\end{align*}
By considering the conditions on $M$ and $\varepsilon$ we conclude that 
$$
\|\mathcal{H}(m)\|_{L^2(S_i(B))}\leq C2^{-i\eta}|2^iB|^{1/2}\|\chi_{2^iB}\|_{p(\cdot)}^{-1},
$$
for some $\eta >n(1/p^--1/p^+)$. Also, the $L^2$-boundedness of $S_L$ and \eqref{m} give this estimation also for $i=0,1,2$, because, in these cases 
$$
\|{\mathcal H}(m)\|_{L^2(S_i(B))}\leq \|{\mathcal H}(m)\|_2\leq C\|S_L(m)\|_2\leq C\|m\|_2\leq C 2^{-i\eta}|2^iB|^{1/2}\|\chi_{2^iB}\|_{p(\cdot)}^{-1},
$$
with $\eta >0$.

In order to study the operator ${\mathcal Q}$ we proceed in a similar way. We observe that the operator ${\mathcal Q}$ is bounded from $L^2(\mathbb R^n)$ into itself. Indeed, for every $g\in L^2(\mathbb R^n)$, we have that
\begin{align*}
\|{\mathcal Q}(g)\|_2= & \int_{\mathbb R^n}\int_0^1\int_{B(x,t)} |t^2e^{-t^2}e^{-t^2L}(g)(y)|^2 \frac{d y d t}{t^{n+1}}dx \\
& = \int_{\mathbb R^n}\int_0^1|t^2e^{-t^2}e^{-t^2L}(g)(y)|^2 \frac{d y d t}{t}\\
& \leq \int_0^1t^3e^{-t^2}\|e^{-t^2L}(g)\|_2^2d t\leq C\|g\|_2^2. 
\end{align*}

This fact, the ${\mathcal L}(L^2(\mathbb R^n))$- boundedness of the family of operators $\{e^{-t^2L/2}(L+I)^Mt^{2M}e^{-t^2L/2}\}_{t\in (0,1)}$ and the upper Gaussian estimate for the heat kernel associated with $L$ allow us to show that
$$
\|{\mathcal Q}(m)\|_{L^2(S_i(B))}\leq C 2^{-i\eta}|2^iB|^{1/2}\|\chi_{2^iB}\|_{p(\cdot)}^{-1},\;\;\;i\in\mathbb N _0,
$$
for some $\eta >n(1/p^--1/p^+)$.

By putting together the above estimates we conclude that
\[\|S_L^{\rm loc}(m)\|_{L^2(S_i(B))}\leq C 2^{-i\eta}|2^iB|^{1/2}\|\chi_{2^iB}\|_{p(\cdot)}^{-1},\]
with $\eta >n(1/p^--1/p^+)$.

On the other hand, we observe that the operator $(L+I)^Me^{-L}$ is bounded from $L^2(\mathbb R^n)$ into itself. Moreover the kernel, $K$, of this integral operator satisfies that $|K(x,y)|\leq Ce^{-c|x-y|^2}$, $x,y\in\mathbb R^n$. By proceeding as in the part of  proof of Theorem \ref{molchar} after (\ref{B6}) we deduce that, for certain $\eta >n(1/p^--1/p^+)$,
$$
\|S_I(e^{-L}m)\|_{L^2(S_i(B))}\leq C 2^{-i\eta}|2^iB|^{1/2}\|\chi_{2^iB}\|_{p(\cdot)}^{-1}, \quad i\in \mathbb{N}_0,
$$
and so, \eqref{D1} is established. A density argument allow us to conclude that $H_{L+I}^{p(\cdot)}(\mathbb R^n)\subset h_L^{p(\cdot)}(\mathbb R^n)$ algebraic and topologically.

We now prove that $h_L^{p(\cdot)}(\mathbb{R}^n) \subset H_{L+I}^{p(\cdot )}(\mathbb{R}^n)$. Let us consider $M\in \mathbb{N}$ such that $2(M-1)>n(2/p^--1/2-1/p^+)$ and $\varepsilon >n(1/p^--1/p^+)$. Let $f\in h_L^{p(\cdot)}(\mathbb{R}^n)\cap L^2(\mathbb{R}^n)$. According to Theorem \ref{molchar}, for every $j\in \mathbb{N}$, there exist $\lambda _j>0$ and a $(p(\cdot ),2,M,\varepsilon )_{L,\rm{loc}}$-molecule $m_j$ associated to the ball $B_j$ such that $f=\sum_{j\in \mathbb{N}}\lambda _jm_j$, in $L^2(\mathbb{R}^n)$ and in $h_L^{p(\cdot )}(\mathbb{R}^n)$, and
\[\left\|\left(\sum_{j\in \mathbb{N}}\left(\lambda_j\frac{\chi_{B_j}}{\|\chi_{B_j}\|_{p(\cdot)}}\right)^{\underline{p}}\right)^{1/\underline{p}}\right\|_{p(\cdot)}\leq C\|f\|_{h_L^{p(\cdot )}(\mathbb{R}^n)}.
\]

Our objective is to see that $f\in H_{L+I}^{p(\cdot )}(\mathbb{R}^n)$ and that 
\[\|S_{L+I}(f)\|_{p(\cdot )}\leq C\left\|\left(\sum_{j\in \mathbb{N}}\left(\lambda_j\frac{\chi_{B_j}}{\|\chi_{B_j}\|_{p(\cdot)}}\right)^{\underline{p}}\right)^{1/\underline{p}}\right\|_{p(\cdot)}.
\]

According to Lemma \ref{tecnico} these facts will be proved when we show that, for every $j\in \mathbb{N}$,
\[\|S_{L+I}(m_j)\|_{L^2(S_i(B_j))}\leq C2^{-i\eta}|2^iB_j|^{1/2}\|2^iB_j\|_{p(\cdot )}^{-1},\quad i\in \mathbb{N}_0,
\]
for some $\eta >n(1/p^--1/p^+)$.

Assume that $m$ is a $(p(\cdot ),2,M,\varepsilon )_{L,\rm{loc}}$-molecule associated to the ball $B=B(x_B,r_B)$ with $x_B\in \mathbb{R}^n$ and $r_B>0$. We can write
\[S_{L+I}(m)\leq S_{L+I}^{\textrm{loc}}(m)+S_{L+I}^\infty (m),
\]
where 
\[S_{L+I}^\infty (m)(x)=\left(\int_1^\infty \int_{B(x,t)}|t^2(L+I)e^{-t^2(L+I)}m(y)|^2\frac{dydt}{t^{n+1}}\right)^{1/2},\quad x\in \mathbb{R}^n.
\]

Suppose that $r_B\geq 1$. Then, $m$ is also a $(p(\cdot ),2,M,\varepsilon )_{L+I,\rm{loc}}$-molecule. Hence, by \eqref{objective} with $L+I$ instead of $L$,
\[\|S_{L+I}^{\textrm{loc}}(m)\|_{L^2(S_i(B))}\leq C2^{-i\eta}|2^iB|^{1/2}\|\chi _{2^iB}\|_{p(\cdot )}^{-1},
\]
for some $\eta >n(1/p^--1/p^+)$.

On the other hand, by denoting $W_t^{L+I}$ and $W_t^L$ the heat kernels defined by $L+I$ and $L$, respectively, we have that $W_t^{L+I}=e^{-t}W_t^L$, $t>0$, and
\[|W_t^{L+I}(x,y)|\leq Ce^{-t}\frac{e^{-c|x-y|^2/t}}{t^{n/2}},\quad x,y\in \mathbb{R}^n\mbox{ and }t>0.
\]
Also, for every $t>0$, the kernel $\mathbb{W}_t$ of the integral operator $(L+I)e^{-t^2(L+I)}$ satisfies that
\[|\mathbb{W}_t(x,y)|\leq Ce^{-t^2}\frac{e^{-c|x-y|^2/t^2}}{t^{n+2}},\quad x,y\in \mathbb{R}^n.
\]

Since $S_{L+I}$ is a bounded (sublinear) operator from $L^2(\mathbb{R}^n)$ into itself, we can proceed as in the proof of  \eqref{C1} for molecules of $(I)$-type to obtain that
\[\|S_{L+I}^\infty (m)\|_{L^2(S_i(B))}\leq C2^{-i\eta}|2^iB|^{1/2}\|\chi _{2^iB}\|_{p(\cdot )}^{-1},\quad i\in \mathbb{N}_0,
\]
for some $\eta>n(1/p^--1/p^+)$.

Suppose now that $r_B\in (0,1)$. There exists $b\in L^2(\mathbb{R}^n)$ such that $m=L^Mb$ and, for every $k\in \{0,...,M \}$,
\[\|L^k(b)\|_{L^2(S_i(B))}\leq 2^{-i\varepsilon }r_B^{2(M-k)}|2^iB|^{1/2}\|\chi _{2^iB}\|_{p(\cdot )}^{-1},\quad i\in \mathbb{N}_0.
\]

We define $\mathfrak{b}=L^{M-1}b$. Since $r_B\in (0,1)$, $\mathfrak{b}$ s a $(p(\cdot ),2,M-1,\varepsilon)_{L,{\rm loc}}$-molecule associated with the ball $B$ of $(II)$-type. It is clear that $m=L\mathfrak{b}$.

We can write
\begin{align*}
(S_{L+I}^{\rm loc}(m)(x))^2&=\int_0^1\int_{B(x,t)}|t^2(L+I)e^{-t^2(L+I)}(m)(y)|^2\frac{dydt}{t^{n+1}}\\
&\leq C\left(\int_0^1\int_{B(x,t)}|e^{-t^2}t^2Le^{-t^2L}(m)(y)|^2\frac{dydt}{t^{n+1}}\right.\\
&\quad \left. +\int_0^1\int_{B(x,t)}|e^{-t^2}t^2Le^{-t^2L}(\mathfrak{b})(y)|^2\frac{dydt}{t^{n+1}}\right)\\
&\leq C((S_L^{\rm loc}(m)(x))^2+(S_L^{\rm loc}(\mathfrak{b})(x))^2),\quad x\in \mathbb{R}^n.
\end{align*}

According to \eqref{objective} we conclude that
\begin{align*}
\|S_{L+I}^{\rm loc}(m)\|_{L^2(S_i(B))}&\leq C(\|S_L^{\rm loc}(m)\|_{L^2(S_i(B))}+\|S_L^{\rm loc}(\mathfrak{b})\|_{L^2(S_i(B))})\\
&\leq C2^{-i\eta}|2^iB|^{1/2}\|\chi _{2^iB}\|_{p(\cdot )}^{-1},\quad i\in \mathbb{N}_0,
\end{align*}
for some $\eta >n(1/p^--1/p^+)$.

Also, for every $i\in \mathbb{N}_0$, we have that  
\begin{align*}
\|S_{L+I}^\infty (m)\|_{L^2(S_i(B))}^2&=\int_{S_i(B)}\int_1^\infty \int_{B(x,t)}|t^2(L+I)e^{-t^2(L+I)}(L^Mb)(y)|^2\frac{dydt}{t^{n+1}}dx\\
&=\int_{S_i(B)}\int_1^\infty \int_{B(x,t)}|e^{-t^2}(t^2L)^{M+1}e^{-t^2L}(b)(y)|^2\frac{dydt}{t^{n+4M+1}}dx \\
&+\int_{S_i(B)}\int_1^\infty \int_{B(x,t)}|e^{-t^2}(t^2L)^{M}e^{-t^2L}(b)(y)|^2\frac{dydt}{t^{n+4M-3}}dx.
\end{align*}
We can proceed as in the proof of \eqref{C1} for molecules of $(II)$-type to obtain
\[\|S_{L+I}^\infty (m)\|_{L^2(S_i(B))}\leq C2^{-i\eta}|2^iB|^{1/2}\|\chi _{2^iB}\|_{p(\cdot) }^{-1},\quad i\in \mathbb{N}_0.
\]
Hence 
\[\|S_{L+I}(m)\|_{L^2(S_i(B))}\leq C2^{-i\eta}|2^iB|^{1/2}\|\chi _{2^iB}\|_{p(\cdot )}^{-1},\quad i\in \mathbb{N}_0,
\]
and we deduce that $f\in H_{L+I}^{p(\cdot )}(\mathbb{R}^n)$ and 
\[\|f\|_{H_{L+I}^{p(\cdot )}(\mathbb{R}^n)}\leq C
\left\|\left(\sum_{j\in \mathbb{N}}\left(\lambda_j\frac{\chi_{B_j}}{\|\chi_{B_j}\|_{p(\cdot)}}\right)^{\underline{p}}\right)^{1/\underline{p}}\right\|_{p(\cdot)}.
\]
As usual by using density arguments we get that $h_L^{p(\cdot )}(\mathbb{R}^n)\subset H_{L+I}^{p(\cdot )}(\mathbb{R}^n)$ and the inclusion is continuous.


\begin{thebibliography}{10}

\bibitem{ABR}
{\sc Almeida, V., Betancor, J., and Rodr\'{\i}guez-Mesa, L.}
\newblock Anisotropic {H}ardy-{L}orentz spaces with variable exponents.
\newblock To appear in Canad. J. Math.,
  http://dx.doi.org/10.4153/CJM-2016-053-6.

\bibitem{AM}
{\sc Alvarado, R., and Mitrea, M.}
\newblock {\em Hardy spaces on {A}hlfors-regular quasi metric spaces. A sharp
  theory}, vol.~2142 of {\em Lecture Notes in Mathematics}.
\newblock Springer, Cham, 2015.

\bibitem{AK}
{\sc Amenta, A., and Kemppainen, M.}
\newblock Non-uniformly local tent spaces.
\newblock {\em Publ. Mat. 59}, 1 (2015), 245--270.

\bibitem{ADM}
{\sc Auscher, P., Duong, X.~T., and McIntosh, A.}
\newblock Boundedness of {B}anach space valued singular integral and {H}ardy
  spaces.
\newblock Unplublished Manuscript (2005).

\bibitem{AMar}
{\sc Auscher, P., and Martell, J.~M.}
\newblock Weighted norm inequalities, off-diagonal estimates and elliptic
  operators. {II}. {O}ff-diagonal estimates on spaces of homogeneous type.
\newblock {\em J. Evol. Equ. 7}, 2 (2007), 265--316.

\bibitem{AMR}
{\sc Auscher, P., McIntosh, A., and Russ, E.}
\newblock Hardy spaces of differential forms on {R}iemannian manifolds.
\newblock {\em J. Geom. Anal. 18}, 1 (2008), 192--248.

\bibitem{BD}
{\sc Betancor, J.~J., and Dami\'an, W.~n.}
\newblock Anisotropic local {H}ardy spaces.
\newblock {\em J. Fourier Anal. Appl. 16}, 5 (2010), 658--675.

\bibitem{BDY}
{\sc Bui, H.-Q., Duong, X.~T., and Yan, L.}
\newblock Calder\'on reproducing formulas and new {B}esov spaces associated
  with operators.
\newblock {\em Adv. Math. 229\/} (2012), 2249--2502.

\bibitem{BCKYY}
{\sc Bui, T.~A., Cao, J., Ky, L.~D., Yang, D., and Yang, S.}
\newblock Weighted {H}ardy spaces associated with operators satisfying
  reinforced off-diagonal estimates.
\newblock {\em Taiwanese J. Math. 17}, 4 (2013), 1127--1166.

\bibitem{CFJY}
{\sc Cao, J., Fu, Z., Jiang, R., and Yang, D.}
\newblock Hardy spaces associated with a pair of commuting operators.
\newblock {\em Forum Math. 27}, 5 (2015), 2775--2824.

\bibitem{CMY}
{\sc Cao, J., Mayboroda, S., and Yang, D.}
\newblock Local {H}ardy spaces associated with inhomogeneous higher order
  elliptic operators.
\newblock {\em Anal. Appl. (Singap.) 15}, 2 (2017), 137--224.

\bibitem{CMM}
{\sc Carbonaro, A., McIntosh, A., and Morris, A.~J.}
\newblock Local {H}ardy spaces of differential forms on {R}iemannian manifolds.
\newblock {\em J. Geom. Anal. 23}, 1 (2013), 106--169.

\bibitem{CMS}
{\sc Coifman, R.~R., Meyer, Y., and Stein, E.~M.}
\newblock Some new function spaces and their applications to harmonic analysis.
\newblock {\em J. Funct. Anal. 62}, 2 (1985), 304--335.

\bibitem{CW}
{\sc Coifman, R.~R., and Weiss, G.}
\newblock Extensions of {H}ardy spaces and their use in analysis.
\newblock {\em Bull. Amer. Math. Soc. 83}, 4 (1977), 569--645.

\bibitem{CS}
{\sc Coulhon, T., and Sikora, A.}
\newblock {G}aussian heat kernel upper bounds via the
  {P}hragm\'en-–{L}indel\"of theorem.
\newblock {\em Proc. London Math. Soc. 96}, 3 (2008), 507--544.

\bibitem{CuFN}
{\sc Cruz-Uribe, D., Fiorenza, A., and Neugebauer, C.~J.}
\newblock The maximal function on variable {$L^p$} spaces.
\newblock {\em Ann. Acad. Sci. Fenn. Math. 28}, 1 (2003), 223--238.

\bibitem{CuW}
{\sc Cruz-Uribe, D., and Wang, L.-A.~D.}
\newblock Variable {H}ardy spaces.
\newblock {\em Indiana Univ. Math. J. 63}, 2 (2014), 447--493.

\bibitem{CuF}
{\sc Cruz-Uribe, D.~V., and Fiorenza, A.}
\newblock {\em Variable {L}ebesgue spaces}.
\newblock Applied and Numerical Harmonic Analysis. Birkh\"auser/Springer,
  Heidelberg, 2013.
\newblock Foundations and harmonic analysis.

\bibitem{CuMP}
{\sc Cruz-Uribe, D.~V., Martell, J., and P\'erez, C.}
\newblock {\em Weights, extrapolation and the theory of Rubio de Francia}.
\newblock vol. 215 of Operator Theory: Advances and Applications.
  Birkh\"auser/Springer, Basel, 2011.

\bibitem{DY}
{\sc Dafni, G., and Yue, H.}
\newblock Some characterizations of local bmo and {$h^1$} on metric measure
  spaces.
\newblock {\em Anal. Math. Phys. 2}, 3 (2012), 285--318.

\bibitem{D}
{\sc Diening, L.}
\newblock Maximal function on generalized {L}ebesgue spaces {$L^{p(\cdot)}$}.
\newblock {\em Math. Inequal. Appl. 7}, 2 (2004), 245--253.

\bibitem{DHHMS}
{\sc Diening, L., Harjulehto, P., H\"ast\"o, P., Mizuta, Y., and Shimomura, T.}
\newblock Maximal functions in variable exponent spaces: limiting cases of the
  exponent.
\newblock {\em Ann. Acad. Sci. Fenn. Math. 34}, 2 (2009), 503--522.

\bibitem{DHHR}
{\sc Diening, L., Harjulehto, P., H{\"a}st{\"o}, P., and
  R{{u}}{\v{z}}i{\v{c}}ka, M.}
\newblock {\em Lebesgue and {S}obolev spaces with variable exponents},
  vol.~2017 of {\em Lecture Notes in Mathematics}.
\newblock Springer, Heidelberg, 2011.

\bibitem{DHR}
{\sc Diening, L., H\"ast\"o, P., and Roudenko, S.}
\newblock Function spaces of variable smoothness and integrability.
\newblock {\em J. Funct. Anal. 256}, 6 (2009), 1731--1768.

\bibitem{DY2}
{\sc Duong, X.~T., and Yan, L.}
\newblock Duality of {H}ardy and {BMO} spaces associated with operators with
  heat kernel bounds.
\newblock {\em J. Amer. Math. Soc. 18}, 4 (2005), 943--973 (electronic).

\bibitem{DY1}
{\sc Duong, X.~T., and Yan, L.}
\newblock New function spaces of {BMO} type, the {J}ohn-{N}irenberg inequality,
  interpolation, and applications.
\newblock {\em Comm. Pure Appl. Math. 58}, 10 (2005), 1375--1420.

\bibitem{EKS}
{\sc Ephremidze, L., Kokilashvili, V., and Samko, S.}
\newblock Fractional, maximal and singular operators in variable exponent
  {L}orentz spaces.
\newblock {\em Fract. Calc. Appl. Anal. 11}, 4 (2008), 407--420.

\bibitem{FS}
{\sc Fefferman, C., and Stein, E.~M.}
\newblock {$H^{p}$} spaces of several variables.
\newblock {\em Acta Math. 129}, 3-4 (1972), 137--193.

\bibitem{Go}
{\sc Goldberg, D.}
\newblock A local version of real {H}ardy spaces.
\newblock {\em Duke Math. J. 46}, 1 (1979), 27--42.

\bibitem{GLiYa}
{\sc Gong, R., Li, J., and Yan, L.}
\newblock A local version of {H}ardy spaces associated with operators on metric
  spaces.
\newblock {\em Sci. China Math. 56}, 2 (2013), 315--330.

\bibitem{GoYa}
{\sc Gong, R., and Yan, L.}
\newblock Weighted {$L^p$} estimates for the area integral associated to
  self-adjoint operators.
\newblock {\em Manuscripta Math. 144}, 1-2 (2014), 25--49.

\bibitem{Ha}
{\sc H\"ast\"o, P.~A.}
\newblock Local-to-global results in variable exponent spaces.
\newblock {\em Math. Res. Lett. 16}, 2 (2009), 263--278.

\bibitem{HM}
{\sc Hofmann, S., and Mayboroda, S.}
\newblock Hardy and {BMO} spaces associated to divergence form elliptic
  operators.
\newblock {\em Math. Ann. 344}, 1 (2009), 37--116.

\bibitem{HuLi}
{\sc Huang, J., and Liu, Y.}
\newblock Molecular characterization of {H}ardy spaces associated with twisted
  convolution.
\newblock {\em J. Funct. Spaces\/} (2014), Art. ID 326940, 6.

\bibitem{HuWa}
{\sc Huang, J., and Wang, J.}
\newblock {${H}^p$}-boundedness of {W}eyl multipliers.
\newblock {\em J. Inequal. Appl.\/} (2014), 2014:422, 9.

\bibitem{HyYaYa}
{\sc Hyt\"onen, T., Yang, D., and Yang, D.}
\newblock The {H}ardy space {$H^1$} on non-homogeneous metric spaces.
\newblock {\em Math. Proc. Cambridge Philos. Soc. 153}, 1 (2012), 9--31.

\bibitem{JiYaYa1}
{\sc Jiang, R., Yang, D., and Yang, D.}
\newblock Maximal function characterizations of {H}ardy spaces associated with
  magnetic {S}chr\"odinger operators.
\newblock {\em Forum Math. 24}, 3 (2012), 471--494.

\bibitem{JiYaZh}
{\sc Jiang, R., Yang, D., and Zhou, Y.}
\newblock Localized {H}ardy spaces associated with operators.
\newblock {\em Appl. Anal. 88}, 9 (2009), 1409--1427.

\bibitem{KV}
{\sc Kempka, H., and Vyb{\'{\i}}ral, J.}
\newblock Lorentz spaces with variable exponents.
\newblock {\em Math. Nachr. 287}, 8-9 (2014), 938--954.

\bibitem{Kem}
{\sc Kemppainen, M.}
\newblock A note on local {H}ardy spaces.
\newblock {\em Forum Math. 29}, 4 (2017), 941--949.

\bibitem{KR}
{\sc Koch, H., and Ricci, F.}
\newblock Spectral projections for the twisted {L}aplacian.
\newblock {\em Studia Math. 180}, 2 (2007), 103--110.

\bibitem{Le}
{\sc Lerner, A.~K.}
\newblock Some remarks on the {H}ardy-{L}ittlewood maximal function on variable
  {$L^p$} spaces.
\newblock {\em Math. Z. 251}, 3 (2005), 509--521.

\bibitem{LYY}
{\sc Liang, Y., Yang, D., and Yang, S.}
\newblock Applications of orlicz spaces associated with operators satisfying
  poisson estimates.
\newblock {\em Sci. China Math. 54}, 11 (2011), 2395--2426.

\bibitem{LiYaYu}
{\sc Liu, J., Yang, D., and Yuang, W.}
\newblock Intrinsic anisotropic variable {H}ardy-{L}orentz spaces and their
  real interpolation.
\newblock To appear in J. Math Anal. Appl.

\bibitem{LiZhZh}
{\sc Liu, S., Zhao, K., and Zhou, S.}
\newblock Weighted {H}ardy spaces associated to self-adjoint operators and
  {$BMO_{L,w}$}.
\newblock {\em Taiwanese J. Math. 18}, 5 (2014), 1663--1678.

\bibitem{MPR}
{\sc Mauceri, G., Picardello, M.~A., and Ricci, F.}
\newblock A {H}ardy space associated with twisted convolution.
\newblock {\em Adv. in Math. 39}, 3 (1981), 270--288.

\bibitem{McI}
{\sc McIntosh, A.}
\newblock Operators which have an {$H_\infty$} functional calculus.
\newblock In {\em Miniconference on operator theory and partial differential
  equations ({N}orth {R}yde, 1986)}, vol.~14 of {\em Proc. Centre Math. Anal.
  Austral. Nat. Univ.} Austral. Nat. Univ., Canberra, 1986, pp.~210--231.

\bibitem{NS}
{\sc Nakai, E., and Sawano, Y.}
\newblock Hardy spaces with variable exponents and generalized {C}ampanato
  spaces.
\newblock {\em J. Funct. Anal. 262}, 9 (2012), 3665--3748.

\bibitem{Nek}
{\sc Nekvinda, A.~s.}
\newblock Hardy-{L}ittlewood maximal operator on {$L^{p(x)}(\mathbb R)$}.
\newblock {\em Math. Inequal. Appl. 7}, 2 (2004), 255--265.

\bibitem{Ou1}
{\sc Ouhabaz, E.~M.}
\newblock {\em Analysis of heat equations on domains}, vol.~31 of {\em London
  Mathematical Society Monographs Series}.
\newblock Princeton University Press, Princeton, NJ, 2005.

\bibitem{Ou2}
{\sc Ouhabaz, E.~M.}
\newblock Sharp {G}aussian bounds and {$L^p$}-growth of semigroups associated
  with elliptic and {S}chr\"odinger operators.
\newblock {\em Proc. Amer. Math. Soc. 134}, 12 (2006), 3567--3575.

\bibitem{Ru}
{\sc Russ, E.}
\newblock The atomic decomposition for tent spaces on spaces of homogeneous
  type.
\newblock In {\em C{MA}/{AMSI} {R}esearch {S}ymposium ``{A}symptotic
  {G}eometric {A}nalysis, {H}armonic {A}nalysis, and {R}elated {T}opics''},
  vol.~42 of {\em Proc. Centre Math. Appl. Austral. Nat. Univ.} Austral. Nat.
  Univ., Canberra, 2007, pp.~125--135.

\bibitem{Sa}
{\sc Sawano, Y.}
\newblock Atomic decompositions of {H}ardy spaces with variable exponents and
  its application to bounded linear operators.
\newblock {\em Integral Equations Operator Theory 77}, 1 (2013), 123--148.

\bibitem{SZi}
{\sc Stempak, K., and Zienkiewicz, J.}
\newblock Twisted convolution and {R}iesz means.
\newblock {\em J. Anal. Math. 76\/} (1998), 93--107.

\bibitem{Stri}
{\sc Strichartz, R.~S.}
\newblock The {H}ardy space {$H^{1}$} on manifolds and submanifolds.
\newblock {\em Canad. J. Math. 24\/} (1972), 915--925.

\bibitem{T}
{\sc Tang, L.}
\newblock Weighted local {H}ardy spaces and their applications.
\newblock {\em Illinois J. Math. 56}, 2 (2012), 453--495.

\bibitem{Th}
{\sc Thangavelu, S.}
\newblock Some remarks on {B}ochner-{R}iesz means.
\newblock {\em Colloq. Math. 83}, 2 (2000), 217--230.

\bibitem{Tol}
{\sc Tolsa, X.}
\newblock The space {$H^1$} for nondoubling measures in terms of a grand
  maximal operator.
\newblock {\em Trans. Amer. Math. Soc. 355}, 1 (2003), 315--348.

\bibitem{Ya}
{\sc Yang, D.}
\newblock Local {H}ardy and {BMO} spaces on non-homogeneous spaces.
\newblock {\em J. Aust. Math. Soc. 79}, 2 (2005), 149--182.

\bibitem{YaYa}
{\sc Yang, D., and Yang, S.}
\newblock Local {H}ardy spaces of {M}usielak-{O}rlicz type and their
  applications.
\newblock {\em Sci. China Math. 55}, 8 (2012), 1677--1720.

\bibitem{YaZhZh}
{\sc Yang, D., Zhang, J., and Zhuo, C.}
\newblock Variable {H}ardy spaces associated with operators satisfying
  reinforced off-diagonal estimates.
\newblock Preprint 2016
  \href{http://arxiv.org/abs/1601.06358}{\texttt{(arXiv:1601.06358v1)}}.

\bibitem{YaZh1}
{\sc Yang, D., and Zhuo, C.}
\newblock Molecular characterizations and dualities of variable exponent
  {H}ardy spaces associated with operators.
\newblock {\em Ann. Acad. Sci. Fenn. Math. 41}, 1 (2016), 357--398.

\bibitem{YaZhN}
{\sc Yang, D., Zhuo, C., and Nakai, E.}
\newblock Characterizations of variable exponent {H}ardy spaces via {R}iesz
  transforms.
\newblock {\em Rev. Mat. Complut. 29}, 2 (2016), 245--270.

\bibitem{ZhSaYa}
{\sc Zhuo, C., Sawano, Y., and Yang, D.}
\newblock Hardy spaces with variable exponents on {RD}-spaces and applications.
\newblock {\em Dissertationes Math. (Rozprawy Mat.) 520\/} (2016), 74.

\bibitem{ZhYa}
{\sc Zhuo, C., and Yang, D.}
\newblock Maximal function characterizations of variable {H}ardy spaces
  associated with non-negative self-adjoint operators satisfying {G}aussian
  estimates.
\newblock {\em Nonlinear Anal. 141\/} (2016), 16--42.

\bibitem{ZhYaLi}
{\sc Zhuo, C., Yang, D., and Liang, Y.}
\newblock Intrinsic square function characterizations of {H}ardy spaces with
  variable exponents.
\newblock {\em Bull. Malays. Math. Sci. Soc. 39}, 4 (2016), 1541--1577.

\end{thebibliography}

\end{document}